\documentclass[11pt]{article}

\usepackage{fullpage}
\usepackage{amssymb,amsmath,amsthm,mathtools}
\usepackage{algorithm}
\usepackage{algorithmic}
\usepackage{graphicx}
\usepackage{xcolor}
\usepackage{colortbl}
\usepackage{enumitem}
\usepackage{array}
\usepackage{longtable}
\usepackage{threeparttable}
\usepackage{pifont}
\usepackage{xspace}
\usepackage[round]{natbib}

\usepackage{hyperref}
\usepackage{bookmark}
\definecolor{niceblue}{rgb}{0.0,0.19,0.56}
\hypersetup{
  colorlinks = true,
  linkcolor  = blue,
  citecolor  = niceblue,
  urlcolor   = blue,
  pdftitle   = {Weak-L-alpha Noise, Strong Guarantees: Gradient Clipping under Affine Heavy Tails}
}

\usepackage{mdframed}
\usepackage{thmtools}
\makeatletter
\@ifundefined{newcounteralias}{}{%
  \renewcommand\thmt@autorefsetup{%
    \@xa\def\csname\thmt@envname autorefname\@xa\endcsname\@xa{\thmt@thmname}%
  }%
}
\makeatother
\definecolor{shadecolor}{gray}{0.9}
\declaretheoremstyle[
headfont=\normalfont\bfseries,
notefont=\mdseries, notebraces={(}{)},
bodyfont=\normalfont,
postheadspace=0.5em,
spaceabove=1pt,
mdframed={
 skipabove=8pt,
 skipbelow=8pt,
 hidealllines=true,
 backgroundcolor={shadecolor},
 innerleftmargin=4pt,
 innerrightmargin=4pt}
]{shaded}
\declaretheorem[style=shaded,within=section]{theorem}
\declaretheorem[style=shaded,sibling=theorem]{lemma}
\declaretheorem[style=shaded,sibling=theorem]{proposition}
\declaretheorem[style=shaded,sibling=theorem]{corollary}
\declaretheorem[style=shaded,sibling=theorem]{definition}
\declaretheorem[style=shaded,sibling=theorem]{remark}
\declaretheorem[style=shaded,within=section]{assumption}
\usepackage{needspace,etoolbox}
\makeatletter
\newcommand{\boxneedspace}{\if@nobreak\else\Needspace{8\baselineskip}\fi}
\@for\thm@env:=theorem,lemma,proposition,corollary,definition,remark,assumption\do{\expandafter\BeforeBeginEnvironment\expandafter{\thm@env}{\boxneedspace}}
\makeatother

\newcommand{\R}{\mathbb{R}}
\newcommand{\E}{\mathbb{E}}
\newcommand{\PP}{\mathbb{P}}
\newcommand{\D}{\mathcal{D}}
\newcommand{\F}{\mathcal{F}}
\newcommand{\G}{\mathcal{G}}

\newcommand*{\dif}{\mathop{}\!\mathrm{d}}

\newcommand{\abr}[1]{\left\langle#1\right\rangle}
\newcommand{\ind}{\mathbf{1}}
\newcommand{\tg}{\widetilde{g}}
\newcommand{\tX}{\widetilde{X}}
\def\clip{\texttt{clip}}
\DeclareMathOperator{\TV}{TV}
\DeclareMathOperator{\Var}{Var}
\newcommand{\wnorm}[1]{\left\lVert #1\right\rVert_{L^{\alpha,\infty}}}
\newcommand{\Mnorm}[1]{\left\lVert #1\right\rVert_{M_\alpha}}

\definecolor{PineGreen}{RGB}{0,110,51}
\newcommand{\algname}[1]{{\color{PineGreen}\sf #1}\xspace}
\newcommand{\clipsgd}{\algname{Clip-SGD}}
\newcommand{\clipsstm}{\algname{Clip-SSTM}}
\newcommand{\circledOne}{\text{\ding{172}}}
\newcommand{\circledTwo}{\text{\ding{173}}}
\newcommand{\circledThree}{\text{\ding{174}}}
\newcommand{\circledFour}{\text{\ding{175}}}
\newcommand{\circledFive}{\text{\ding{176}}}

\newcommand{\stk}[1]{{\renewcommand{\arraystretch}{1.05}\begin{tabular}[c]{@{}l@{}}#1\end{tabular}}}

\title{\bf Weak-$L^\alpha$ Noise, Strong Guarantees:\\ Gradient Clipping under Affine Heavy Tails}
\author{\begin{tabular}{c}
     {\bf Alyona Nekrasova$^{1}$ \quad Makar Sidorov$^{2}$ \quad Savelii Chezhegov$^{3}$}\\
     {\bf Alexander Gasnikov$^{4}$ \quad Margarita Preobrazhenskaia$^{2}$}\\
     {\bf Igor Preobrazhensky$^{2}$ \quad Eduard Gorbunov$^{5}$}\\[0.7em]
     {\small $^{1}$Independent researcher \qquad $^{2}$Yaroslavl State University}\\
     {\small $^{3}$Moscow Independent Research Institute of Artificial Intelligence}\\
     {\small $^{4}$Innopolis University \qquad $^{5}$Mohamed bin Zayed University of Artificial Intelligence}
\end{tabular}}
\date{\today}

\hypersetup{
  pdftitle={Weak-L-alpha Noise, Strong Guarantees: Gradient Clipping under Affine Heavy Tails},
  pdfauthor={Alyona Nekrasova, Makar Sidorov, Savelii Chezhegov, Alexander Gasnikov, Margarita Preobrazhenskaia, Igor Preobrazhensky, Eduard Gorbunov}
}

\begin{document}
\widowpenalty=10000 \clubpenalty=10000 \displaywidowpenalty=10000
\setlength{\emergencystretch}{1.5em}
\hyphenation{Mar-cin-kie-wicz}

\maketitle
\addtocontents{toc}{\protect\setcounter{tocdepth}{-1}}

\begin{abstract}

Gradient clipping is a key tool for obtaining high-probability convergence guarantees for stochastic first-order methods under heavy-tailed noise. Existing analyses typically assume a bounded central moment of order $\alpha\in(1,2]$, which excludes distributions with tails of order $u^{-\alpha}$ at the critical exponent. In this paper, we consider smooth convex optimization under affine heavy-tailed noise, whose scale may grow with the gradient norm. We derive high-probability convergence guarantees for clipped SGD (\clipsgd) and its accelerated variant \clipsstm under a weak-$L^\alpha$ tail condition, which is strictly weaker than the corresponding moment bound. For $1<\alpha<2$, the clipping bounds have the same powers of the clipping level as in the moment case, and both methods attain the additive stochastic rate $K^{-(\alpha-1)/\alpha}$. The gradient-dependent noise contributes a faster-decaying term of order $K^{-2(\alpha-1)/\alpha}$, up to logarithmic factors. At $\alpha=2$, we characterize the additional logarithmic factor on the Lorentz scale and establish its necessity for the additive error at fixed confidence. We also analyze \clipsgd with prescribed clipping levels and derive finite-horizon guarantees for reaching a neighborhood determined by the clipping bias. For additive noise, matching lower bounds show that this neighborhood is unavoidable, up to constants and the natural error cap. Numerical experiments with Pareto noise illustrate the theoretical results.
\end{abstract}

\section{Introduction}\label{sec:intro}

Gradient clipping rescales a stochastic gradient whenever its norm exceeds a prescribed level $\lambda$. It is a standard tool for stabilizing the training of machine learning models. It was introduced to handle exploding gradients in recurrent neural networks \citep{pascanu2013difficulty} and is now used in large-scale training pipelines, including GPT-3 \citep{brown2020language} and LLaMA \citep{touvron2023llama}. It is also a basic component of differentially private SGD \citep{abadi2016deep}. One explanation for its usefulness is that stochastic gradient noise can be heavy-tailed, as observed in deep learning and language-model training \citep{simsekli2019tail,zhang2020adaptive}. Moreover, SGD iterates can have heavy-tailed stationary distributions even for light-tailed data \citep{gurbuzbalaban2021heavy}. Under a bounded central moment of order $\alpha\in(1,2]$, clipped SGD (\clipsgd) converges with high probability and logarithmic dependence on the inverse failure probability, whereas SGD without clipping may fail when the variance is infinite \citep{zhang2020adaptive,gorbunov2020stochastic,sadiev2023high}. Section~\ref{sec:related} discusses extensions of these results.

\begin{table}[t]
\caption{Summary of known and new high-probability complexity results for \clipsgd and \clipsstm with the clipping level tuned to the horizon. Complexity is the total number of oracle calls needed for a method to guarantee $f-f^*\le\varepsilon$ with probability at least $1-\beta$; the constants depending only on $\alpha$ and the logarithmic factors are omitted for simplicity. Notation: $\zeta:=g(x,\xi)-\nabla f(x)$ is the noise, $\alpha\in(1,2]$ is its tail index and $\sigma$ its level; $R\ge\|x^0-x^*\|$ is any upper bound on the initial distance; $\beta\in(0,1)$ is the failure probability. Under the affine conditions, $\E\|\zeta\|^\alpha\le\sigma_0^\alpha+\sigma_1^\alpha\|\nabla f(x)\|^\alpha$ (Assumption~\ref{as:moment}) or $\PP\{\|\zeta\|>u\}\le(\sigma_0^\alpha+\sigma_1^\alpha\|\nabla f(x)\|^\alpha)u^{-\alpha}$ for all $u>0$ (Assumption~\ref{as:weak}). The error $f-f^*$ is measured at the averaged iterate $\bar x^K$ for \clipsgd and at the last iterate $y^K$ for \clipsstm. The results of this paper are highlighted in blue.
}\label{tab:complexity}
\begin{center}
\footnotesize
\renewcommand{\arraystretch}{1.7}
\setlength{\tabcolsep}{3.5pt}
\resizebox{\textwidth}{!}{%
\begin{tabular}{llll}
\textbf{Method} & \textbf{Result} & \textbf{Noise condition} & \textbf{Oracle calls}\\
\hline
\clipsgd & \citet{sadiev2023high} & $\E\|\zeta\|^\alpha\le\sigma^\alpha$, $\alpha\in(1,2]$ & $\max\left\{\frac{LR^2}{\varepsilon},\left(\frac{\sigma R}{\varepsilon}\right)^{\frac{\alpha}{\alpha-1}}\right\}$\\
\rowcolor{blue!10} \clipsgd & Theorem~\ref{thm:sgd} & Assumption~\ref{as:moment} or~\ref{as:weak} & $\max\left\{\frac{LR^2}{\varepsilon},\left(\frac{\sigma_0R}{\varepsilon}\right)^{\frac{\alpha}{\alpha-1}},\left(\frac{\sigma_1^2LR^2}{\varepsilon}\right)^{\frac{\alpha}{2(\alpha-1)}}\right\}$\\
\clipsstm & \citet{gorbunov2020stochastic} & $\E\|\zeta\|^2\le\sigma^2$ & $\max\left\{\sqrt{\frac{LR^2}{\varepsilon}},\frac{\sigma^2R^2}{\varepsilon^2}\right\}$\\
\clipsstm & \citet{sadiev2023high} & $\E\|\zeta\|^\alpha\le\sigma^\alpha$, $\alpha\in(1,2]$ & $\max\left\{\sqrt{\frac{LR^2}{\varepsilon}},\left(\frac{\sigma R}{\varepsilon}\right)^{\frac{\alpha}{\alpha-1}}\right\}$\\
\rowcolor{blue!10} \clipsstm & uniform scale (Cor.~\ref{cor:uniform}) & Assumption~\ref{as:moment} or~\ref{as:weak} & $\max\left\{\sqrt{\frac{LR^2}{\varepsilon}},\left(\frac{\sigma_0R}{\varepsilon}\right)^{\frac{\alpha}{\alpha-1}},\left(\frac{\sigma_1LR^2}{\varepsilon}\right)^{\frac{\alpha}{\alpha-1}}\right\}$\\
\rowcolor{blue!10} \clipsstm & Theorem~\ref{thm:sstm-main} & Assumption~\ref{as:moment} or~\ref{as:weak} & $\max\left\{\sqrt{\frac{LR^2}{\varepsilon}},\left(\frac{\sigma_0R}{\varepsilon}\right)^{\frac{\alpha}{\alpha-1}},\left(\frac{\sigma_1^2LR^2}{\varepsilon}\right)^{\frac{\alpha}{2(\alpha-1)}}\right\}$\\
\end{tabular}%
}
\end{center}
\end{table}

Two limitations of much of the existing high-probability theory motivate this work. First, the moment assumption $\E\|g(x,\xi)-\nabla f(x)\|^\alpha\le\sigma^\alpha$ is stronger than the tail control used in our analysis. Pareto distributions and non-Gaussian $\alpha$-stable distributions with $\alpha\in(1,2)$ have tails of order $u^{-\alpha}$ but an infinite $\alpha$-th moment; the latter have been proposed as models of stochastic gradient noise in deep learning \citep{simsekli2019tail}. A moment theorem therefore applies only at a smaller exponent, with a slower polynomial rate when that exponent is fixed. The weak-$L^\alpha$ (Marcinkiewicz) condition $\PP\{\|g(x,\xi)-\nabla f(x)\|>u\}\le\sigma^\alpha u^{-\alpha}$ controls the tail directly.

Second, many high-probability analyses tune the clipping level to the noise parameters and the horizon and assume a uniform noise scale on the relevant region. Practical implementations often use a prescribed constant level, and the noise scale may vary with the gradient. Although prior work addresses some of these features separately (Section~\ref{sec:related}), their combination under a local weak-tail condition motivates the following question:
\begin{center}
    \emph{What do clipped methods guarantee under a pure tail condition on the noise, with a noise scale that grows with the gradient norm, and with a clipping level chosen by the user?}
\end{center}
In this paper, we address this question for \clipsgd and its accelerated variant \clipsstm.

\paragraph{Our contributions.} We analyze \clipsgd and \clipsstm under affine moment and weak-$L^\alpha$ noise conditions, which allow the noise scale to depend on the gradient norm. Table~\ref{tab:complexity} compares the guarantees for horizon-dependent tuning with prior results. Table~\ref{tab:prescribed} summarizes the guarantees and lower bounds for prescribed clipping levels. Our contributions are as follows.
\begin{itemize}[leftmargin=*]
    \item \textbf{Clipping under tail conditions.} We bound the clipping bias and the second moment of the clipped error under a weak-$L^\alpha$ tail condition. The bounds have the same powers of the clipping level as in the moment case, with an unavoidable logarithmic factor at $\alpha=2$ (Lemma~\ref{lem:clip}).
    \item \textbf{New high-probability results for \clipsgd and \clipsstm.} With horizon-dependent tuning, both methods attain high-probability rates with explicit constants (Theorems~\ref{thm:sgd} and~\ref{thm:sstm-main}). In the nontrivial stochastic regime, the additive term is optimal in $K$, $\sigma_0$, and $R$; the corresponding confidence dependence is optimal under the moment condition and under the weak-tail condition for $\alpha<2$ (Theorem~\ref{thm:maximal} and Propositions~\ref{prop:lower-rate} and~\ref{prop:confidence-lower}). At the weak-$L^2$ endpoint, the additive term contains the necessary factor $\sqrt{\ln(eK)}$ at fixed confidence (Corollary~\ref{cor:lorentz-upper}). The gradient-dependent noise contributes a faster-decaying term. For \clipsstm, this follows from a weighted bound on the true gradients; we do not establish optimality of the multiplicative term.
    \item \textbf{Prescribed clipping levels.} For sufficiently large prescribed clipping levels, we bound the error of the averaged iterate by a clipping-bias neighborhood plus terms that vanish with the horizon (Theorem~\ref{thm:fixed-main}). For arbitrary positive clipping levels, we obtain a best-iterate guarantee under additive or sufficiently small multiplicative noise; the weak-tail version requires $\alpha<2$ (Theorem~\ref{thm:any-main}). Additive-noise lower bounds show that the neighborhood is unavoidable up to $\alpha$-dependent constants and truncation at $LR^2$ (Theorem~\ref{thm:lower-main}). For additive noise, a clipping level proportional to $(K+1)^{1/\alpha}$, held constant within each run, recovers the rate $K^{-(\alpha-1)/\alpha}$ up to logarithmic factors.
    \item \textbf{Numerical experiments.} In Section~\ref{sec:experiments} and Appendix~\ref{app:experiments}, we illustrate the derived results in numerical experiments with Pareto noise.
\end{itemize}

\section{Related Work}\label{sec:related}

We discuss the closely related works below and defer a more detailed overview to Appendix~\ref{app:related}.

\paragraph{High-probability bounds under heavy-tailed noise.} High-probability guarantees with logarithmic dependence on the inverse failure probability were first obtained for clipped and truncated stochastic methods under bounded variance by \citet{nazin2019algorithms} and \citet{gorbunov2020stochastic}. \citet{sadiev2023high} extended the analysis to bounded central moments of order $\alpha\in(1,2]$. For smooth convex problems, they derive the rate $\widetilde O(LR^2/K+\sigma R K^{-(\alpha-1)/\alpha})$ for \clipsgd and its accelerated counterpart for \clipsstm. These are the closest constant-scale results: Theorems~\ref{thm:sgd} and~\ref{thm:sstm-main} recover them when $\sigma_1=0$ and the moment condition holds (Table~\ref{tab:complexity}).

The weak-$L^p$ condition itself is used by \citet{elbakkali2026stochastic} for zeroth-order non-convex optimization under global sample-wise smoothness and a constant noise scale. Their methods clip scalar directional derivatives and use mini-batches or a momentum scheme with a warm-start batch. We consider first-order methods for convex problems with one oracle call per iteration, under a local tail condition whose scale may grow with the gradient norm. Our assumption requires neither symmetry nor a density. Appendix~\ref{app:related} discusses the full range of bounded-moment results, structural noise assumptions, and methods without clipping, including the accelerated convex method of \citet{he2025accelerated}.

\paragraph{Affine noise, fixed clipping levels, and lower bounds.} Affine variance is classical in the analysis of SGD \citep{bottou2018optimization,gower2019sgd,khaled2023better}. Most high-probability analyses under heavy-tailed noise use a constant scale on the relevant region; state-dependent finite-variance bounds appear in \citet{tsai2022heavy,das2024near}, and affine moment bounds are considered in the concurrent work of \citet{lobanov2026accelerated}.

For a fixed clipping level, \citet{koloskova2023revisiting} characterize the clipping bias and derive tight in-expectation guarantees under bounded variance. \citet{vatankhah2026differentially} obtain high-probability bounds for the best iterate with an arbitrary clipping level under a bounded moment assumption. Table~\ref{tab:prescribed} compares their regime $\lambda>4LR$. We extend prescribed-level guarantees to affine weak-$L^\alpha$ noise and derive matching additive-noise lower bounds, distinguishing averaged and best iterates.

Stochastic convex lower bounds go back to \citet{nemirovski1983problem}. \citet{vural2022mirror} establish the optimal heavy-tailed rate $K^{-(\alpha-1)/\alpha}$, and \citet{liu2026clipped} derive high-probability lower bounds for non-smooth convex problems. Our lower bounds concern smooth unconstrained problems and capture the confidence dependence and the additional logarithmic factors at the weak-$L^2$ endpoint.

\paragraph{Concurrent work.} \citet{lobanov2026accelerated} analyze a restarted accelerated method with clipping and projection under $(H_0,H_1)$-smoothness \citep{alimisis2025why} and a moment bound combining additive, gradient-dependent, and function-gap-dependent noise. Their method uses one sample per iteration. For $L$-smooth objectives and no function-gap-dependent term, their convex complexity has the same polynomial dependence as Theorem~\ref{thm:sstm-main}, including the dependence on $\sigma_1$, whose optimality remains open in both works.

Their analysis covers generalized smoothness, strong convexity, and noise decaying with the function gap. We analyze \clipsgd and \clipsstm without restarts or projections under a tail condition on a ball, including the weak-$L^2$ endpoint with infinite variance, and study prescribed clipping levels. Appendix~\ref{app:related} gives the precise assumptions, complexity comparison, and relation between the lower bounds.

\section{Problem Setup}\label{sec:setup}

We consider the following stochastic optimization problem and assumptions.

\paragraph{Notation.} We use the standard notation: $\abr{\cdot,\cdot}$ denotes the standard inner product in $\R^d$ and $\|\cdot\|$ is the induced Euclidean norm, $B_r(x):=\{y\in\R^d:\|y-x\|\le r\}$ is the ball with center at $x$ and radius $r$, $(a)_+:=\max\{a,0\}$, and $\ind_A$ is the indicator of an event $A$. For a $\sigma$-algebra $\G$, $\E[\cdot\mid\G]$ denotes the conditional expectation w.r.t.\ $\G$.

\paragraph{Problem.} We consider the stochastic optimization problem
\begin{equation}\label{eq:problem}
    \min_{x\in\R^d} f(x),
\end{equation}
where $f:\R^d\to\R$ is Borel measurable and accessible only through the stochastic oracle specified below. We assume that a minimizer $x^*$ exists and write $f^*:=f(x^*)$. Following \citet{sadiev2023high}, we impose the smoothness, convexity, and noise assumptions only on a ball around a minimizer. Given $R\ge\|x^0-x^*\|>0$, we use $Q=B_{2R}(x^*)$ for \clipsgd and $Q=B_{3R}(x^*)$ for \clipsstm. The analysis shows that the relevant iterates remain in these balls with high probability. No smoothness or convexity is required outside $Q$.

\paragraph{Assumptions on $f$.} We start with the standard assumptions on the objective function.

\begin{assumption}[Local $L$-smoothness]\label{as:smooth}
We assume that $f$ is $L$-smooth on $Q$, i.e., $\|\nabla f(x)-\nabla f(y)\|\le L\|x-y\|$ for all $x,y\in Q$.
\end{assumption}

\begin{assumption}[Gradient bound]\label{as:grad-bound}
We assume that $\|\nabla f(x)\|^2\le 2L\left(f(x)-f^*\right)$ for all $x\in Q$.
\end{assumption}

\begin{assumption}[Local convexity]\label{as:convex}
We assume that $f$ is convex on $Q$, i.e., $f(y)\ge f(x)+\abr{\nabla f(x),y-x}$ for all $x,y\in Q$.
\end{assumption}

If $f$ is convex and $L$-smooth on $\R^d$, then Assumptions~\ref{as:smooth}--\ref{as:convex} hold for any $R$. Here these properties are required only on $Q$, which contains the relevant iterates with high probability. Assumption~\ref{as:grad-bound} follows from Assumptions~\ref{as:smooth} and~\ref{as:convex}, since $Q$ is a ball centered at $x^*$ (Lemma~\ref{lem:grad-bound}). We state it separately to simplify the comparison with \citet{sadiev2023high}.

\paragraph{Stochastic oracle.} Let $\G_0:=\sigma(x^0)$ and $\G_k:=\sigma(x^0,\xi^0,\dots,\xi^{k-1})$ for $k\ge1$. The oracle is a jointly measurable, finite vector field $g(x,\xi)\in\R^d$ satisfying $\E_\xi[g(x,\xi)]=\nabla f(x)$ for all $x\in Q$. It need not be the gradient of a random function; expected-loss minimization with $g=\nabla f_\xi$ is a special case. At iteration $k$, the method samples $\xi^k\sim\D$ independently of $\G_k$ and queries the oracle at a $\G_k$-measurable point. The noise assumptions hold only on $Q$; outside $Q$, the oracle is an arbitrary finite jointly measurable extension. Appendix~\ref{app:aux} explains why these pointwise assumptions apply to adaptive queries and why the extension does not enter the bounds. We consider two affine noise models, whose scale may grow with the gradient norm.

\begin{assumption}[Affine moment condition]\label{as:moment}
We assume that there exist $\sigma_0,\sigma_1\ge0$ and $\alpha\in(1,2]$ such that for all $x\in Q$ we have $\E_\xi[g(x,\xi)]=\nabla f(x)$ and
\begin{equation}\label{eq:moment-affine}
    \E_\xi\|g(x,\xi)-\nabla f(x)\|^\alpha\le\sigma_0^\alpha+\sigma_1^\alpha\|\nabla f(x)\|^\alpha .
\end{equation}
\end{assumption}

\begin{assumption}[Affine weak-$L^\alpha$ (Marcinkiewicz) condition]\label{as:weak}
We assume that there exist $\sigma_0,\sigma_1\ge0$ and $\alpha\in(1,2]$ such that for all $x\in Q$ we have $\E_\xi[g(x,\xi)]=\nabla f(x)$ and, for all $u>0$,
\begin{equation}\label{eq:weak-tail}
    \PP\left\{\|g(x,\xi)-\nabla f(x)\|>u\right\}\le\frac{\sigma_0^\alpha+\sigma_1^\alpha\|\nabla f(x)\|^\alpha}{u^\alpha}.
\end{equation}
\end{assumption}

Condition~\eqref{eq:weak-tail} restricts only the tail of the noise. Equivalently, $\wnorm{\|g(x,\xi)-\nabla f(x)\|}^{\alpha}\le\sigma_0^\alpha+\sigma_1^\alpha\|\nabla f(x)\|^\alpha$, where $\wnorm{\cdot}$ denotes the weak-$L^\alpha$ quasi-norm. For $\alpha>1$, this quasi-norm is equivalent to the norm of the Marcinkiewicz space $M_\alpha$, which is a classical rearrangement-invariant Banach space \citep{krein2002interpolation,bennett1988interpolation}; see Appendix~\ref{app:spaces} for the details. By Markov's inequality, Assumption~\ref{as:moment} implies Assumption~\ref{as:weak} with the same $\sigma_0,\sigma_1,\alpha$, but the converse implication does not hold at the same exponent: weak-$L^\alpha$ noise may have infinite $\alpha$-th moment and, for $\alpha=2$, infinite variance (see the example in Section~\ref{sec:pareto}). When $\sigma_1=0$, Assumption~\ref{as:moment} recovers the standard heavy-tailed noise assumption \citep{zhang2020adaptive,sadiev2023high}.

\paragraph{Example: Pareto noise at the critical index.}\label{sec:pareto}
Weak-$L^\alpha$ noise can have an infinite $\alpha$-th moment. For example, non-Gaussian symmetric $\alpha$-stable distributions with $\alpha\in(1,2)$, proposed as models of stochastic gradient noise in deep learning \citep{simsekli2019tail}, have tails $\PP\{|Y|>u\}\sim c_\alpha u^{-\alpha}$ as $u\to\infty$ \citep{samoradnitsky2017stable}. For a concrete example, let $\sigma>0$, $\alpha\in(1,2]$, and $\xi=\sigma\varepsilon\mathsf R$, where $\varepsilon$ is an independent Rademacher sign and $\PP\{\mathsf R>u\}=\min\{1,u^{-\alpha}\}$. Then $\E\xi=0$, $\wnorm{\xi}=\sigma$, and $\E|\xi|^\alpha=\infty$, while $\E|\xi|^q=\sigma^q\alpha/(\alpha-q)$ for $0<q<\alpha$ (Proposition~\ref{prop:pareto}). Thus, the moment condition applies only at a smaller exponent, whereas the weak-tail condition holds at the critical exponent with the exact scale. The Pareto tail bound is attained for $u\ge\sigma$, and a $d$-dimensional version is $\xi=\sigma\mathsf R U$, where $U$ is independent and uniform on the unit sphere. Our experiments use symmetric or centered asymmetric Pareto noise in scalar or one-directional form, with an affine scale $S(x)$.

\paragraph{Considered methods.} For $\lambda>0$, the clipping operator is defined as follows: $\clip(v,\lambda):=\min\{1,\lambda/\|v\|\}\,v$ for $v\neq0$ and $\clip(0,\lambda):=0$. We study the convergence of clipped SGD (\clipsgd, Algorithm~\ref{alg:clip-sgd}) in terms of its averaged iterate $\bar x^K:=\frac{1}{K+1}\sum_{k=0}^K x^k$, and of its accelerated variant \clipsstm, which we introduce in Section~\ref{sec:sstm}.

\begin{algorithm}[t]
\caption{\clipsgd}\label{alg:clip-sgd}
\begin{algorithmic}[1]
\REQUIRE starting point $x^0$, stepsize $\gamma>0$, clipping levels $\{\lambda_k\}_{k\ge0}$, number of iterations $K+1$
\FOR{$k=0,1,\dots,K$}
\STATE Sample $\xi^k\sim\D$ independently of $\G_k$; set $g^k=g(x^k,\xi^k)$
\STATE $\tg^k=\clip(g^k,\lambda_k)$
\STATE $x^{k+1}=x^k-\gamma\tg^k$
\ENDFOR
\ENSURE $\bar x^K=\frac{1}{K+1}\sum_{k=0}^K x^k$
\end{algorithmic}
\end{algorithm}

\section{Clipping under Affine Heavy-Tailed Noise}\label{sec:clipping}

The key technical ingredient of the high-probability analysis of clipped methods is a lemma bounding the bias and the variance of the clipped stochastic gradient. In this section, we generalize Lemma~5.1 from \citet{sadiev2023high} to the setting of this paper, i.e., to the noise satisfying either the affine weak-$L^\alpha$ condition or the affine $\alpha$-th moment condition. For brevity, we denote $A(x):=\sigma_0^\alpha+\sigma_1^\alpha\|x\|^\alpha$.

\begin{lemma}[Clipping under affine heavy-tailed noise]\label{lem:clip}
Let $\alpha\in(1,2]$, $\lambda>0$, and let $X$ be a random vector in $\R^d$ such that $\E X=x$, $\|x\|\le\lambda/2$, and either (a) $\PP\{\|X-x\|>u\}\le A(x)u^{-\alpha}$ for all $u>0$, or (b) $\E\|X-x\|^\alpha\le A(x)$. Let $\tX:=\clip(X,\lambda)$. Then $\|\tX-\E\tX\|\le2\lambda$ almost surely, $\E\|\tX-\E\tX\|^2\le\E\|\tX-x\|^2$, and
\begin{equation}\label{eq:clip-unified}
    \|\E\tX-x\|\le\frac{b\,A(x)}{\lambda^{\alpha-1}},\qquad \E\|\tX-x\|^2\le v\,\lambda^{2-\alpha}A(x),
\end{equation}
where $(b,v)=(b_\alpha,v_\alpha((3\lambda/2)^\alpha/A(x)))$ in case~(a) (the second bound reads $0$ if $A(x)=0$) and $(b,v)=(b_\alpha^{\rm mom},v_\alpha^{\rm mom})$ in case~(b), with
\begin{gather}
    v_\alpha(s):=\left(\tfrac32\right)^{2-\alpha}\min\left\{\tfrac{2}{2-\alpha},\,1+\tfrac2\alpha\ln(1+s)\right\},\label{eq:BC-alpha}\\
    b_\alpha:=\frac{2^{\alpha-1}}{\alpha-1},\qquad b_\alpha^{\rm mom}:=\frac{2^{\alpha-1}(\alpha-1)^{\alpha-1}}{\alpha^\alpha},\label{eq:moment-clipping-constants}
\end{gather}
and $v_\alpha^{\rm mom}:=(\frac32)^{2-\alpha}$, for $s\in[0,\infty]$, where $\frac{2}{2-\alpha}:=+\infty$ for $\alpha=2$; here $\frac12\le b_\alpha^{\rm mom}\le1$ and $1\le v_\alpha^{\rm mom}\le\frac32$.
\end{lemma}

Both cases also hold conditionally on a $\sigma$-algebra $\G$ with respect to which $x$ is measurable (Corollary~\ref{cor:conditional}). Appendix~\ref{app:clipping} gives the separate statements, proofs, and endpoint remarks. The bias is bounded by $bA(x)\lambda^{1-\alpha}$, and the second moment by $vA(x)\lambda^{2-\alpha}$. For $\alpha<2$, these bounds have the same powers of $\lambda$ as under the moment assumption considered by \citet{sadiev2023high}. Case~(a) applies to the strictly larger class $L^{\alpha,\infty}$ (Propositions~\ref{prop:embeddings} and~\ref{prop:pareto}). Its constants reflect the weaker assumption: $v_\alpha(s)\le v_\alpha(\infty)=\frac{2}{2-\alpha}(\frac32)^{2-\alpha}$ for $\alpha<2$, while at $\alpha=2$ the logarithmic factor in~\eqref{eq:BC-alpha} is unavoidable (Remark~\ref{rem:alpha-two}).

\section{Main Results}\label{sec:results}

In this section, we present high-probability convergence guarantees for \clipsgd with a constant clipping level tuned as in \citet{sadiev2023high}. To handle both noise models simultaneously, let $(b,v)$ be the clipping constants from Lemma~\ref{lem:clip} and $c:=180\max\{v,b/3\}$. Under Assumption~\ref{as:moment}, we have $(b,v)=(b_\alpha^{\rm mom},v_\alpha^{\rm mom})$ and $c=c_\alpha^{\rm mom}:=180(\frac32)^{2-\alpha}\in[180,270)$. Under Assumption~\ref{as:weak}, we have $(b,v)=(b_\alpha,v_\alpha(\Lambda))$ and $c=c_\alpha(\Lambda):=180\max\{v_\alpha(\Lambda),b_\alpha/3\}$, where $\Lambda$ is the ratio of the clipping level to the average noise scale: for the tuned clipping level, $\Lambda=\Lambda_K:=\max\{80(K+1),\Lambda_{2\sqrt2LR}\}$, where $\Lambda_\lambda$ is the explicit ratio~\eqref{eq:Lambda} from Appendix~\ref{app:main} ($\Lambda_\lambda=(3\lambda/(2\sigma_0))^\alpha$ if $\sigma_1=0$ and $\sigma_0>0$). For $\alpha<2$, $c_\alpha(\Lambda)$ is bounded by $c_\alpha(\infty)=\max\{180v_\alpha(\infty),60b_\alpha\}$, which depends only on $\alpha$, whereas $c_2(\Lambda)=180(1+\ln(1+\Lambda))$.

\begin{theorem}[High-probability convergence of \clipsgd]\label{thm:sgd}
Let Assumptions~\ref{as:smooth}, \ref{as:grad-bound}, and~\ref{as:convex} hold for $Q=B_{2R}(x^*)$, $R\ge\|x^0-x^*\|>0$, and let either Assumption~\ref{as:moment} or Assumption~\ref{as:weak} hold with $\alpha\in(1,2]$. Let $c$ be defined as above with $\Lambda=\Lambda_K$, $\beta\in(0,1)$, $K\ge0$, and $\ell:=\ln\frac{4(K+1)}{\beta}$. Then there exists an explicit stepsize $\gamma=\gamma_*(K,\beta)$ (see Theorem~\ref{thm:main} in Appendix~\ref{app:main}) such that with probability at least $1-\beta$ the iterates produced by Algorithm~\ref{alg:clip-sgd} with $\lambda_k\equiv\lambda:=\frac{R}{40\gamma\ell}$ satisfy $x^k\in B_{\sqrt2R}(x^*)$ for all $k=0,1,\dots,K+1$ and
\begin{align}
    f(\bar x^K)-f^*\le\max\Bigg\{\frac{160\sqrt2\,LR^2\ell}{K+1},\ \frac{2\left(c\,40^{\alpha-1}\right)^{\frac1\alpha}\sigma_0R\,\ell^{\frac{\alpha-1}{\alpha}}}{(K+1)^{\frac{\alpha-1}{\alpha}}},\ \frac{2\left(c\,2^\alpha40^{\alpha-1}\right)^{\frac2\alpha}\sigma_1^2LR^2\ell^{\frac{2(\alpha-1)}{\alpha}}}{(K+1)^{\frac{2(\alpha-1)}{\alpha}}}\Bigg\}.\label{eq:sgd-rate}
\end{align}
\end{theorem}

The full statement (Theorem~\ref{thm:main}) gives the bound $2R^2/(\gamma(K+1))$ for every admissible stepsize $\gamma\le\gamma_*$. Choosing $\gamma=\gamma_*$ yields~\eqref{eq:sgd-rate}; Corollary~\ref{cor:rate} gives the oracle complexity summarized in Table~\ref{tab:complexity}. Writing $r=(\alpha-1)/\alpha$, the three terms in the rate have polynomial orders $K^{-1}$, $K^{-r}$, and $K^{-2r}$: they correspond to the deterministic error, additive noise, and gradient-dependent noise, respectively.

The horizon logarithm in $\ell$ comes from a union bound. A maximal inequality replaces $\ell$ by $\ln(4/\beta)$ under the moment condition and under the weak-tail condition for $\alpha<2$ with $v=v_\alpha(\infty)$ (Theorem~\ref{thm:maximal}; see also \citealp{nguyen2023improved}). Proposition~\ref{prop:confidence-lower} establishes optimality of this confidence dependence for the additive stochastic term. The same replacement applies to \clipsstm. For prescribed clipping levels, it is valid under either noise condition for every $\alpha\in(1,2]$, because the clipping level does not depend on $\ell$.

\paragraph{Discussion.} Under Assumption~\ref{as:moment} with $\sigma_1=0$, the stochastic term $\sigma_0R\,K^{-(\alpha-1)/\alpha}$ cannot be improved, up to the logarithmic factor and the constant, by any algorithm making $K+1$ oracle calls, even if one asks for the guarantee with constant probability: Proposition~\ref{prop:lower-rate} (Appendix~\ref{app:lower}) provides a convex $L$-smooth instance on $\R$ with $\|x^0-x^*\|=R$ and $\E|\zeta|^\alpha\le\sigma_0^\alpha$ such that the error is at least $\frac14\min\{\sigma_0R(K+1)^{-(\alpha-1)/\alpha},2LR^2\}$ with probability at least $1/4$; the proof relies on the two-point argument of \citet{bubeck2013bandits}. Since Assumption~\ref{as:moment} implies Assumption~\ref{as:weak} with the same parameters, the same lower bound holds for the larger class as well. That is, for $\alpha<2$ the weaker assumption costs nothing in the rate and affects only the $\alpha$-dependent constants. Next, the terms depending on $\sigma_1$ decay faster in $K$ than the $\sigma_0$-term; for $\alpha=2$ under Assumption~\ref{as:moment}, the last term in~\eqref{eq:sgd-rate} is of the order $\sigma_1^2LR^2\ell/K$, which together with the first term recovers the well-known $(1+\sigma_1^2)LR^2/K$ scaling under affine variance.

\paragraph{The endpoint $\alpha=2$ and the Pareto example.} Under Assumption~\ref{as:weak} with $\alpha=2$, the noise may have infinite variance at every point (Section~\ref{sec:pareto}), meaning that the analyses relying on the bounded second moment are not applicable. Nevertheless, Theorem~\ref{thm:sgd} gives the rate of the finite-variance case, $\max\{\frac{LR^2\ell}{K},\sigma_0R\sqrt{\frac{c\ell}{K}},\frac{c\sigma_1^2LR^2\ell}{K}\}$ up to absolute constants, with the logarithmic factor $c=c_2(\Lambda_K)=O(\ln(K+1)+\ln(1+\Lambda_{2\sqrt2LR}))$; the same holds for \clipsstm (Theorem~\ref{thm:sstm-main}). A logarithmic factor is unavoidable both in the clipping lemma (Remark~\ref{rem:alpha-two}) and in the rate: at a fixed confidence level, no algorithm using $K$ stochastic gradients can avoid the factor $\sqrt{\ln(eK)}$ in the additive term over the weak-$L^2$ class, and, more generally, the factor $[\ln(eK)]^{1/2-1/s}$ on the Lorentz scale $L^{2,s}$ (Proposition~\ref{prop:lorentz-opt-lower} in Appendix~\ref{app:lorentz}; this order is attained, see Corollary~\ref{cor:lorentz-upper}). Therefore, for the Pareto example from Section~\ref{sec:pareto}, the weak-$L^\alpha$ analysis gives the exponent $(\alpha-1)/\alpha$ with a $K$-independent constant for $\alpha<2$ and the finite-variance exponent $1/2$ up to a logarithmic factor for $\alpha=2$, whereas the moment-based analysis can be applied only with some $q<\alpha$ and $\sigma_0=(\frac{\alpha}{\alpha-q})^{1/q}\sigma$: a fixed $q$ leads to the slower exponent $(q-1)/q$, and choosing $q\to\alpha$ with the growth of the horizon recovers the exponent $(\alpha-1)/\alpha$ only up to a factor of the order $(\ln K)^{1/\alpha}$ (see Corollary~\ref{cor:pareto} in Appendix~\ref{app:main}).

\subsection{Acceleration: Clip-SSTM}\label{sec:sstm}

Next, we consider the accelerated version of \clipsgd called Clipped Stochastic Similar Triangles Method (\clipsstm) \citep{gorbunov2020stochastic,sadiev2023high}, see Algorithm~\ref{alg:clip-sstm}. For a parameter $a\ge1$, the method uses the weights $\omega_{k+1}:=\frac{k+2}{2aL}$ and $A_{k+1}:=A_k+\omega_{k+1}$, so that $A_k=\frac{k(k+3)}{4aL}$. We analyze \clipsstm under either of the two noise models with $\alpha\in(1,2]$ and with the clipping constants $(b,v)$ from Section~\ref{sec:results}; under Assumption~\ref{as:weak}, $v=v_\alpha(\Lambda_K^{\rm S})$ for an explicit $\Lambda_K^{\rm S}\ge200\,b\,K$ defined in~\eqref{eq:Lambda-sstm} (Appendix~\ref{app:sstm}), which plays the same role as $\Lambda_K$ in Theorem~\ref{thm:sgd}.

\begin{algorithm}[t]
\caption{\clipsstm}\label{alg:clip-sstm}
\begin{algorithmic}[1]
\REQUIRE starting point $x^0$, number of iterations $K\ge1$, parameter $a\ge1$, clipping levels $\{\lambda_k\}_{k=0}^{K-1}$
\STATE Set $A_0=0$ and $y^0=z^0=x^0$ (the supplied starting point)
\FOR{$k=0,1,\dots,K-1$}
\STATE $\omega_{k+1}=\frac{k+2}{2aL}$, $A_{k+1}=A_k+\omega_{k+1}$
\STATE $x^{k+1}=(A_ky^k+\omega_{k+1}z^k)/A_{k+1}$
\STATE Sample $\xi^k\sim\D$ independently of $\G_k$; set $\tg_k=\clip(g(x^{k+1},\xi^k),\lambda_k)$
\STATE $z^{k+1}=z^k-\omega_{k+1}\tg_k$
\STATE $y^{k+1}=(A_ky^k+\omega_{k+1}z^{k+1})/A_{k+1}$
\ENDFOR
\ENSURE $y^K$
\end{algorithmic}
\end{algorithm}

\begin{theorem}[High-probability convergence of \clipsstm]\label{thm:sstm-main}
Let Assumptions~\ref{as:smooth}, \ref{as:grad-bound}, and~\ref{as:convex} hold for $Q=B_{3R}(x^*)$, $R\ge\|x^0-x^*\|>0$, and let either Assumption~\ref{as:moment} or Assumption~\ref{as:weak} hold with $\alpha\in(1,2]$. Let $K\ge1$ be an integer, $\beta\in(0,1]$, $H:=\ln\frac{4K}{\beta}$, and $r:=\frac{\alpha-1}{\alpha}$. Then there exist an explicit parameter $a_*=a_*(K,\beta)$ and constants $c_{0,\alpha},c_{1,\alpha}$, determined by $(b,v)$ and $\alpha$ (see Theorem~\ref{thm:sstm} in Appendix~\ref{app:sstm}), such that on one event of probability at least $1-\beta$ all points $x^k,y^k,z^k$, $k=0,\dots,K$, produced by Algorithm~\ref{alg:clip-sstm} with $a=a_*$ and $\lambda_k=\frac{R}{30\,\omega_{k+1}H}$ lie in $B_{2R}(x^*)$ and
\begin{align}
    f(y^K)-f^*\le\max\Big\{\frac{1058400\,LR^2H^2}{K^2},\ \frac{12c_{0,\alpha}\sigma_0RH^r}{K^r},\ \frac{6c_{1,\alpha}\sigma_1^2LR^2H^{2r}}{K^{2r}}\Big\}.\label{eq:sstm-rate-main}
\end{align}
Under Assumption~\ref{as:moment}, and under Assumption~\ref{as:weak} with $\alpha<2$ (if one computes $v=v_\alpha(\infty)$, which is allowed), the constants $c_{0,\alpha}$ and $c_{1,\alpha}$ depend only on $\alpha$ (at $\alpha=2$ under Assumption~\ref{as:moment}, $c_{0,2}=\sqrt{2700}$ and $c_{1,2}=49\cdot5400$); at $\alpha=2$ under Assumption~\ref{as:weak}, $c_{0,2}=O(\sqrt v)$ and $c_{1,2}=O(v)$ with $v=1+\ln(1+\Lambda_K^{\rm S})=O(\ln(eK)+\ln(1+\Lambda_K^0))$, where $\Lambda_K^0$ is defined in~\eqref{eq:Lambda-sstm}.
\end{theorem}

The oracle complexity implied by~\eqref{eq:sstm-rate-main} is given in Table~\ref{tab:complexity} (see Corollary~\ref{cor:sstm-complexity}). When $\sigma_1=0$, bound~\eqref{eq:sstm-rate-main} has the same powers of $K$ as the moment-based results of \citet{gorbunov2020stochastic} ($\alpha=2$) and \citet{sadiev2023high}, but it holds for the larger class $L^{\alpha,\infty}$; at $\alpha=2$ under Assumption~\ref{as:weak}, it gives the finite-variance rates for the noise whose variance may be infinite, with the logarithmic factor $v$, which Proposition~\ref{prop:lorentz-sstm} (Appendix~\ref{app:lorentz}) refines on the Lorentz scale $L^{2,s}$.

\paragraph{Role of the direct Lorentz analysis.} For fixed $L,R,\sigma_0>0$, additive noise, and fixed confidence, the moment-embedding route in Corollary~\ref{cor:lorentz-upper} improves the direct Lorentz bound by a factor $\sqrt{\ln(eK)}$. Its additive term matches the lower bound in Proposition~\ref{prop:lorentz-opt-lower} in the stochastic regime. The direct analysis remains useful because it gives explicit finite-horizon parameters under the original affine Lorentz condition. Appendix~\ref{app:lorentz} gives the detailed comparison. Neither result establishes optimality of the multiplicative term or of the joint dependence on $K$ and a vanishing failure probability.

\paragraph{Why the affine structure is preserved.} The key idea of the proof is to bound the true gradients along the trajectory by induction via the function gap: on the induction event, the potential of SSTM gives $A_t(f(y^t)-f^*)\le\frac32R^2$, and Assumptions~\ref{as:smooth}--\ref{as:grad-bound} turn this into $\omega_{t+1}\|\nabla f(x^{t+1})\|\le7R/\sqrt a$, which implies $\sum_{t<T}\omega_{t+1}^\alpha\|\nabla f(x^{t+1})\|^\alpha\le7^\alpha R^\alpha T/a^{\alpha/2}$ (see Lemma~\ref{lem:budget} in Appendix~\ref{app:sstm-budget}), whereas the crude bound $\|\nabla f(x^{t+1})\|\le2LR$ gives only $R^\alpha K^{\alpha+1}/a^\alpha$.

This weighted estimate preserves the gradient-dependent noise scale throughout the clipping analysis. As a result, the multiplicative contribution is bounded by $\sigma_1^2LR^2(H/K)^{2r}$, compared with $\sigma_1LR^2(H/K)^r$ when the noise is bounded uniformly on $Q$ (Corollary~\ref{cor:uniform}). This improves the bound's dependence on $K$; it does not establish optimality of the multiplicative term. Faster rates are known under additional multiplicative-noise structure (Appendix~\ref{app:sstm-uniform}).

\subsection{Prescribed Clipping Levels}\label{sec:planned}

The clipping levels in Theorems~\ref{thm:sgd} and~\ref{thm:sstm-main} depend on the parameters of the problem and on the horizon, whereas in practice $\lambda$ is a constant chosen by the user. In this section, we allow the user to prescribe $\lambda$. Theorem~\ref{thm:fixed-main} bounds the error of the computable averaged iterate $\bar x^K$ for any $\lambda\ge2\sqrt2LR$; this is a geometric condition guaranteeing that the true gradient is never clipped in $B_{\sqrt2R}(x^*)$, and it involves neither $\beta$ nor $K$ (nor the noise, except for the case $\alpha=2$ with $\sigma_1>0$). Theorem~\ref{thm:any-main} covers \emph{any} $\lambda>0$ (for additive noise or for multiplicative noise below an absolute threshold; under Assumption~\ref{as:weak}, only for $\alpha<2$), but it bounds the best iterate $\min_{t\le K}f(x^t)$, which cannot be selected using the gradient information only. In both results, only the clipping level is free: the stepsizes still depend on the noise and confidence parameters. Below, $(b,v)$ are the clipping constants from Section~\ref{sec:results} with $\Lambda=\Lambda_\lambda$ under Assumption~\ref{as:weak}; only $b$ affects the size of the neighborhood.

\begin{theorem}[\clipsgd with a prescribed clipping level]\label{thm:fixed-main}
Let Assumptions~\ref{as:smooth}, \ref{as:grad-bound}, and~\ref{as:convex} hold for $Q=B_{2R}(x^*)$, $R\ge\|x^0-x^*\|>0$, and let either Assumption~\ref{as:moment} or Assumption~\ref{as:weak} hold with $\alpha\in(1,2]$. Let $\beta\in(0,1)$, $K\ge0$, $\ell:=\ln\frac{4(K+1)}{\beta}$, and let $\lambda\ge2\sqrt2LR$ be arbitrary (with $\lambda\ge240\,b\,\sigma_1^2LR$ if $\alpha=2$ and $\sigma_1>0$). Then there exists an explicit stepsize $\gamma=\gamma(\lambda,K,\beta)\le\frac{R}{40\lambda\ell}$ (see Corollary~\ref{cor:fixed} in Appendix~\ref{app:main}) such that with probability at least $1-\beta$ the iterates produced by Algorithm~\ref{alg:clip-sgd} with $\lambda_k\equiv\lambda$ satisfy $x^k\in B_{\sqrt2R}(x^*)$ for all $k=0,1,\dots,K+1$ and
\begin{equation}\label{eq:fixed-bound-main}
    f(\bar x^K)-f^*\le\max\left\{\frac{80\lambda R\ell}{K+1},\,\mathcal N_\lambda,\,\mathcal V_{\lambda,K}\right\},
\end{equation}
where the neighborhood
\begin{equation}\label{eq:neighborhood-main}
    \mathcal N_\lambda:=\max\left\{\frac{120\,bR\sigma_0^\alpha}{\lambda^{\alpha-1}},2\left(\frac{60\cdot2^\alpha bRL^{\frac\alpha2}\sigma_1^\alpha}{\lambda^{\alpha-1}}\right)^{\frac2{2-\alpha}}\right\}
\end{equation}
does not depend on $K$ (its second term is absent for $\alpha=2$), whereas $\mathcal V_{\lambda,K}$ vanishes as $K\to\infty$: for $\sigma_1=0$, $\mathcal V_{\lambda,K}=120\sqrt2\,R\sqrt{v\lambda^{2-\alpha}\sigma_0^\alpha\ell/(K+1)}$, and for $\sigma_1>0$ it has three more terms, of the orders $K^{-1}$, $K^{-1}$ and $K^{-2/(4-\alpha)}$ up to logarithmic factors. If $\sigma_1=0$ and $\lambda=\lambda_0(K+1)^{1/\alpha}$ with $\lambda_0>0$ and $K+1\ge(2\sqrt2LR/\lambda_0)^\alpha$, then, with probability at least $1-\beta$, $f(\bar x^K)-f^*\le\Xi R\,(K+1)^{-(\alpha-1)/\alpha}$ with $\Xi:=80\lambda_0\ell+120\,b\,\sigma_0^\alpha\lambda_0^{1-\alpha}+120\sqrt{2v\ell}\,\sigma_0^{\alpha/2}\lambda_0^{1-\alpha/2}$.
\end{theorem}

Theorem~\ref{thm:fixed-main} concerns a separate run for each horizon $K$, with the stepsize chosen for that horizon. As $K$ increases, the bound approaches $\mathcal N_\lambda$; the slowest vanishing term has order $K^{-1/2}$ up to logarithms. The additive part of $\mathcal N_\lambda$ is proportional to the clipping-bias constant $b$ from Lemma~\ref{lem:clip}; the factor $120$ comes from the optimization analysis. This neighborhood does not depend on the variance constant $v$, including its logarithm at $\alpha=2$. Its dependence on $R$, $\sigma_0$, and $\lambda$ is unimprovable up to $\alpha$-dependent constants in the nontrivial regime (Theorem~\ref{thm:lower-main}). The multiplicative neighborhood decreases faster in $\lambda$.

For $\sigma_1=0$ and $\lambda=\lambda_0(K+1)^{1/\alpha}$, any $\lambda_0>0$ recovers the polynomial rate from Theorem~\ref{thm:sgd}; Appendix~\ref{app:main-statements} quantifies the cost of misspecifying the noise level. Under the moment condition, the non-private specialization of \citet{vatankhah2026differentially} gives the same three-term order for the best iterate when $\lambda>4LR$, after combining it with the initial-point bound $f(x^0)-f^*\le LR^2/2$ (see the discussion of Table~\ref{tab:prescribed}).

\paragraph{Arbitrary clipping levels.} When $\lambda<2\sqrt2LR$, the true gradient may be clipped along the trajectory, and the proof of Theorem~\ref{thm:fixed-main} does not work. The next result covers any $\lambda>0$, for additive noise and for multiplicative noise below an absolute threshold. Its neighborhood depends on $\lambda$ only through the bias scale
\begin{equation}\label{eq:bias-scale}
    \bar b_\lambda:=\frac{\sigma_0^\alpha}{\max\{\lambda,\sigma_0\}^{\alpha-1}},
\end{equation}
which equals $\sigma_0^\alpha/\lambda^{\alpha-1}$ for $\lambda\ge\sigma_0$ and saturates at $\sigma_0$ for $\lambda<\sigma_0$. Let $(\kappa_1,\kappa_2):=(1,1)$ under Assumption~\ref{as:moment} and $(\kappa_1,\kappa_2):=\left(\frac{\alpha}{\alpha-1},\frac{2}{2-\alpha}\right)$ under Assumption~\ref{as:weak} with $\alpha<2$, and let $c_\perp:=2^{\alpha-1}(\kappa_2+3\kappa_1)$ and $\bar\sigma_1:=\frac{1}{2\cdot28^{1/\alpha}}$.

\begin{theorem}[Any clipping level]\label{thm:any-main}
Let Assumptions~\ref{as:smooth}, \ref{as:grad-bound}, and~\ref{as:convex} hold for $Q=B_{2R}(x^*)$, $R\ge\|x^0-x^*\|>0$, and let Assumption~\ref{as:moment} (with $\alpha\in(1,2]$) or Assumption~\ref{as:weak} (with $\alpha\in(1,2)$) hold with $\sigma_1\le\bar\sigma_1$. Let $\lambda>0$ be arbitrary, $\beta\in(0,1)$, $K\ge0$, and $\ell:=\ln\frac{4(K+1)}{\beta}$. Then there exists an explicit stepsize $\gamma=\gamma(\lambda,K,\beta)$ (see Theorem~\ref{thm:any-lambda} in Appendix~\ref{app:any-lambda}) such that with probability at least $1-\beta$ the iterates produced by Algorithm~\ref{alg:clip-sgd} with $\lambda_k\equiv\lambda$ satisfy
\begin{equation}\label{eq:any-bound-main}
    \min_{0\le t\le K}f(x^t)-f^*\le\max\left\{\varepsilon_N,\ \varepsilon_M,\ \mathcal V'_{\lambda,K}\right\},
\end{equation}
where $\varepsilon_N:=48\sqrt2\,c_\perp R\,\bar b_\lambda$, $\varepsilon_M:=384\,c_\perp\sigma_1^\alpha LR^2\min\{1,\,96\,c_\perp\sigma_1^\alpha\}$, and $\mathcal V'_{\lambda,K}$ collects the terms vanishing as $K\to\infty$, of the orders $K^{-1/2}$, $K^{-1}$ and $K^{-2}$ up to logarithmic factors, which are given explicitly in Theorem~\ref{thm:any-lambda}.
\end{theorem}

Theorem~\ref{thm:any-main} bounds the first time an iterate reaches the stated neighborhood, so its stepsize need not decrease as $1/K$. The tuning still depends on the horizon. The guarantee concerns the best iterate in objective value, as in \citet{vatankhah2026differentially}; selecting an output with the same guarantee requires function values or an additional validation procedure.

For $\sigma_1>0$, the bound has the additional term $\varepsilon_M=O(\sigma_1^{2\alpha}LR^2)$, which does not decrease in $\lambda$. The key new ingredient of the proof is a clipping lemma for arbitrary clipping levels (Lemma~\ref{lem:arb-clip}): when $\|\nabla f(x)\|>\lambda/2$ and $\|\nabla f(x)\|\gtrsim\sigma_0$, the mean clipped gradient has a component of size at least $\min\{\|\nabla f(x)\|,\lambda\}/6$ along $\nabla f(x)$, provided that $\sigma_1\le\bar\sigma_1$, and an orthogonal component bounded by a quantity of order $\lambda A(\nabla f(x))/\max\{\|\nabla f(x)\|,\lambda\}^\alpha$. The analysis of \citet{vatankhah2026differentially} compares the mean clipped gradient with $\clip(\nabla f(x),\lambda/2)$. It covers arbitrary $\lambda>0$, with additional terms involving $\zeta_\lambda:=\max\{0,2LR-\lambda/2\}$ when $\lambda<4LR$. Table~\ref{tab:prescribed} compares the regime $\lambda>4LR$, where $\zeta_\lambda=0$. We also notice that some restriction on $\sigma_1$ is necessary for small clipping levels: for the multiplicative oracle of \citet{faw2023beyond} ($\sigma_0=0$, $\sigma_1\le9$) on $f(x)=\frac L2x^2$, \clipsgd with any fixed $\lambda>0$ and $\gamma_k\le1/L$ started from $x^0=\lambda/L$ stays at distance $\ge\lambda/(2L)$ from $x^*$ forever with probability $\ge2/3$ (Proposition~\ref{prop:diverge}).

\subsection{Lower Bounds for a Fixed Clipping Level}\label{sec:lower}

The next result shows that the additive-noise neighborhood from Theorem~\ref{thm:any-main}, and hence the first part of $\mathcal N_\lambda$ in Theorem~\ref{thm:fixed-main}, cannot be improved for any $\lambda>0$, up to the factors depending on $\alpha$ and the trivial level $LR^2$. More precisely, for any $\lambda>0$ we construct a one-dimensional problem with additive noise such that \clipsgd with $\lambda_k\equiv\lambda$ and any deterministic stepsizes does not approach the solution.

\begin{theorem}[Lower bound for a fixed clipping level]\label{thm:lower-main}
Let $L,R,\sigma_0,\lambda>0$, $\alpha\in(1,2]$, and let $\varepsilon_\lambda:=\frac{1}{48\cdot4^\alpha}\min\{R\bar b_\lambda,LR^2\}$ with $\bar b_\lambda=\min\{\sigma_0,\sigma_0^\alpha/\lambda^{\alpha-1}\}$ from~\eqref{eq:bias-scale}. Then there exist a convex $L$-smooth function $f:\R\to\R$ with minimizer $x^*=0$, a point $x^0$ with $|x^0-x^*|=R$, and an oracle $g(x,\xi)=f'(x)+\zeta$ with $\E\zeta=0$ and $\E|\zeta|^\alpha\le\sigma_0^\alpha$ (i.e., Assumptions~\ref{as:smooth}--\ref{as:weak} hold on $\R$ with $\sigma_1=0$) such that \clipsgd with $\lambda_k\equiv\lambda$ and any deterministic stepsizes $\gamma_k>0$ satisfies $\E f(\hat x)-f^*\ge\varepsilon_\lambda$ for every iterate $\hat x=x^k$ and every convex combination $\hat x$ of $x^0,\dots,x^K$ with deterministic weights, in particular, for $\bar x^K$; moreover, $\E f(x^k)-f^*$ grows linearly in $\sum_{j<k}\gamma_j$. If $\gamma_k\le\frac{R}{16\lambda}$ for all $k$, then, with probability at least $3/4$, $f(x^k)-f^*\ge\varepsilon_\lambda/2$ and $f(\bar x^k)-f^*\ge\varepsilon_\lambda/2$ for all $k\ge0$. The explicit growth rate is given in Theorem~\ref{thm:lower} (Appendix~\ref{app:lower}).
\end{theorem}

\paragraph{Comparison with the upper bounds.} Let $\sigma_1=0$, and assume that the conditions of the upper-bound result under consideration are satisfied. Let $\bar b_\lambda=\min\{\sigma_0,\sigma_0^\alpha/\lambda^{\alpha-1}\}$. Theorem~\ref{thm:lower-main} gives the lower bound $\varepsilon_\lambda=(48\cdot4^\alpha)^{-1}\min\{R\bar b_\lambda,LR^2\}$ in expectation for every iterate and every deterministic average under arbitrary deterministic stepsizes; if $\gamma_k\le R/(16\lambda)$, it also gives the lower bound $\varepsilon_\lambda/2$ with probability at least $3/4$, simultaneously for all iterates and their averages. This restriction on the stepsizes is satisfied by the tuning $\gamma\le R/(40\lambda\ell)$ used in the upper bounds. In the regime $R\bar b_\lambda\le LR^2$, the best-iterate neighborhood $\varepsilon_N=48\sqrt2\,c_\perp R\bar b_\lambda$ from Theorem~\ref{thm:any-main}, at the confidence level $\beta<1/4$, matches this high-probability lower bound up to the factor $2\cdot48^2\sqrt2\cdot4^\alpha c_\perp$, which depends only on $\alpha$. For the averaged iterate, we impose $\lambda\ge2\sqrt2LR$ and the nontrivial-regime condition $R\sigma_0^\alpha/\lambda^{\alpha-1}\le LR^2$. These conditions imply $\sigma_0<\lambda$, and, therefore, the neighborhood $120\,bR\sigma_0^\alpha/\lambda^{\alpha-1}$ from Theorem~\ref{thm:fixed-main} matches the same lower bound up to the factor $11520\cdot4^\alpha b$; see Table~\ref{tab:prescribed} in Appendix~\ref{app:lower}. Note that these comparisons concern only the neighborhoods: the lower bound does not depend on $K$ and does not match the terms vanishing with the horizon.

The proof (Appendix~\ref{app:lower}) relies on a Huber function whose slope is smaller than the clipping bias of a two-point noise, so that the mean clipped field has the same sign everywhere and pushes the iterates away from $x^*$. Consequently, a uniform guarantee in expectation $\E f(\bar x^K)-f^*\le\varepsilon$ over the class of problems with additive noise and parameters $(L,R,\sigma_0,\alpha)$, where $\varepsilon<\min\{R\sigma_0,LR^2\}/(48\cdot4^\alpha)$, requires $\lambda\ge\sigma_0(\sigma_0R/(48\cdot4^\alpha\varepsilon))^{1/(\alpha-1)}$ regardless of the deterministic stepsizes; on the other hand, for the high-probability criterion with $\lambda=\sigma_0(120b\,\sigma_0R/\varepsilon)^{1/(\alpha-1)}$ (if this value is at least $2\sqrt2LR$), the neighborhood from Theorem~\ref{thm:fixed-main} equals $\varepsilon$, and the bound becomes at most $\varepsilon$ after $K=O_\alpha(v\ell(\sigma_0R/\varepsilon)^{\alpha/(\alpha-1)})$ iterations, which is the complexity from Table~\ref{tab:complexity}. Moreover, the expected error of the last iterate also grows linearly in the total stepsize, which makes the scale of the bias condition from Theorem~\ref{thm:fixed-main} necessary for the averaged iterate with a constant stepsize, and the bias persists even on strongly convex quadratics (see Appendix~\ref{app:lower-statement}).

\section{Numerical Experiments}\label{sec:experiments}

We illustrate the effect of clipping under weak-$L^\alpha$ noise with an infinite $\alpha$-th moment. On a quadratic problem with asymmetric noise, we compare the plateaus produced by prescribed clipping levels with the decreasing errors obtained using a horizon-dependent level. Appendix~\ref{app:experiments} reports additional experiments with symmetric noise, ridge regression, and the full theoretical parameter choices for \clipsgd and \clipsstm.

\paragraph{Protocol.} We use $\alpha=3/2$, $1024$ independent runs for each configuration, and horizons $n\le10^6$ oracle calls, and we report the mean of $f(\bar x^{n-1})-f^*$ together with pointwise approximate $95\%$ confidence bands; the horizon-dependent configurations are rerun for each horizon, with $\beta=0.05$. The noise is built from a Pareto random variable $P$ with $\PP\{P>t\}=t^{-\alpha}$ for $t\ge1$ (so that $\E P=3$ and $\E P^\alpha=\infty$): $g(x)=x+S(x)Z$ with the centered asymmetric $Z=(P-3)/2$ and $S(x)=(\sigma_0^\alpha+\sigma_1^\alpha|x|^\alpha)^{1/\alpha}$. This oracle satisfies the affine weak-tail inequality with $(\sigma_0,\sigma_1)$, and its $\alpha$-th moment is infinite whenever the noise scale is non-zero.

\paragraph{Influence of the clipping level.} We consider the quadratic $f(x)=x^2/2$ on $\R$ ($L=1$, $x^*=0$, $x^0=1$) with the asymmetric noise and $(\sigma_0,\sigma_1)=(0.05,0.001)$, and compare the constant clipping levels $\lambda\in\{0.1,0.5,1,4\}$ with the level $\lambda(n):=R/(40\gamma_*(n)\ell_n)$ from Theorem~\ref{thm:sgd}, computed for each horizon $n$ with $\beta=0.05$ and the stepsize $\gamma_*(n)$ from Theorem~\ref{thm:main}; this level grows as $n^{1/\alpha}$ up to logarithmic factors. All variants use $\gamma=0.05$ (see Figure~\ref{fig:exp-clipping-bias}).

Figure~\ref{fig:exp-clipping-bias} shows distinct plateaus for constant clipping levels, with smaller plateaus at larger levels. The dotted lines show objective values at zeros of the mean clipped field and provide a reference for these plateaus. With $\lambda(n)$, the observed error continues to decrease over the tested horizons, consistently with the displayed reference slope and the heuristic $n^{-2/3}$ balance between squared bias and the variance of averaged clipped noise. Increasing the level reduces clipping bias while increasing the second-moment bound. All configurations use the empirical stepsize $\gamma=0.05$, so this comparison isolates the clipping level rather than testing the theorem's full parameter prescription.

\begin{figure}[t]
    \centering
    \includegraphics[width=0.6\linewidth]{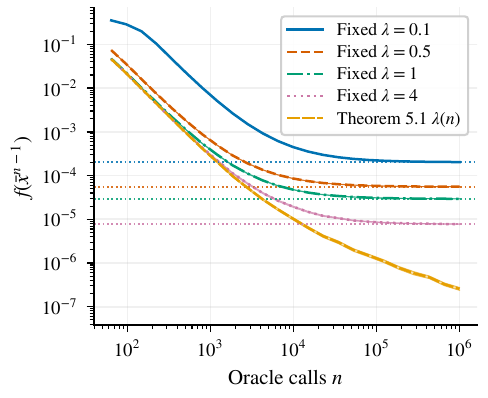}
    \caption{Influence of the clipping level under centered asymmetric Pareto noise on $f(x)=x^2/2$.}
    \label{fig:exp-clipping-bias}
\end{figure}

\section{Conclusion}\label{sec:conclusion}

We derived high-probability guarantees for \clipsgd and \clipsstm under affine weak-$L^\alpha$ noise. For $\alpha<2$, the tail condition preserves the polynomial rates available under a bounded $\alpha$-th moment. At $\alpha=2$, the additive term requires a logarithmic correction; our Lorentz-scale upper and lower bounds identify its optimal asymptotic order at fixed confidence and fixed positive problem parameters. The analysis also separates the additive and gradient-dependent noise contributions and gives guarantees for prescribed clipping levels. In the additive case, matching lower bounds characterize the resulting clipping-bias neighborhood up to constants and truncation at $LR^2$.

Several questions remain open. These include lower bounds for the multiplicative term, sharper guarantees for small clipping levels under multiplicative noise, and the joint dependence on the horizon and the failure probability at the endpoint.

\section*{AI use statement}

The main results presented in this document were obtained by the authors. We used a generative AI assistant to improve the writing, strengthen the initial core results, review mathematical arguments, and check references. The authors take responsibility for the final content of this work.

\bibliography{refs}

\clearpage
\appendix
\addtocontents{toc}{\protect\setcounter{tocdepth}{2}}
\thispagestyle{empty}

\begin{center}
{\LARGE\bf Weak-$L^\alpha$ Noise, Strong Guarantees:\\[0.3em] Gradient Clipping under Affine Heavy Tails\\[0.6em] Supplementary Material}
\end{center}
\vskip 1.5em

\setcounter{tocdepth}{2}
\tableofcontents
\clearpage

\section{Notation}\label{app:notation}

In Table~\ref{tab:notation}, we summarize the notation and the constants used in the statements of the paper, in the order of their appearance.

\begin{center}
\small
\renewcommand{\arraystretch}{1.25}
\begin{longtable}{p{0.22\linewidth}>{\raggedright\arraybackslash}p{0.53\linewidth}>{\raggedright\arraybackslash}p{0.17\linewidth}}
\caption{Notation.}\label{tab:notation}\\
\hline
\textbf{Symbol} & \textbf{Definition} & \textbf{Introduced in}\\
\hline
\endfirsthead
\hline
\textbf{Symbol} & \textbf{Definition} & \textbf{Introduced in}\\
\hline
\endhead
\hline
\endfoot
\multicolumn{3}{l}{\emph{Problem, oracle and method}}\\
$\abr{\cdot,\cdot}$, $\|\cdot\|$, $B_r(x)$, $(a)_+$, $\ind_A$ & Euclidean inner product and norm; closed ball of radius $r$ around $x$; positive part; indicator of $A$ & Section~\ref{sec:setup}\\
$f$, $g(x,\xi)$, $\D$, $x^*$, $f^*$ & objective, unbiased stochastic oracle with $\xi\sim\D$ ($\E_\xi g(x,\xi)=\nabla f(x)$ on $Q$), a minimizer $x^*$ and $f^*=f(x^*)$ & \eqref{eq:problem}, Section~\ref{sec:setup}\\
$x^0$, $R$ & starting point and a number $R\ge\|x^0-x^*\|>0$ & Section~\ref{sec:setup}\\
$Q=B_{\rho R}(x^*)$ & ball on which all assumptions are imposed; $\rho=2$ for \clipsgd, $\rho=3$ for \clipsstm & Section~\ref{sec:setup}\\
$L$ & local smoothness constant on $Q$ & Assumption~\ref{as:smooth}\\
$\G_k$ & $\sigma$-algebra generated by $x^0,\xi^0,\dots,\xi^{k-1}$ ($\G_0=\sigma(x^0)$) & Section~\ref{sec:setup}\\
$\sigma_0$, $\sigma_1$, $\alpha$ & additive and multiplicative noise scales and tail index, $\alpha\in(1,2]$ & Assumptions~\ref{as:moment}, \ref{as:weak}\\
$\wnorm{Y}$, $L^{\alpha,\infty}$ & weak-$L^\alpha$ quasi-norm $\sup_{u>0}u\,\PP\{|Y|>u\}^{1/\alpha}$ and the corresponding space & Definition~\ref{def:weak}, \eqref{eq:weak-qn}\\
$d_Y$, $Y^*$, $Y^{**}$ & distribution function $d_Y(u)=\PP\{|Y|>u\}$, non-increasing rearrangement $Y^*(t)=\inf\{u>0:d_Y(u)\le t\}$ and its average $Y^{**}(t)=\frac1t\int_0^tY^*$ & Appendix~\ref{app:spaces}\\
$\Mnorm{Y}$, $M_\alpha$ & Marcinkiewicz norm $\sup_{0<t\le1}t^{1/\alpha}Y^{**}(t)$ and space ($\alpha>1$) & Definition~\ref{def:marc}\\
$\|Y\|_{2,s}$, $L^{2,s}$ & Lorentz quasi-norm $(\int_0^1[t^{1/2}Y^*(t)]^s\frac{\dif t}{t})^{1/s}$, $2\le s<\infty$, $\|Y\|_{2,\infty}=\wnorm{Y}|_{\alpha=2}$, and the corresponding space & Appendix~\ref{app:lorentz}\\
$\clip(v,\lambda)$ & clipping operator $\min\{1,\lambda/\|v\|\}v$ & Section~\ref{sec:setup}\\
$\lambda$, $\lambda_k$, $\gamma$, $K$ & clipping level(s), stepsize, horizon: \clipsgd makes $K+1$ iterations $k=0,\dots,K$, \clipsstm makes $K$ & Algorithms~\ref{alg:clip-sgd}, \ref{alg:clip-sstm}\\
$g^k$, $\tg^k$, $\bar x^K$ & stochastic gradient, clipped stochastic gradient, averaged iterate $\frac{1}{K+1}\sum_{k=0}^Kx^k$ of \clipsgd & Algorithm~\ref{alg:clip-sgd}\\
$x^k,y^k,z^k$, $a$, $\omega_{k+1}$, $A_k$ & \clipsstm sequences; parameter $a\ge1$; weights $\omega_{k+1}=\frac{k+2}{2aL}$, $A_{k+1}=A_k+\omega_{k+1}$, $A_k=\frac{k(k+3)}{4aL}$ & Algorithm~\ref{alg:clip-sstm}, \eqref{eq:Ak}\\
$\beta$ & failure probability of the high-probability statements & Theorem~\ref{thm:main}\\
$\ell$, $H$, $h$ & $\ell=\ln\frac{4(K+1)}{\beta}$ (\clipsgd), $H=\ln\frac{4K}{\beta}$ (\clipsstm), $h=\ln\frac4\beta$ (horizon-free confidence) & Theorems~\ref{thm:main}, \ref{thm:sstm}, \eqref{eq:h}\\
$A(x)$ & affine noise scale $\sigma_0^\alpha+\sigma_1^\alpha\|x\|^\alpha$ (so that $A(\nabla f(x))$ bounds the tail or the moment of the noise at $x$) & Section~\ref{sec:clipping}\\
$\theta_k$, $\theta_k^u$, $\theta_k^b$ & error of the clipped gradient $\tg^k-\nabla f(x^k)$ and its decomposition into the martingale part $\tg^k-\E[\tg^k\mid\G_k]$ and the bias $\E[\tg^k\mid\G_k]-\nabla f(x^k)$ & \eqref{eq:theta}\\
$R_k$, $\delta_k$ & $\|x^k-x^*\|$ and $f(x^k)-f^*$ & Appendix~\ref{app:main}\\
\multicolumn{3}{l}{\emph{Clipping constants}}\\
$b_\alpha$ & bias constant under the weak-$L^\alpha$ condition, $\frac{2^{\alpha-1}}{\alpha-1}$ & Lemma~\ref{lem:weak-clip}\\
$v_\alpha(s)$, $v_\alpha$ & variance constant under the weak-$L^\alpha$ condition, $(\frac32)^{2-\alpha}\min\{\frac{2}{2-\alpha},1+\frac2\alpha\ln(1+s)\}$ for $s\in[0,\infty]$ ($\frac{2}{2-\alpha}=+\infty$ at $\alpha=2$); $v_\alpha=v_\alpha(\infty)=\frac{2}{2-\alpha}(\frac32)^{2-\alpha}$ for $\alpha<2$ & \eqref{eq:BC-alpha}, Section~\ref{sec:clipping}\\
$b_\alpha^{\rm mom}$, $v_\alpha^{\rm mom}$ & bias and variance constants under the moment condition, $\frac{2^{\alpha-1}(\alpha-1)^{\alpha-1}}{\alpha^\alpha}\in[\frac12,1]$ and $(\frac32)^{2-\alpha}\in[1,\frac32]$ & \eqref{eq:moment-clipping-constants}\\
$(b,v)$ & generic clipping constants: $(b_\alpha^{\rm mom},v_\alpha^{\rm mom})$ under Assumption~\ref{as:moment}, $(b_\alpha,v_\alpha(\Lambda))$ under Assumption~\ref{as:weak}, with $\Lambda=\Lambda_K$ (Theorem~\ref{thm:main}), $\Lambda_\lambda$ (Theorem~\ref{thm:fixed}), $\Lambda_K^{\rm S}$ (Theorem~\ref{thm:sstm}) or $\Lambda_{\gamma,\lambda}$ (Proposition~\ref{prop:master}) & Section~\ref{sec:results}\\
$c$, $c_\alpha^{\rm mom}$, $c_\alpha(\Lambda)$ & rate constant $c=180\max\{v,b/3\}$: $c_\alpha^{\rm mom}=180(\frac32)^{2-\alpha}\in[180,270)$ under Assumption~\ref{as:moment}, $c_\alpha(\Lambda)=180\max\{v_\alpha(\Lambda),b_\alpha/3\}$ under Assumption~\ref{as:weak} & \eqref{eq:moment-rate-constant}, Section~\ref{sec:results}\\
$\Lambda_\lambda$, $\Lambda_K$ & $\Lambda_\lambda=\frac{(K+1)(3\lambda/2)^\alpha}{(K+1)\sigma_0^\alpha+2^{\alpha/2}\sigma_1^\alpha(LR)^\alpha}$ (ratio of the clipping level to the average noise scale; $0$ if $\sigma_0=\sigma_1=0$), $\Lambda_K=\max\{80(K+1),\Lambda_{2\sqrt2LR}\}$ & \eqref{eq:Lambda}\\
$\Lambda_{\gamma,\lambda}$ & $\frac{(K+1)(3\lambda/2)^\alpha}{\Sigma_{K+1}}\le\Lambda_\lambda$, the ratio used in the proofs for \clipsgd & \eqref{eq:Lambda-gl}\\
$\Lambda_K^{\rm S}$, $\Lambda_K^0$ & $\Lambda_K^{\rm S}=\max\{200bK,\Lambda_K^0\}$ with $\Lambda_K^0=[(\frac{\sigma_0(K+1)}{17640LRH})^\alpha+(\frac{\sigma_1}{3})^\alpha]^{-1}$ ($0$ if $\sigma_0=\sigma_1=0$), the analog of $\Lambda_K$ for \clipsstm & \eqref{eq:Lambda-sstm}\\
\multicolumn{3}{l}{\emph{\clipsgd with the tuned clipping level (Theorem~\ref{thm:main}, Corollary~\ref{cor:rate})}}\\
$\gamma_L,\gamma_0,\gamma_1,\gamma_2$ & stepsize thresholds $\frac{1}{80\sqrt2L\ell}$, $\frac{R}{(c40^{\alpha-1}(K+1)\ell^{\alpha-1})^{1/\alpha}\sigma_0}$, $\frac{1}{(c2^\alpha40^{\alpha-1}\ell^{\alpha-1})^{2/\alpha}\sigma_1^2L(K+1)^{(2-\alpha)/\alpha}}$, $\frac{1}{(c2^{\alpha/2}40^{\alpha-1}\ell^{\alpha-1})^{1/\alpha}\sigma_1L}$; the level is $\lambda=\frac{R}{40\gamma\ell}$ & Theorem~\ref{thm:main}\\
$\Delta_L,\Delta_0,\Delta_1$ ($\Delta_2$) & the three terms of the rate, $\Delta_j=\frac{2R^2}{\gamma_j(K+1)}$ & Corollary~\ref{cor:rate}\\
$\Sigma_T$, $\Phi_T$ & $\Sigma_T=\sigma_0^\alpha T+\sigma_1^\alpha2^\alpha L^{\alpha/2}R^\alpha\gamma^{-\alpha/2}T^{(2-\alpha)/2}+\sigma_1^\alpha2^{\alpha/2}L^\alpha R^\alpha$ (accumulated noise scale over $T$ steps) and $\Phi_T=\gamma\lambda^{1-\alpha}\Sigma_T$ (accumulated bias scale) & \eqref{eq:Sigma-def}\\
\multicolumn{3}{l}{\emph{\clipsstm (Theorem~\ref{thm:sstm}, Corollaries~\ref{cor:sstm-complexity}, \ref{cor:uniform})}}\\
$r$ & $\frac{\alpha-1}{\alpha}$, the exponent of the additive term & Theorem~\ref{thm:sstm}\\
$q_\alpha$, $c_{0,\alpha}$, $c_{1,\alpha}$ & $q_\alpha=\min\{\frac{1}{1350v30^{\alpha-2}},\frac{1}{15b30^{\alpha-1}}\}$, $c_{0,\alpha}=(2/q_\alpha)^{1/\alpha}$, $c_{1,\alpha}=49(4/q_\alpha)^{2/\alpha}$ (at $\alpha=2$ under Assumption~\ref{as:moment}: $q_2=\frac{1}{1350}$, $c_{0,2}=\sqrt{2700}$, $c_{1,2}=49\cdot5400$) & Theorem~\ref{thm:sstm}\\
$a_*(K,\beta)$ & $\max\{176400H^2,c_{0,\alpha}\frac{\sigma_0}{LR}(K+1)K^{1/\alpha}H^r,c_{1,\alpha}\sigma_1^2K^{2/\alpha}H^{2r}\}$, the prescribed value of $a$; the levels are $\lambda_k=\frac{R}{30\omega_{k+1}H}$ & Theorem~\ref{thm:sstm}\\
$C_\alpha$, $M_\varepsilon$ & $C_\alpha=\max\{1058400,12c_{0,\alpha},6c_{1,\alpha}\}$ and the horizon scale $M_\varepsilon$ of the oracle complexity & Corollary~\ref{cor:sstm-complexity}\\
$\bar\sigma$, $a_{\rm uni}$ & uniform noise scale $(\sigma_0^\alpha+(3LR\sigma_1)^\alpha)^{1/\alpha}$ on $Q=B_{3R}(x^*)$ and the corresponding value of $a$ & Section~\ref{sec:sstm}, Corollary~\ref{cor:uniform}\\
$C_{p,s}$, $p_K$ & $C_{p,s}=(\frac{2(s-p)}{s(2-p)})^{1/p-1/s}$ (embedding constant $\|Y\|_p\le C_{p,s}\|Y\|_{2,s}$) and the horizon-dependent exponent $p_K=2-\min\{\frac12,\frac1{\ln(eK)}\}$ of the moment route & Remark~\ref{rem:lorentz-moment}\\
\multicolumn{3}{l}{\emph{Prescribed clipping levels (Theorems~\ref{thm:fixed}, \ref{thm:any-lambda}, \ref{thm:lower}, Corollary~\ref{cor:fixed})}}\\
$S_0,S_1,S_2$, $\Sigma_{K+1}(\gamma)$ & $S_0=\sigma_0^\alpha(K+1)$, $S_1=2^\alpha\sigma_1^\alpha L^{\alpha/2}R^\alpha(K+1)^{(2-\alpha)/2}$, $S_2=2^{\alpha/2}\sigma_1^\alpha(LR)^\alpha$, and $\Sigma_{K+1}(\gamma)=S_0+S_1\gamma^{-\alpha/2}+S_2$ (equal to $\Sigma_{K+1}$ of~\eqref{eq:Sigma-def}) & Theorem~\ref{thm:fixed}\\
$\gamma_b^\lambda$, $\gamma_v^\lambda$ & explicit stepsize thresholds for the bias condition~\eqref{eq:fixed-bias} and the variance condition~\eqref{eq:fixed-var} & \eqref{eq:fixed-thresholds}\\
$\mathcal N_\lambda$, $\mathcal V_{\lambda,K}$, $\Xi$ & neighborhood $\max\{\frac{120bR\sigma_0^\alpha}{\lambda^{\alpha-1}},2(\frac{60\cdot2^\alpha bRL^{\alpha/2}\sigma_1^\alpha}{\lambda^{\alpha-1}})^{2/(2-\alpha)}\}$, the terms vanishing with $K$, and the constant of the rate for $\lambda=\lambda_0(K+1)^{1/\alpha}$ & \eqref{eq:neighborhood}, \eqref{eq:V-terms}, \eqref{eq:horizon-bound}\\
$\bar b_\lambda$ & bias scale $\frac{\sigma_0^\alpha}{\max\{\lambda,\sigma_0\}^{\alpha-1}}=\min\{\frac{\sigma_0^\alpha}{\lambda^{\alpha-1}},\sigma_0\}$ & \eqref{eq:bias-scale}\\
$(\kappa_1,\kappa_2)$, $c_\perp$, $\bar\sigma_1$ & $(1,1)$ under Assumption~\ref{as:moment}, $(\frac{\alpha}{\alpha-1},\frac{2}{2-\alpha})$ under Assumption~\ref{as:weak} with $\alpha<2$; $c_\perp=2^{\alpha-1}(\kappa_2+3\kappa_1)$; threshold $\bar\sigma_1=\frac{1}{2\cdot28^{1/\alpha}}$ on the multiplicative noise & Section~\ref{sec:planned}\\
$\bar A$, $V_\lambda$, $\varepsilon_N$, $\varepsilon_M$ & $\bar A=\sigma_0^\alpha+\sigma_1^\alpha(\sqrt2LR)^\alpha$, $V_\lambda=\min\{4\lambda^2,\kappa_2(2\lambda)^{2-\alpha}\bar A\}$ (variance scale of the clipped gradient), neighborhoods $\varepsilon_N=48\sqrt2c_\perp R\bar b_\lambda$ and $\varepsilon_M=384c_\perp\sigma_1^\alpha LR^2\min\{1,96c_\perp\sigma_1^\alpha\}$ & Theorem~\ref{thm:any-lambda}\\
$\varepsilon_\lambda$, $s_\lambda$ & lower bound $\frac{1}{48\cdot4^\alpha}\min\{\frac{R\sigma_0^\alpha}{\lambda^{\alpha-1}},R\sigma_0,LR^2\}$ and growth rate $\frac{1}{4^{\alpha+1}}\min\{\frac{\sigma_0^\alpha}{\lambda^{\alpha-1}},\lambda\}$ & Theorem~\ref{thm:lower}\\
\end{longtable}
\end{center}

\section{Extended Related Work}\label{app:related}

We expand the comparison from Section~\ref{sec:related}, distinguishing the noise assumptions, the dependence of its scale on the current point, the clipping schedule, and the type of guarantee. Our analysis concerns smooth convex problems and high-probability bounds. The non-convex, non-smooth, light-tailed, and in-expectation results below provide complementary context.

\paragraph{Light-tailed noise and boosting.} Logarithmic dependence on the inverse failure probability is classical under sub-Gaussian noise \citep{nemirovski2009robust} and extends to sub-Weibull noise, which has all moments \citep{li2022high}. For strongly convex problems, \citet{davis2021low} obtain high-confidence guarantees from low-confidence solvers using robust distance estimation. Their concrete stochastic-gradient instantiations assume light tails or bounded variance, whereas the abstract reduction applies whenever suitable solvers for the regularized subproblems are available. Below, we focus on direct analyses of clipping and on alternative methods for heavy-tailed noise.

\paragraph{Clipped methods for convex problems under bounded moments.} Under bounded variance, \citet{nazin2019algorithms} analyze truncated Mirror Descent, and \citet{gorbunov2020stochastic} analyze \clipsgd and \clipsstm for smooth convex and strongly convex problems. Related guarantees cover monotone and structured non-monotone variational inequalities \citep{gorbunov2022clipped}, non-smooth convex problems \citep{gorbunov2024near}, and subgradient schemes \citep{parletta2024high}. \citet{sadiev2023high} replace bounded variance with a bounded central moment of order $\alpha\in(1,2]$ for smooth convex, strongly convex, and monotone variational inequality problems. We recover their smooth convex rates when $\sigma_1=0$ (Table~\ref{tab:complexity}). The smooth convex analyses of \citet{gorbunov2020stochastic,sadiev2023high} tune the clipping level to the horizon and bound the noise scale by a constant on a ball around the solution.

\citet{nguyen2023improved} remove the horizon logarithm from the confidence factor for clipped Mirror Descent and its accelerated variant. \citet{liu2023stochastic} obtain a similar confidence dependence for non-smooth convex problems, together with in-expectation rates and adaptation to the initial distance. Our Theorem~\ref{thm:maximal} gives the corresponding reduction via a maximal Freedman inequality. \citet{liu2026clipped} refine the non-smooth analysis using a generalized effective dimension and derive high-probability and in-expectation lower bounds, which we compare with ours in Section~\ref{sec:results}.

Further extensions cover composite and distributed problems and variational inequalities \citep{gorbunov2024high}, distributed \clipsgd over networks \citep{yang2026high}, and, in expectation, distributed methods with local updates \citep{lee2025efficient}. Clipping has also been analyzed for AdaGrad-Norm and Adam-Norm \citep{chezhegov2025clipping}, convex $(L_0,L_1)$-smooth problems \citep{chezhegov2025convergence}, weakly convex problems \citep{zhu2025stochastic}, and zeroth-order non-smooth convex optimization \citep{kornilov2023accelerated}. Other settings include parameter-free online learning \citep{zhang2022parameter}, differential privacy with an arbitrary clipping level \citep{vatankhah2026differentially}, and streaming estimation \citep{tsai2022heavy,das2024near}. These guarantees use moment or variance bounds.

\paragraph{Concurrent work.} \citet{lobanov2026accelerated} (arXiv, 27 September 2026) analyze RCSAG -- a restarted clipped stochastic accelerated gradient method -- for differentiable convex and strongly convex objectives under $(H_0,H_1)$-smoothness, i.e., $f(y)\le f(x)+\abr{\nabla f(x),y-x}+(H_0+H_1(f(x)-f^*))\|y-x\|^2$ for $\|y-x\|\le r_H$ ($L$-smoothness corresponds to $H_1=0$ with $L=2H_0$), and a conditionally unbiased oracle satisfying, in our notation, $\E\|g(x,\xi)-\nabla f(x)\|^\alpha\le\sigma_0^\alpha+\sigma_1^\alpha\|\nabla f(x)\|^\alpha+\sigma_2^\alpha(f(x)-f^*)^{\alpha/2}$, $\alpha\in(1,2]$, at every query point (their Assumption~3). This is a moment condition: for $\sigma_2=0$, it coincides with Assumption~\ref{as:moment} imposed on $\R^d$, and otherwise it has an additional term decaying with the function gap.

In the convex case, the method runs a logarithmic number of phases starting from an upper bound on $f(x^0)-f^*$; each phase starts from a point whose gap is at most the current bound, runs an accelerated scheme of Nesterov type with one clipped sample per iteration, whose auxiliary sequence is projected onto a bounded convex set $Q\ni x^0,x^*$ of diameter $D$, with the clipping level and the number of iterations chosen according to this bound, $D$, and the noise and smoothness parameters, and halves the bound with high probability (their Theorem~8, which relies on a maximal Freedman inequality for a stopped potential).

The total number of oracle calls (their Theorem~9) is $\widetilde O\big(D\sqrt{H_0/\varepsilon}+\sqrt{H_1}DJ_H+(\sigma_0D/\varepsilon)^{\alpha/(\alpha-1)}+(\sigma_1D[\sqrt{H_0/\varepsilon}+\sqrt{H_1}J_H^{(\alpha-1)/\alpha}])^{\alpha/(\alpha-1)}+(\sigma_2D/\sqrt\varepsilon)^{\alpha/(\alpha-1)}+S\big)$, where $S=\lceil\log_2(\Delta_0/\varepsilon)\rceil$ for an initial bound $\Delta_0\ge f(x^0)-f^*$ and $0<\varepsilon\le\Delta_0$, and $J_H\le S$ counts the phases with $H_1 2^{-s}\Delta_0>H_0$; for $H_1=0$ and $\sigma_2=0$, this has the same polynomial dependence, up to constants and logarithmic factors, as the complexity of Theorem~\ref{thm:sstm-main} in Table~\ref{tab:complexity}, with $D=2R$ in place of $R$ (their Appendix~G.1 takes $Q=B_R(x^0)$ when $R$ is known), including the dependence on $\sigma_1$, whose optimality is open in both works.

Under strong convexity (their Theorem~10), the method uses shrinking balls centered at the restart points, and only the $\sigma_0$-term is polynomial in $1/\varepsilon$. For the noise satisfying $\E\|g(x,\xi)-\nabla f(x)\|^\alpha\lesssim(f(x)-f^*)^{\varrho}$, the stochastic cost is polynomial, logarithmic, or bounded in $1/\varepsilon$ depending on whether $\varrho$ is below, at, or above $\alpha$ for convex and $\alpha/2$ for strongly convex objectives (their Theorem~11). Their lower bounds for the convex case and the $\sigma_0$- and $\sigma_2$-terms (their Theorems~13 and~15) are based on a two-point construction with a rare informative response, similar to the one used in our Propositions~\ref{prop:lower-rate} and~\ref{prop:confidence-lower}, and have the factor $\ln(1/(2\rho))$ in terms of the failure probability $\rho$.

The two works are complementary: \citet{lobanov2026accelerated} consider generalized smoothness, strong convexity, and the noise decaying with the function gap, under moment bounds, and analyze a restarted and projected method whose parameters depend on the bound on $f(x^0)-f^*$, $D$, $\alpha$, $\sigma_i$, and $H_i$ (the horizon-driven variant from their Appendix~I removes the last three from the inner updates, and their Appendix~G constructs the initial bound from stochastic gradients); we consider \clipsgd and \clipsstm without restarts and projections, under the tail condition on a ball, and also cover the weak-$L^2$ endpoint with infinite variance, prescribed clipping levels, and the corresponding lower bounds.

\paragraph{Beyond moment bounds: symmetry, medians, and tail conditions.} Several approaches relax moment assumptions by imposing structure on the noise. Under symmetry, \citet{jakovetic2023nonlinear} analyze a general class of non-linear SGD methods (clipping, sign, normalization, and their component-wise versions) for strongly convex objectives under the noise whose components have a symmetric density and finite first moment only, and derive almost sure convergence and mean-squared-error rates; \citet{armacki2023high} derive high-probability bounds for non-convex and strongly convex problems in the same noise model, and \citet{armacki2025high} extend this framework to online learning, treating sign, quantization, and clipping non-linearities in a unified way, for symmetric noise without any moment assumption and for mixtures of symmetric and non-symmetric components, with convergence to a neighborhood in the latter case; \citet{armacki2026large} derive large-deviation and mean-squared-error bounds under a symmetric density that is positive near zero and possibly unbounded moments; \citet{yu2026smoothed} obtain mean-squared-error rates for smoothed clipping with error feedback for decentralized strongly convex problems; and \citet{kornilov2024median} use median clipping for zeroth-order convex problems and bandits under symmetric noise with bounded $\kappa$-th moment for some $\kappa>0$, in which case the noise may have no finite expectation.

\citet{puchkin2024breaking} stabilize the stochastic gradients via smoothed medians of means before clipping and derive, under a structural condition on the marginal densities of the noise (a symmetric heavy-tailed part plus an antisymmetric remainder with finite second moment), high-probability rates for (quasi-)strongly convex problems of the order $\widetilde O(K^{-1/2})$ in general, which improves upon the heavy-tailed rate $O(K^{-2(\alpha-1)/\alpha})$ for $\alpha<4/3$, and of the order $\widetilde O(K^{-1})$ for symmetric noise.

A tail condition also appears in \citet{shestakov2026from}: there, a regularly varying tail $\PP\{\|g(x^*,\xi)-\nabla f(x^*)\|\ge u\}\sim cu^{-\alpha}$ at the minimizer quantifies how heavy-tailed data affect the finite-sum minimization and the generalization of ClipSGD, with the bounds in expectation over the algorithm.

The weak-$L^p$ tail condition itself appears in \citet{elbakkali2026stochastic}, who impose it on the sample-gradient noise with a constant scale on $\R^d$ (their Assumption~1.3) for zeroth-order optimization of smooth non-convex objectives under global sample-wise $L$-smoothness (their Assumption~1.2): their method clips the scalar directional derivatives of a two-point estimator before aggregation, and they derive high-probability guarantees for finding $\varepsilon$-stationary points using $\widetilde O(d^{p/(2(p-1))}\varepsilon^{-(3p-2)/(p-1)})$ function evaluations, which gives $\widetilde O(d\varepsilon^{-4})$ at $p=2$ under a condition allowing infinite variance.

Their analysis of the base method controls the error of the estimator at every step separately, via a vector Bernstein inequality for the mini-batch and a union bound over the steps; therefore, their base method requires mini-batches of size $\widetilde\Theta\big((\sqrt d\,(\sqrt{L(f(x^0)-f^*)}+\sigma)/\varepsilon)^{p/(p-1)}\big)$ at each of its $O(L(f(x^0)-f^*)/\varepsilon^2)$ steps, where $\sigma$ is their noise scale (their Remark~2.4); at $p=2$, both the batch size and the number of steps grow as $\varepsilon^{-2}$.

Their momentum variant reduces the running batch size in proportion to one minus the momentum parameter, down to a single two-point estimate per step (two function evaluations with shared randomness), at the price of the stepsize reduced in the same proportion, a warm-start batch of the full size, and the same total number of function evaluations up to logarithmic factors and the additive warm-start cost (their Theorem~3.1 and Remark~3.2).

Our Assumption~\ref{as:weak} differs from the symmetry-based and minimizer-based conditions in that it is a weak-$L^\alpha$ bound on the whole ball $Q$, without any symmetry or density assumptions; it has the same tail form as the condition from \citet{elbakkali2026stochastic}, in the affine and local version, and we use it for first-order methods on convex problems with one oracle call per iteration, where the polynomial rates coincide with those for the moment case. We characterize the additional logarithmic factor at the weak-$L^2$ endpoint with infinite variance, analyze prescribed clipping levels, and derive lower bounds.

Methods without clipping can also converge under infinite variance \citep{wang2021convergence,fatkhullin2025can,liu2026expectation}. These results and the in-expectation analysis under H\"older smoothness by \citet{zaman2026convergence} use moment bounds. For convex composite problems, \citet{he2025accelerated} analyze an accelerated stochastic proximal method without clipping under a uniform bounded-moment assumption. Their heavy-tailed high-probability bounds have polynomial confidence dependence; logarithmic dependence is obtained under stronger sub-Weibull assumptions. Their bounded-domain and proximal setting differs from our local analysis of unconstrained first-order methods.

\paragraph{Non-convex problems and alternatives to clipping.} \citet{zhang2020adaptive} show that clipped SGD attains the optimal in-expectation rate $K^{-(\alpha-1)/(3\alpha-2)}$ for the gradient norm under a bounded $\alpha$-th moment assumption. \citet{cutkosky2021high} derive high-probability bounds by combining clipping, normalization, and momentum. \citet{liu2023breaking} improve the analysis and obtain acceleration beyond the general lower bound when the objective is an expectation of smooth sample functions.

\citet{hubler2025gradient} analyze normalized SGD with mini-batches or momentum without clipping: they obtain in-expectation convergence without knowledge of the problem parameters and high-probability guarantees for the mini-batched variant. \citet{liu2025nonconvex} show that normalized SGD with momentum and mini-batches achieves the optimal in-expectation rate when the tail index is known; when it is unknown, their rate is $K^{-(\alpha-1)/(2\alpha)}$. \citet{sun2025revisiting} clarify the roles of normalization and clipping, and \citet{kornilov2025sign} analyze sign methods with high probability under $(L_0,L_1)$-smoothness. \citet{li2022high} obtain high-probability guarantees for non-convex SGD under sub-Weibull noise. These results concern non-convex objectives and complement our smooth convex analysis.

\paragraph{Affine and multiplicative noise.} Affine variance \citep{bottou2018optimization}, expected smoothness \citep{gower2019sgd}, and the ABC condition \citep{khaled2023better} are standard assumptions in the in-expectation analysis of SGD. Adaptive methods under affine variance are studied by \citet{faw2022power,wang2023convergence} and, under $(L_0,L_1)$-smoothness, by \citet{faw2023beyond}. These works concern non-convex objectives and include high-probability guarantees with polynomial dependence on the inverse failure probability. \citet{attia2023sgd} obtain high-probability bounds under an almost sure affine noise bound, with a sub-Gaussian extension; these assumptions exclude infinite variance.

State-dependent finite-variance bounds also appear in streaming estimation. For mean estimation and linear regression, \citet{tsai2022heavy} analyze clipped SGD with a decaying stepsize and variance growing quadratically with the distance to the solution. Their clipping level scales as $\sqrt N$ and depends on the confidence level, and their squared-error bound has order $1/N$ with logarithmic confidence dependence. \citet{das2024near} derive high-probability guarantees for smooth strongly convex objectives under a covariance bound growing quadratically with the distance to the solution (their Assumption~2 and Theorem~2). Under infinite variance, \citet{wang2021convergence} analyze unmodified SGD with state-dependent noise and a finite $p$-th moment, $p\in[1,2)$, for a class of strongly convex objectives satisfying a $p$-positive-definiteness condition on the Hessian. They obtain $L^p$ rates and a generalized central limit theorem.

Multiplicative noise also explains how heavy-tailed iterates can arise from light-tailed data \citep{gurbuzbalaban2021heavy,hodgkinson2021multiplicative}; \citet{jiao2024emergence} study this phenomenon in homogenized SGD. Heavy-tailed gradient noise is documented and modeled by \citet{simsekli2019tail,nguyen2019first,zhang2020adaptive}, and \citet{barsbey2021heavy} relate heavy tails to compressibility and generalization. Under global strong growth or multiplicative finite-variance conditions, accelerated methods attain faster rates in expectation \citep{vaswani2019fast,gupta2024nesterov}; see the comparison in Appendix~\ref{app:sstm}.

Our assumptions combine a gradient-dependent scale with heavy tails. Replacing this scale by $\sigma=\sigma_0+\sigma_1\sup_Q\|\nabla f\|$ recovers a constant-scale condition but loses the separation between the additive and multiplicative contributions (Corollary~\ref{cor:uniform}).

\paragraph{Prescribed clipping levels and clipping bias.} Clipping with a fixed level may introduce a persistent bias. Under bounded variance, \citet{koloskova2023revisiting} derive tight in-expectation guarantees for convergence to a neighborhood and show that decreasing the stepsize alone cannot remove the bias. \citet{chen2020understanding} study this effect in private SGD through the symmetry of the gradient distribution. \citet{mai2021stability} establish stability estimates beyond Lipschitz continuity, including a fixed-threshold variant.

\citet{zhang2020why} explain how clipping can accelerate optimization under $(L_0,L_1)$-smoothness. For convex objectives in this class, \citet{gorbunov2025methods} analyze gradient descent with smoothed clipping and its stochastic finite-sum variant under a common-minimizer assumption, in expectation. \citet{gaash2025convergence} derive high-probability bounds for clipped SGD with double sampling under sub-Gaussian noise, and \citet{chezhegov2025convergence} obtain such bounds under a bounded $\alpha$-th moment without exponentially large factors. Under strong growth and interpolation, \citet{lobanov2026avoiding} analyze clipping and normalization under generalized smoothness and show that a suitable batch size can eliminate the error floor. Their heavy-tailed extension concerns normalized SGD and uses a moment bound.

Under standard smoothness and a bounded $\alpha$-th moment, \citet{vatankhah2026differentially} obtain high-probability guarantees for the best iterate with any clipping level. Our Theorem~\ref{thm:fixed} controls the averaged iterate under affine weak-$L^\alpha$ noise for $\lambda\ge2\sqrt2LR$ (and $\lambda\ge240b\sigma_1^2LR$ when $\alpha=2$ and $\sigma_1>0$). Theorem~\ref{thm:any-lambda} covers any $\lambda>0$ for the best iterate when $\sigma_1\le\bar\sigma_1$ (under Assumption~\ref{as:weak}, for $\alpha<2$). For additive noise, Theorem~\ref{thm:lower} gives the unavoidable neighborhood $\frac{1}{48\cdot4^\alpha}\min\{R\sigma_0^\alpha/\lambda^{\alpha-1},R\sigma_0,LR^2\}$ through a distribution satisfying the moment condition.

\paragraph{Lower bounds.} Stochastic convex lower bounds go back to \citet{nemirovski1983problem}. Under infinite variance, \citet{vural2022mirror} prove that $K^{-(\alpha-1)/\alpha}$ is optimal for non-smooth convex problems and that Mirror Descent attains it in expectation. \citet{zhang2020adaptive} derive the non-convex counterpart $K^{-(\alpha-1)/(3\alpha-2)}$. \citet{liu2026clipped} obtain high-probability and in-expectation lower bounds for non-smooth convex problems in terms of two noise scales.

Our Propositions~\ref{prop:lower-rate} and~\ref{prop:confidence-lower} concern smooth unconstrained problems, give explicit constants, and capture the failure probability. Proposition~\ref{prop:lorentz-opt-lower} establishes the necessity of the logarithmic factors on the Lorentz scale at $\alpha=2$, in the regimes stated there. We do not derive a lower bound for the multiplicative-noise term.

\paragraph{Robust estimation.} Another approach replaces each stochastic gradient with a robust estimate computed from a batch, using, for example, coordinate-wise Catoni-type estimators \citep{holland2019efficient} or median-of-means-type estimators \citep{prasad2020robust}. \citet{lugosi2019mean} survey mean estimation under heavy tails. These analyses use batches to estimate the mean gradient, whereas the methods studied here clip individual stochastic gradients.

\section{\texorpdfstring{Weak-$L^\alpha$ and Marcinkiewicz Spaces}{Weak-L-alpha and Marcinkiewicz Spaces}}\label{app:spaces}

In this section, we recall the basic facts about weak-$L^\alpha$ and Marcinkiewicz spaces that we use in the paper. Rearrangement-invariant function spaces are the natural language for power-law tails. Marcinkiewicz spaces are one of the basic classes of such spaces, along with Lorentz and Orlicz spaces; we refer to the monographs by \citet{krein2002interpolation}, \citet{bennett1988interpolation}, \citet{lindenstrauss2013classical}, and \citet{braverman1994independent} for the details. The connection between rearrangement-invariant spaces and probabilistic inequalities for independent random variables is surveyed by \citet{astashkin2010independent}, and martingale inequalities in such spaces are studied by \citet{kikuchi1999averaging,kikuchi2004new}. In this paper, we use the Marcinkiewicz space with fundamental function $\varphi(t)=t^{1/\alpha}$; as a set, it coincides with weak-$L^\alpha=L^{\alpha,\infty}$.

Let $Y$ be a real random variable on $(\Omega,\F,\PP)$ with tail function $d_Y(u):=\PP\{|Y|>u\}$, $u\ge0$. Its non-increasing rearrangement $Y^*:(0,1]\to[0,\infty]$ and maximal rearrangement are defined as follows:
\begin{align*}
    Y^*(t)&:=\inf\{u>0:\ d_Y(u)\le t\},\\
    Y^{**}(t)&:=\frac1t\int_0^tY^*(s)\dif s .
\end{align*}

\begin{definition}[Weak-$L^\alpha$ quasi-norm]\label{def:weak}
For $\alpha>0$, let
\begin{equation}\label{eq:weak-qn}
    \wnorm{Y}:=\sup_{u>0}u\,\PP\{|Y|>u\}^{1/\alpha}=\sup_{0<t\le1}t^{1/\alpha}Y^*(t),
\end{equation}
and let $L^{\alpha,\infty}$ be the set of all $Y$ with $\wnorm{Y}<\infty$.
\end{definition}

The functional~\eqref{eq:weak-qn} is positively homogeneous and rearrangement invariant, but, in general, it does not satisfy the triangle inequality, i.e., it is only a quasi-norm \citep{bennett1988interpolation}. Nevertheless, for $\alpha>1$, the same class of random variables admits an equivalent norm.

\begin{definition}[Marcinkiewicz norm]\label{def:marc}
For $\alpha>1$ and a random variable $Y$, let $\Mnorm{Y}:=\sup_{0<t\le1}t^{1/\alpha}Y^{**}(t)$. The space $M_\alpha$ of all $Y$ with $\Mnorm{Y}<\infty$ is the Marcinkiewicz space with fundamental function $\varphi(t)=t^{1/\alpha}$.
\end{definition}

\begin{proposition}[Equivalence of weak-$L^\alpha$ and $M_\alpha$]\label{prop:eq-norms}
For $1<\alpha<\infty$, the spaces $L^{\alpha,\infty}$ and $M_\alpha$ coincide as sets and
\begin{equation}\label{eq:norm-equivalence}
    \wnorm{Y}\le\Mnorm{Y}\le\frac{\alpha}{\alpha-1}\wnorm{Y}.
\end{equation}
\end{proposition}

\begin{proposition}[Fundamental function and monotonicity]\label{prop:fund}
For every event $A\in\F$, $\Mnorm{\ind_A}=\PP(A)^{1/\alpha}$. If $|Y|\le|Z|$ almost surely, then $\Mnorm{Y}\le\Mnorm{Z}$. The norm depends only on the distribution of $|Y|$, i.e., $M_\alpha$ is rearrangement invariant.
\end{proposition}

Therefore, the absence of the triangle inequality for the tail quasi-norm does not cause any difficulties in the clipping lemma (Lemma~\ref{lem:weak-clip}): it is convenient to prove the lemma using the quasi-norm, and one can always pass to the equivalent Marcinkiewicz norm. The next statement clarifies the exact relation between the moment and the weak-$L^\alpha$ settings.

\begin{proposition}[Embeddings]\label{prop:embeddings}
Let $\alpha>0$.
\begin{enumerate}
    \item If $Y\in L^\alpha$, then $\wnorm{Y}\le\left(\E|Y|^\alpha\right)^{1/\alpha}$.
    \item If $Y\in L^{\alpha,\infty}$ and $0<q<\alpha$, then $Y\in L^q$ and
    \begin{equation}\label{eq:weak-to-q}
        \E|Y|^q\le\frac{\alpha}{\alpha-q}\wnorm{Y}^q .
    \end{equation}
\end{enumerate}
In general, $Y\in L^{\alpha,\infty}$ does not imply $Y\in L^\alpha$.
\end{proposition}

\begin{remark}[Direction of the implications]\label{rem:direction}
For the same exponent $\alpha$, weak-$L^\alpha$ control does \emph{not} imply a bound on the $\alpha$-th moment. The correct statements are the following: strong $L^\alpha$ control implies weak-$L^\alpha$ control with the same constant, and weak-$L^\alpha$ control implies finiteness of every moment of order $q<\alpha$ with the explicit constant~\eqref{eq:weak-to-q}. In particular, Assumption~\ref{as:weak} implies Assumption~\ref{as:moment} with $\alpha$ replaced by any $q\in(1,\alpha)$ and $\sigma_j$ replaced by $(\alpha/(\alpha-q))^{1/q}\sigma_j$, $j\in\{0,1\}$ (it is sufficient to use~\eqref{eq:weak-to-q} and $(s+t)^{q/\alpha}\le s^{q/\alpha}+t^{q/\alpha}$). That is, it is the \emph{result} obtained under the weaker assumption that is stronger.
\end{remark}

\subsection{Proof of Proposition~\ref{prop:eq-norms}}

\begin{proof}
Since $Y^*$ is non-increasing, we have $Y^{**}(t)=\frac1t\int_0^tY^*(s)\dif s\ge Y^*(t)$, which gives the first inequality in~\eqref{eq:norm-equivalence}.

{Next, let $M:=\wnorm{Y}$. If $M=\infty$, the second inequality in~\eqref{eq:norm-equivalence} is immediate. Otherwise, $Y^*(s)\le Ms^{-1/\alpha}$ for every $s\in(0,1]$, and therefore}
\[
    Y^{**}(t)\le\frac Mt\int_0^ts^{-1/\alpha}\dif s=\frac{\alpha}{\alpha-1}Mt^{-1/\alpha}.
\]
Multiplying both sides by $t^{1/\alpha}$ and taking the supremum over $t$, we get the second inequality in~\eqref{eq:norm-equivalence}. Finally, the completeness of the space with norm $\Mnorm{\cdot}$ is a standard property of Marcinkiewicz spaces \citep{bennett1988interpolation,braverman1994independent}. Thus, for $\alpha>1$, the weak-$L^\alpha$ quasi-norm is equivalent to a Banach norm, which concludes the proof.
\end{proof}

\subsection{Proof of Proposition~\ref{prop:fund}}

\begin{proof}
Let $s:=\PP(A)$. For $Y=\ind_A$, we have
\[
    Y^*(t)=\begin{cases}1,&0<t<s,\\0,&s\le t\le1,\end{cases}\qquad
    Y^{**}(t)=\begin{cases}1,&0<t\le s,\\ s/t,&s<t\le1,\end{cases}
\]
since $d_Y(u)=s$ for $0<u<1$ and $d_Y(u)=0$ for $u\ge1$, meaning that $Y^*(t)=0$ as soon as $t\ge s$ (in particular, at $t=s$). Thus, $\sup_{0<t\le1}t^{1/\alpha}Y^{**}(t)=s^{1/\alpha}$. Next, if $|Y|\le|Z|$, then $Y^*\le Z^*$ and, as a consequence, $Y^{**}\le Z^{**}$, which gives the lattice monotonicity of the norm. Finally, if $|Y|$ and $|Z|$ are equidistributed, then $Y^*=Z^*$, and the norms coincide, which concludes the proof.
\end{proof}

\subsection{Proof of Proposition~\ref{prop:embeddings}}

\begin{proof}
The first claim follows from Markov's inequality: $\PP\{|Y|>u\}=\PP\{|Y|^\alpha>u^\alpha\}\le\E|Y|^\alpha/u^\alpha$. Thus, $u\,\PP\{|Y|>u\}^{1/\alpha}\le(\E|Y|^\alpha)^{1/\alpha}$, and it remains to take the supremum over $u>0$.

{For the second estimate, let $M:=\wnorm{Y}$. If $M=0$, then $Y=0$ almost surely and the claim is immediate; if $M=\infty$, the estimate with an infinite right-hand side is trivial. It remains to consider $0<M<\infty$. Then $\PP\{|Y|>u\}\le\min\{1,M^\alpha/u^\alpha\}$, and the layer-cake representation gives}
\begin{align*}
    \E|Y|^q&=q\int_0^\infty u^{q-1}\PP\{|Y|>u\}\dif u
    \le q\int_0^M u^{q-1}\dif u+qM^\alpha\int_M^\infty u^{q-\alpha-1}\dif u\\
    &=M^q+\frac{q}{\alpha-q}M^q=\frac{\alpha}{\alpha-q}M^q .
\end{align*}
{To prove the last claim for every $\alpha>0$, take a nonnegative random variable $Y$ with $\PP\{Y>u\}=\min\{1,u^{-\alpha}\}$ for $u>0$. Then $\wnorm{Y}=1$, whereas the layer-cake formula gives $\E Y^\alpha\ge\alpha\int_1^\infty u^{-1}\dif u=\infty$. This counterexample requires no centering or finite first moment and concludes the proof.}
\end{proof}

\subsection{The Pareto Example}\label{app:pareto-example}

The following example, which we summarize in Section~\ref{sec:pareto}, shows that the strict inclusion $L^\alpha\subsetneq L^{\alpha,\infty}$ is attained by the simplest heavy-tailed oracle. {For $L,\sigma>0$, we consider the one-dimensional quadratic problem $f(x)=\frac L2x^2$, $\nabla f(x)=Lx$, with}
\begin{equation}\label{eq:pareto-g}
    g(x,\xi)=Lx+\xi,\qquad \xi=\sigma\varepsilon\mathsf{R},
\end{equation}
where $\varepsilon$ is a Rademacher variable independent of $\mathsf{R}$, and $\mathsf{R}$ has the Pareto tail:
\begin{equation}\label{eq:pareto-tail}
    \PP\{\mathsf{R}>u\}=\begin{cases}1,&0<u<1,\\ u^{-\alpha},&u\ge1,\end{cases}\qquad 1<\alpha\le2.
\end{equation}

\begin{proposition}[Pareto noise at the critical index]\label{prop:pareto}
For the noise~\eqref{eq:pareto-g}--\eqref{eq:pareto-tail}, $\E\xi=0$, $\wnorm{\xi}=\sigma$ and $\E|\xi|^\alpha=\infty$. Moreover, $\E|\xi|^q=\sigma^q\frac{\alpha}{\alpha-q}$ for every $0<q<\alpha$.
\end{proposition}

That is, the moment-based theory with exponent $\alpha$ is not applicable to this example, while Assumption~\ref{as:weak} holds with $\sigma_0=\sigma$, $\sigma_1=0$, and the exact scale.

\subsection{Proof of Proposition~\ref{prop:pareto}}

\begin{proof}
By the independence of $\varepsilon$ and $\mathsf R$ and the symmetry of $\varepsilon$, we have $\E\xi=\sigma\,\E\varepsilon\,\E\mathsf R=0$, where $\E\mathsf R<\infty$ since $\alpha>1$. Next, since $|\xi|=\sigma\mathsf R$, for $u\ge\sigma$ we have $\PP\{|\xi|>u\}=\PP\{\mathsf R>u/\sigma\}=(\sigma/u)^\alpha$, and for $0<u<\sigma$ this probability equals $1$. Thus, $\sup_{u>0}u\,\PP\{|\xi|>u\}^{1/\alpha}=\sigma$.

Finally, the density of $\mathsf R$ on $[1,\infty)$ is $p(r)=\alpha r^{-\alpha-1}$. Therefore, $\E\mathsf R^\alpha=\alpha\int_1^\infty r^{-1}\dif r=\infty$, while for $0<q<\alpha$ we have $\E\mathsf R^q=\alpha\int_1^\infty r^{q-\alpha-1}\dif r=\frac{\alpha}{\alpha-q}$. Multiplying the latter by $\sigma^q$, we get the last claim, which concludes the proof.
\end{proof}

\section{Proofs of the Clipping Lemmas}\label{app:clipping}

\subsection{Separate Statements and Remarks}\label{app:clip-statements}

In this section, we provide the proofs of the clipping lemmas. Lemma~\ref{lem:clip} summarizes the following two statements, which we prove in Sections~\ref{app:clip-weak-proof} and~\ref{app:clip-moment-proof}; both of them also hold conditionally (see Corollary~\ref{cor:conditional}).

\begin{lemma}[Clipping under the weak-$L^\alpha$ condition]\label{lem:weak-clip}
{Let $1<\alpha\le2$ and $\lambda>0$, and let $X$ be a random vector in $\R^d$ with $\E X=x$ and $\PP\{\|X-x\|>u\}\le A(x)u^{-\alpha}$ for all $u>0$.} Let $\tX:=\clip(X,\lambda)$ and assume $\|x\|\le\lambda/2$. Then $\|\tX-\E\tX\|\le2\lambda$ almost surely, $\E\|\tX-\E\tX\|^2\le\E\|\tX-x\|^2$,
\begin{align}
    \|\E\tX-x\|&\le b_\alpha\,\frac{A(x)}{\lambda^{\alpha-1}},\label{eq:weak-clip-bias}\\
    \E\|\tX-x\|^2&\le v_\alpha\!\left(\frac{(3\lambda/2)^\alpha}{A(x)}\right)\lambda^{2-\alpha}A(x),\label{eq:weak-clip-second}
\end{align}
where the right-hand side of~\eqref{eq:weak-clip-second} is $0$ if $A(x)=0$, and $b_\alpha$, $v_\alpha(\cdot)$ are defined in~\eqref{eq:moment-clipping-constants}--\eqref{eq:BC-alpha}.
\end{lemma}

\begin{lemma}[Clipping under the moment condition]\label{lem:moment-clip}
{Let $\lambda>0$, let $X$ be a random vector in $\R^d$, and let $\tX:=\clip(X,\lambda)$.} Then $\|\tX-\E\tX\|\le2\lambda$ almost surely. If, in addition, $\E X=x$, $\E\|X-x\|^\alpha\le A(x)<\infty$ for some $\alpha\in(1,2]$, and $\|x\|\le\lambda/2$, then, with $b_\alpha^{\rm mom}$ and $v_\alpha^{\rm mom}$ from~\eqref{eq:moment-clipping-constants},
\begin{align}
    \|\E\tX-x\|&\le\frac{b_\alpha^{\rm mom}A(x)}{\lambda^{\alpha-1}},\label{eq:moment-bias}\\
    \E\|\tX-\E\tX\|^2\le\E\|\tX-x\|^2&\le v_\alpha^{\rm mom}\lambda^{2-\alpha}A(x).\label{eq:moment-second}
\end{align}
Here $1/2\le b_\alpha^{\rm mom}\le1$ and $1\le v_\alpha^{\rm mom}\le3/2$; in particular, $b_2^{\rm mom}=1/2$ and $v_2^{\rm mom}=1$.
\end{lemma}

\paragraph{The constant $v_\alpha(s)$.} For $\alpha<2$, we have $v_\alpha(s)\le v_\alpha:=v_\alpha(\infty)=\frac{2}{2-\alpha}(\frac32)^{2-\alpha}$, and, therefore, \eqref{eq:weak-clip-second} implies $\E\|\tX-x\|^2\le v_\alpha\lambda^{2-\alpha}A(x)$, where the constant does not depend on $\lambda$ but blows up as $\alpha\uparrow2$. The logarithmic branch in~\eqref{eq:BC-alpha} is finite for every $s<\infty$, including the case $\alpha=2$. It is the smaller one if and only if $\ln(1+s)<\frac{\alpha^2}{2(2-\alpha)}$, i.e., unless $\lambda$ exceeds the tail scale $A(x)^{1/\alpha}$ by a factor exponential in $(2-\alpha)^{-1}$. When the noise scale varies along the trajectory, the logarithms are combined using concavity: $s\mapsto s\,v_\alpha(C/s)$ is concave and non-decreasing on $[0,\infty)$ for every $C\ge0$ (see Lemma~\ref{lem:log-sum} in Appendix~\ref{app:clipping}), and, therefore, only the average noise scale enters the logarithm.

\begin{remark}[The endpoint $\alpha=2$]\label{rem:alpha-two}
{At $\alpha=2$ and $A(x)>0$, Lemma~\ref{lem:weak-clip} gives $\E\|\tX-x\|^2\le A(x)\big(1+\ln\big(1+\frac{9\lambda^2}{4A(x)}\big)\big)$; the bound is zero if $A(x)=0$. The logarithm is necessary even for centered scalar noise. Indeed, let $A>0$, $u_0:=A^{1/\alpha}$, and let $Y\ge0$ satisfy $\PP\{Y>u\}=\min\{1,A/u^\alpha\}$, with $1<\alpha\le2$. For an independent Rademacher sign $\epsilon$, the variable $X:=\epsilon Y$ has mean zero and weak-$L^\alpha$ quasi-norm $A^{1/\alpha}$. For $\lambda\ge u_0$, direct integration gives}
\[
    {\E|\clip(X,\lambda)|^2=\E\min\{Y^2,\lambda^2\}
    =\begin{cases}
    u_0^2+\dfrac{2A}{2-\alpha}\big(\lambda^{2-\alpha}-u_0^{2-\alpha}\big),&1<\alpha<2,\\[0.4em]
    A\big(1+2\ln(\lambda/\sqrt A)\big),&\alpha=2.
    \end{cases}}
\]
{At $\alpha=2$, the second expression grows logarithmically with $\lambda$, while $\E X^2=\infty$. For $\alpha<2$, dividing the first expression by $A\lambda^{2-\alpha}$ and taking the supremum over $\lambda\ge u_0$ gives $2/(2-\alpha)$. Thus, any second-moment bound of the form $vA\lambda^{2-\alpha}$ with $v$ independent of $\lambda$ must have $v\ge2/(2-\alpha)$. For each fixed ratio $\lambda/u_0$, the expression instead has the finite logarithmic limit as $\alpha\uparrow2$. In contrast, Lemma~\ref{lem:moment-clip} holds at $\alpha=2$ with $v_2^{\rm mom}=1$.}
\end{remark}

\begin{remark}[The endpoint $\alpha=1$]\label{rem:alpha-one}
{The factor $(\alpha-1)^{-1}$ in the bias bound is necessary within the class of centered noise with a finite mean. To see this, fix $1<\alpha\le2$ and let}
\[
    {\lambda_\alpha:=\left(\frac{2\alpha}{\alpha-1}\right)^{1/(\alpha-1)},\qquad
    p:=\lambda_\alpha^{-\alpha},\qquad d:=\frac{1}{2(1-p)}.}
\]
{Let $Y\ge\lambda_\alpha$ have the Pareto tail $\PP\{Y>u\}=(\lambda_\alpha/u)^\alpha$ for $u\ge\lambda_\alpha$, and let $X$ equal an independent copy of $Y$ with probability $p$ and equal $-d$ otherwise. Since $p\E Y=p\alpha\lambda_\alpha/(\alpha-1)=1/2=(1-p)d$, we have $\E X=0$. Also, $\lambda_\alpha\ge4$, $p\le1/4$, and $d<1$. For $u<d$ the tail bound $\PP\{|X|>u\}\le u^{-\alpha}$ is trivial; for $d\le u<\lambda_\alpha$ it follows from $p\le u^{-\alpha}$; and for $u\ge\lambda_\alpha$ it is an equality. Therefore, $\wnorm{X}=1$, while}
\[
    {\big|\E\clip(X,\lambda_\alpha)\big|
    =\frac12-p\lambda_\alpha
    =\frac1{2\alpha}
    =\frac{\lambda_\alpha^{1-\alpha}}{\alpha-1}.}
\]
{Consequently, any bias constant valid uniformly over centered weak-$L^\alpha$ noise is at least $1/(\alpha-1)$. At $\alpha=1$, weak-tail control alone need not even give a finite mean: for $A>0$ and $\PP\{Y>u\}=\min\{1,A/u\}$, we have $\E(Y-a)_+=\int_a^\infty\PP\{Y>u\}\dif u=\infty$ for every $a>0$.}
\end{remark}

\subsection{Proof of Lemma~\ref{lem:weak-clip}}\label{app:clip-weak-proof}

\begin{proof}
Let $Y:=\|X-x\|$ and $A:=A(x)$. By assumption,
\begin{equation}\label{eq:app-tail}
    \PP\{Y>u\}\le Au^{-\alpha},\qquad u>0.
\end{equation}

\emph{Step 1: boundedness.} By the definition of clipping, we have $\|\tX\|\le\lambda$ almost surely. By Jensen's inequality, $\|\E\tX\|\le\E\|\tX\|\le\lambda$. Thus, $\|\tX-\E\tX\|\le\|\tX\|+\|\E\tX\|\le2\lambda$.

\emph{Step 2: bias.} Since $\E X=x$, we have $\|\E\tX-x\|=\|\E(\tX-X)\|\le\E\|\tX-X\|$. For the clipping operator, we have $\|\tX-X\|=(\|X\|-\lambda)_+$. Since $\|x\|\le\lambda/2$, we have $\|X\|\le Y+\lambda/2$ and, therefore, $(\|X\|-\lambda)_+\le(Y-\lambda/2)_+$. Using the identity $\E(Y-a)_+=\int_a^\infty\PP\{Y>u\}\dif u$ with $a=\lambda/2$ and~\eqref{eq:app-tail}, we get
\[
    \|\E\tX-x\|\le A\int_{\lambda/2}^\infty u^{-\alpha}\dif u=\frac{2^{\alpha-1}}{\alpha-1}\,\frac{A}{\lambda^{\alpha-1}}.
\]

\emph{Step 3: second moment around $x$.} We start with showing that
\begin{equation}\label{eq:clip-contract}
    \|\tX-x\|\le\min\left\{\|X-x\|,\ \tfrac{3\lambda}{2}\right\}.
\end{equation}
The bound by $3\lambda/2$ follows from $\|\tX\|\le\lambda$ and $\|x\|\le\lambda/2$. For the first bound, if $\|X\|\le\lambda$, then $\tX=X$. Otherwise, let $s:=\|X\|>\lambda$ and $e:=X/s$; then $\tX=\lambda e$ and
\[
    \|X-x\|^2-\|\lambda e-x\|^2=(s-\lambda)\left(s+\lambda-2\abr{e,x}\right)\ge0,
\]
since $\abr{e,x}\le\|x\|\le\lambda/2$. Next, let $a:=3\lambda/2$. By~\eqref{eq:clip-contract} and the layer-cake representation, we have
\[
    \E\|\tX-x\|^2\le\E\min\{Y^2,a^2\}=2\int_0^{a}u\,\PP\{Y>u\}\dif u .
\]
If $A=0$, then $Y=0$ almost surely and both sides vanish; therefore, we assume that $A>0$. For $\alpha<2$, \eqref{eq:app-tail} gives $\E\min\{Y^2,a^2\}\le2A\int_0^au^{1-\alpha}\dif u=\frac{2}{2-\alpha}Aa^{2-\alpha}$. Next, we consider the logarithmic branch and let $u_0:=A^{1/\alpha}$. If $a\le u_0$, then $\E\min\{Y^2,a^2\}\le a^2\le Aa^{2-\alpha}$, since $a^\alpha\le A$. If $a>u_0$, we use $\PP\{Y>u\}\le1$ on $(0,u_0)$ and~\eqref{eq:app-tail} on $(u_0,a)$:
\begin{align*}
    \E\min\{Y^2,a^2\}&\le u_0^2+2A\int_{u_0}^au^{2-\alpha}\,\frac{\dif u}{u}\le Au_0^{2-\alpha}+2Aa^{2-\alpha}\ln\frac a{u_0}\\
    &\le Aa^{2-\alpha}\left(1+\frac2\alpha\ln\left(1+\frac{a^\alpha}{A}\right)\right),
\end{align*}
since $u_0^2=Au_0^{2-\alpha}\le Aa^{2-\alpha}$ and $\ln(a/u_0)=\frac1\alpha\ln(a^\alpha/A)$. In all cases, we have $\E\min\{Y^2,a^2\}\le Aa^{2-\alpha}\min\{\frac{2}{2-\alpha},1+\frac2\alpha\ln(1+\frac{a^\alpha}{A})\}=v_\alpha(a^\alpha/A)\lambda^{2-\alpha}A$, which is exactly~\eqref{eq:weak-clip-second}.

\emph{Step 4: centered second moment.} Finally, in the Euclidean space, we have $\E\|\tX-x\|^2=\E\|\tX-\E\tX\|^2+\|\E\tX-x\|^2$, implying $\E\|\tX-\E\tX\|^2\le\E\|\tX-x\|^2$, which concludes the proof.
\end{proof}

\begin{corollary}[Conditional versions]\label{cor:conditional}
Let $\G$ be a $\sigma$-algebra, let $x$ be $\G$-measurable, and let $\E[X\mid\G]=x$ and $\|x\|\le\lambda/2$. If, almost surely, $\PP\{\|X-x\|>u\mid\G\}\le A(x)u^{-\alpha}$ for all $u>0$ (resp.\ $\E[\|X-x\|^\alpha\mid\G]\le A(x)$), then the conclusions of Lemma~\ref{lem:weak-clip} (resp.\ Lemma~\ref{lem:moment-clip}) hold almost surely with $\E[\cdot]$ replaced by $\E[\cdot\mid\G]$; in particular, $\|\tX-\E[\tX\mid\G]\|\le2\lambda$ almost surely.
\end{corollary}

\begin{proof}
All steps of the proofs of Lemmas~\ref{lem:weak-clip} and~\ref{lem:moment-clip} remain valid after replacing expectations and probabilities with their conditional counterparts given $\G$. The quantity $A(x)$ is $\G$-measurable and, therefore, plays the role of a fixed coefficient in the conditional integral bounds. The conditional tail bounds can be assumed to hold for all $u>0$ simultaneously on one event of full probability: it is sufficient to take versions for rational $u$ and extend them by monotonicity and right-continuity. Then, the conditional Tonelli theorem gives, for deterministic $c>0$, $\E[(Y-c)_+\mid\G]=\int_c^\infty\PP\{Y>u\mid\G\}\dif u$ and $\E[\min\{Y^2,c^2\}\mid\G]=2\int_0^cu\,\PP\{Y>u\mid\G\}\dif u$, which we use with $c=\lambda/2$ and $c=3\lambda/2$. For Lemma~\ref{lem:moment-clip}, the two pointwise inequalities from its proof have deterministic coefficients and remain valid after taking conditional expectations, with $A(x)$ replaced by any finite $\G$-measurable upper bound on $\E[\|X-x\|^\alpha\mid\G]$. Finally, the conditional variance decomposition reads $\E[\|\tX-x\|^2\mid\G]=\E[\|\tX-\E[\tX\mid\G]\|^2\mid\G]+\|\E[\tX\mid\G]-x\|^2$, which concludes the proof.
\end{proof}

\subsection{Proof of Lemma~\ref{lem:moment-clip}}\label{app:clip-moment-proof}

\begin{proof}
The bound $\|\tX-\E\tX\|\le2\lambda$ is proved in the same way as in Step~1 of the proof of Lemma~\ref{lem:weak-clip}. Let $Y:=\|X-x\|$ and $c:=\lambda/2$. As in Step~2 of the proof of Lemma~\ref{lem:weak-clip}, we have $\|\E\tX-x\|\le\E(Y-c)_+$. For $y>c$, we have
\[
    \frac{\dif}{\dif y}\,\frac{y-c}{y^\alpha}=y^{-\alpha-1}\big(\alpha c-(\alpha-1)y\big),
\]
meaning that the maximum of $(y-c)_+/y^\alpha$ over $y>0$ is attained at $y=\alpha c/(\alpha-1)$ and equals $c^{1-\alpha}(\alpha-1)^{\alpha-1}/\alpha^\alpha$. Thus, $(Y-c)_+\le\frac{(\alpha-1)^{\alpha-1}}{\alpha^\alpha}c^{1-\alpha}Y^\alpha$ pointwise, and, taking expectations with $c^{1-\alpha}=2^{\alpha-1}\lambda^{1-\alpha}$, we get~\eqref{eq:moment-bias}.

Next, to bound the second moment, we use the contraction~\eqref{eq:clip-contract} (its proof uses only $\|x\|\le\lambda/2$), which gives $\|\tX-x\|^2\le\min\{Y^2,d^2\}$ with $d:=3\lambda/2$. For $y\ge0$, $d>0$, and $1<\alpha\le2$, we have
\[
    \min\{y^2,d^2\}\le d^{2-\alpha}y^\alpha:
\]
for $y\le d$, this follows from $y^{2-\alpha}\le d^{2-\alpha}$, and, for $y\ge d$, from $d^\alpha\le y^\alpha$. Taking expectations, we get $\E\|\tX-x\|^2\le(3/2)^{2-\alpha}\lambda^{2-\alpha}A(x)$, which is the last inequality in~\eqref{eq:moment-second}. The bound on $\E\|\tX-\E\tX\|^2$ follows from the variance decomposition, as in Step~4 of the proof of Lemma~\ref{lem:weak-clip}. Finally, $\frac{\dif}{\dif\alpha}\ln b_\alpha^{\rm mom}=\ln\frac{2(\alpha-1)}{\alpha}\le0$ on $(1,2]$ and $b_\alpha^{\rm mom}\to1$ as $\alpha\downarrow1$, and, therefore, $1/2=b_2^{\rm mom}\le b_\alpha^{\rm mom}\le1$; the bounds on $v_\alpha^{\rm mom}$ follow immediately from the definition, which concludes the proof.
\end{proof}

\Needspace{10\baselineskip}
\subsection{Summing the Variance Bounds}

\begin{lemma}[Summing the variance bounds]\label{lem:log-sum}
Let $\alpha\in(1,2]$ and $C\ge0$, and let $\phi(s):=s\,v_\alpha(C/s)$ for $s>0$ and $\phi(0):=0$. Then $\phi$ is concave and non-decreasing on $[0,\infty)$. Consequently, for all $T\ge1$, $\bar B>0$ and $A_0,\dots,A_{T-1}\ge0$ with $\sum_{t<T}A_t\le\bar B$,
\begin{equation}\label{eq:log-sum}
    \sum_{t=0}^{T-1}\phi(A_t)\le\bar B\,v_\alpha\!\left(\frac{TC}{\bar B}\right).
\end{equation}
\end{lemma}

\begin{proof}
For $s>0$, we have $\phi(s)=(\frac32)^{2-\alpha}\min\{\frac{2s}{2-\alpha},\,s+\frac2\alpha h(s)\}$ with $h(s):=s\ln(1+C/s)$, where the first term is absent for $\alpha=2$. If $C=0$, then $h\equiv0$. If $C>0$, then $h'(s)=\ln(1+\frac Cs)-\frac{C}{s+C}>0$, since $\ln(1+y)>\frac{y}{1+y}$ for $y>0$, and $h''(s)=-\frac{C^2}{s(s+C)^2}<0$. Moreover, $0\le h(s)=s\ln(s+C)-s\ln s\to0$ as $s\downarrow0$. {Thus, with $h(0):=0$, the function $h$ is continuous, concave, and non-decreasing on $[0,\infty)$. Both branches defining $\phi$ have these properties. Their minimum is concave because its hypograph is the intersection of their convex hypographs; it is also continuous and non-decreasing. The same conclusion holds when only one branch is present.} Next, let $B:=\sum_{t<T}A_t$; then $B\le\bar B$ by assumption ($\bar B$ is the upper bound from the statement, and $\phi$ is non-decreasing). By Jensen's inequality and monotonicity, we get $\sum_{t<T}\phi(A_t)\le T\phi(B/T)\le T\phi(\bar B/T)=\bar B\,v_\alpha(TC/\bar B)$, which concludes the proof.
\end{proof}

\section{Auxiliary Results}\label{app:aux}

\paragraph{Measurability conventions.} In the setting of Section~\ref{sec:setup}, the oracle $g(x,\xi)$ is a jointly measurable and finite vector field with $\E_\xi[g(x,\xi)]=\nabla f(x)$ for $x\in Q$; we do not assume any sample-wise gradient representation. The noise $g(x,\xi)-\nabla f(x)$ is controlled only at the points $x\in Q$; Lemma~\ref{lem:adaptive} (see Appendix~\ref{app:adaptive}) transfers these pointwise conditions to random query points in $Q$. Outside $Q$, $g$ is an arbitrary jointly measurable and finite extension, $f$ is not necessarily differentiable, and, for definiteness, we set $\nabla f(x):=0$ there; this extension is used only to define auxiliary random variables on the whole probability space, and it never enters the bounds, which are derived on the events where the relevant points belong to $Q$.

{We assume that $f:\R^d\to\R$ is Borel measurable; no differentiability, smoothness, or convexity is required outside $Q$. Consequently, the objective gaps used in the auxiliary random variables and stopping times are measurable on the whole probability space.}

\paragraph{The gradient bound.} Assumption~\ref{as:grad-bound} is stated separately to simplify the comparison with \citet{sadiev2023high}; however, it is implied by the other two local assumptions, since $Q$ is a ball centered at the minimizer.

\begin{lemma}[Gradient bound from local smoothness and convexity]\label{lem:grad-bound}
Let $f$ be differentiable on $Q=B_r(x^*)$, $r>0$, where $x^*$ is a minimizer of $f$ on $\R^d$, and let Assumptions~\ref{as:smooth} and~\ref{as:convex} hold on $Q$. Then $\|\nabla f(x)\|^2\le2L(f(x)-f^*)$ for all $x\in Q$.
\end{lemma}

\begin{proof}
\emph{Step 1: cocoercivity on $Q$.} We claim that $\abr{\nabla f(x)-\nabla f(z),x-z}\ge\frac1L\|\nabla f(x)-\nabla f(z)\|^2$ for all $x,z\in Q$. (a)~First, we assume that $f$ is twice continuously differentiable on an open convex set $U$ on which Assumptions~\ref{as:smooth} and~\ref{as:convex} hold, and let $x,z\in U$. Then $\nabla f(x)-\nabla f(z)=A(x-z)$ with $A:=\int_0^1\nabla^2f(z+t(x-z))\,\dif t$. Convexity gives $\nabla^2f\succeq0$ on $U$ and the Lipschitz bound gives $\nabla^2f\preceq LI$ on $U$; hence, $0\preceq A\preceq LI$. Since $A$ and $LI-A$ commute, we have $A(LI-A)\succeq0$, i.e., $A^2\preceq LA$, and, therefore, $\|A(x-z)\|^2\le L\abr{A(x-z),x-z}$, which is exactly the claim. (b)~In the general case, let $\rho_\delta$ be a smooth probability density supported in $B_\delta(0)$ and $f_\delta:=f*\rho_\delta$ on $U_\delta:=\{x:\|x-x^*\|<r-\delta\}$. For $x\in U_\delta$ and $\|u\|\le\delta$, we have $x-u\in Q$, meaning that $f_\delta$ is well-defined and smooth on $U_\delta$, convex on $U_\delta$ as an average of the convex functions $x\mapsto f(x-u)$, and $\nabla f_\delta=(\nabla f)*\rho_\delta$ is $L$-Lipschitz on $U_\delta$; moreover, $\|\nabla f_\delta(x)-\nabla f(x)\|\le L\delta$ for $x\in U_\delta$. Applying~(a) to $f_\delta$ on $U_\delta$ and letting $\delta\to0$, we get the claim for all $x,z$ from the open ball $\{\|x-x^*\|<r\}$, and then for all $x,z\in Q$ by the continuity of $\nabla f$ on $Q$.

\emph{Step 2.} Let $x\in Q$. Since $x^*$ is an interior point of $Q$ minimizing $f$ on $\R^d$, we have $\nabla f(x^*)=0$, and Step~1 with $z=x^*$ gives $\abr{\nabla f(x),x-x^*}\ge\frac1L\|\nabla f(x)\|^2$. Next, let $y:=x-\frac1L\nabla f(x)$. Then
\[
    \|y-x^*\|^2=\|x-x^*\|^2-\frac2L\abr{\nabla f(x),x-x^*}+\frac{1}{L^2}\|\nabla f(x)\|^2\le\|x-x^*\|^2-\frac{1}{L^2}\|\nabla f(x)\|^2\le r^2,
\]
meaning that $y\in Q$ and the segment $[x,y]$ belongs to $Q$. Along this segment, the gradient is $L$-Lipschitz; hence, $f(y)\le f(x)+\abr{\nabla f(x),y-x}+\frac L2\|y-x\|^2=f(x)-\frac{1}{2L}\|\nabla f(x)\|^2$, and $f(y)\ge f^*$ concludes the proof.
\end{proof}

Throughout the appendix, we use the following decomposition of the error of the clipped stochastic gradient (with the convention $\nabla f:=0$ outside $Q$, meaning that all quantities below are defined on the whole probability space):
\begin{equation}\label{eq:theta}
    \theta_k:=\tg^k-\nabla f(x^k)=\theta_k^u+\theta_k^b,\qquad
    \theta_k^u:=\tg^k-\E[\tg^k\mid\G_k],\qquad
    \theta_k^b:=\E[\tg^k\mid\G_k]-\nabla f(x^k),
\end{equation}
where $\E[\theta_k^u\mid\G_k]=0$ by construction.

\begin{lemma}[Basic inequality]\label{lem:basic}
{Let Assumptions~\ref{as:grad-bound} and~\ref{as:convex} hold on $Q$, let $0<\gamma\le1/(2L)$, and let $x^0,\dots,x^{t-1}\in Q$ for some integer $t\ge1$.} Then
\begin{align}
    \gamma\sum_{l=0}^{t-1}\left(f(x^l)-f^*\right)&\le\|x^0-x^*\|^2-\|x^t-x^*\|^2-2\gamma\sum_{l=0}^{t-1}\abr{x^l-x^*-\gamma\nabla f(x^l),\theta_l}+\gamma^2\sum_{l=0}^{t-1}\|\theta_l\|^2,\label{eq:basic-sum}
\end{align}
and the left-hand side of~\eqref{eq:basic-sum} is at least $\gamma t\left(f(\bar x^{t-1})-f^*\right)$, where $\bar x^{t-1}:=\frac1t\sum_{l=0}^{t-1}x^l$.
\end{lemma}

\begin{proof}
By the update rule, we have $x^{l+1}=x^l-\gamma\nabla f(x^l)-\gamma\theta_l$. Expanding the squared distance to $x^*$, we get
\begin{align*}
    \|x^{l+1}-x^*\|^2&=\|x^l-x^*\|^2-2\gamma\abr{x^l-x^*,\nabla f(x^l)}+\gamma^2\|\nabla f(x^l)\|^2\\
    &\quad-2\gamma\abr{x^l-x^*-\gamma\nabla f(x^l),\theta_l}+\gamma^2\|\theta_l\|^2 .
\end{align*}
By convexity, we have $\abr{x^l-x^*,\nabla f(x^l)}\ge f(x^l)-f^*$. By Assumption~\ref{as:grad-bound} and $\gamma\le1/(2L)$, we also have $\gamma^2\|\nabla f(x^l)\|^2\le2L\gamma^2(f(x^l)-f^*)\le\gamma(f(x^l)-f^*)$. Therefore,
\[
    \gamma\left(f(x^l)-f^*\right)\le\|x^l-x^*\|^2-\|x^{l+1}-x^*\|^2-2\gamma\abr{x^l-x^*-\gamma\nabla f(x^l),\theta_l}+\gamma^2\|\theta_l\|^2 .
\]
Summing up the above inequality over $l=0,\dots,t-1$, we get~\eqref{eq:basic-sum}. The last claim follows from Jensen's inequality for the convex function $f$ on the convex set $Q\ni\bar x^{t-1}$, which concludes the proof.
\end{proof}

\begin{lemma}[Freedman's inequality]\label{lem:freedman}
Let $(\F_i)_{i\ge0}$ be a filtration and let $X_i$ be $\F_i$-measurable with $\E[X_i\mid\F_{i-1}]=0$ and $|X_i|\le c$ almost surely. Let $v_i:=\E[X_i^2\mid\F_{i-1}]$. Then for all $b,G>0$,
\begin{equation}\label{eq:freedman}
    \PP\left\{\left|\sum_{i=1}^nX_i\right|>b\ \text{ and }\ \sum_{i=1}^nv_i\le G\right\}\le2\exp\left(-\frac{b^2}{2G+2cb/3}\right).
\end{equation}
The same bound holds for the larger event $\big\{\exists m\le n:\ |\sum_{i=1}^mX_i|>b\ \text{ and }\ \sum_{i=1}^mv_i\le G\big\}$ (maximal form).
\end{lemma}

\begin{proof}
This is the martingale version of Bernstein's inequality \citep{freedman1975tail,dzhaparidze2001bernstein}; for completeness, we provide the full proof. For $c=0$, the claim is trivial; therefore, we assume that $c>0$. Since $m!\ge2\cdot3^{m-2}$ for $m\ge2$, for $|x|\le c$ and $0<\mu<3/c$ we have
\[
    e^{\mu x}\le1+\mu x+\frac{\mu^2x^2}{2}\sum_{j=0}^\infty\left(\frac{\mu c}{3}\right)^j=1+\mu x+\frac{\mu^2x^2}{2(1-\mu c/3)} .
\]
Let $\psi(\mu):=\frac{\mu^2}{2(1-\mu c/3)}$. Taking conditional expectations and using $1+u\le e^u$, we get $\E[e^{\mu X_i}\mid\F_{i-1}]\le\exp(\psi(\mu)v_i)$. Thus, $M_m:=\exp\left(\mu\sum_{i\le m}X_i-\psi(\mu)\sum_{i\le m}v_i\right)$ is a non-negative supermartingale with $M_0=1$, implying $\E M_n\le1$. On the event $\{\sum_{i\le n}X_i>b,\ \sum_{i\le n}v_i\le G\}$, we have $M_n>e^{\mu b-\psi(\mu)G}$; by Markov's inequality, the probability of this event is at most $e^{-\mu b+\psi(\mu)G}$. The choice $\mu=\frac{b}{G+cb/3}<\frac3c$ gives $-\mu b+\psi(\mu)G=-\frac{b^2}{2G+2cb/3}$. Applying the same bound to $-X_i$ and taking the union of the two events, we get~\eqref{eq:freedman}. To prove the maximal form, note that on $\{\exists m\le n:\sum_{i\le m}X_i>b,\ \sum_{i\le m}v_i\le G\}$ we have $\max_{m\le n}M_m>e^{\mu b-\psi(\mu)G}$, and Ville's maximal inequality for the non-negative supermartingale $(M_m)$, $\PP\{\max_{m\le n}M_m\ge a\}\le\E M_0/a$, gives the same bound $e^{-\mu b+\psi(\mu)G}$ \citep[Theorem~1.6]{freedman1975tail}, which concludes the proof.
\end{proof}

\begin{lemma}[Jensen for powers]\label{lem:jensen-power}
For $0<r\le1$ and non-negative $a_1,\dots,a_n$, $\sum_{i=1}^na_i^r\le n^{1-r}\left(\sum_{i=1}^na_i\right)^r$.
\end{lemma}

\begin{proof}
The function $t\mapsto t^r$ is concave for $0<r\le1$, and, by Jensen's inequality, $\frac1n\sum_{i=1}^na_i^r\le\left(\frac1n\sum_{i=1}^na_i\right)^r$. Multiplying both sides by $n$, we get the claim.
\end{proof}

\section{Upper Bounds for Clip-SGD: Full Statements and Proofs}\label{app:main}

\subsection{Full Statements}\label{app:main-statements}

Throughout this section, $(b,v)$, $c$, and $\Lambda_K$ are defined as in Section~\ref{sec:results}. More precisely, under Assumption~\ref{as:moment}, since $b_\alpha^{\rm mom}/3<1\le v_\alpha^{\rm mom}$, the variance branch of $c$ is active and
\begin{equation}\label{eq:moment-rate-constant}
    c=c_\alpha^{\rm mom}:=180\left(\tfrac32\right)^{2-\alpha}\in[180,270),\qquad c_2^{\rm mom}=180 ;
\end{equation}
under Assumption~\ref{as:weak}, $c=c_\alpha(\Lambda_K)$ with $\Lambda_K:=\max\{80(K+1),\Lambda_{2\sqrt2LR}\}$, where, for $\lambda>0$,
\begin{equation}\label{eq:Lambda}
    \Lambda_\lambda:=\frac{(K+1)(3\lambda/2)^\alpha}{(K+1)\sigma_0^\alpha+2^{\alpha/2}\sigma_1^\alpha(LR)^\alpha}
\end{equation}
($\Lambda_\lambda:=0$ if $\sigma_0=\sigma_1=0$); in particular, $\Lambda_\lambda=(3\lambda/(2\sigma_0))^\alpha$ if {$\sigma_1=0$ and $\sigma_0>0$}.

\begin{theorem}[\clipsgd with the tuned level: full statement of Theorem~\ref{thm:sgd}]\label{thm:main}
{Let Assumptions~\ref{as:smooth}--\ref{as:convex} hold for $Q=B_{2R}(x^*)$, where $R\ge\|x^0-x^*\|>0$, and let either Assumption~\ref{as:moment} or Assumption~\ref{as:weak} hold with $\alpha\in(1,2]$; let $c$ be defined as in Section~\ref{sec:results} with $\Lambda=\Lambda_K$ under Assumption~\ref{as:weak}. Let $\beta\in(0,1)$, $K\in\{0,1,\dots\}$, $\ell:=\ln\frac{4(K+1)}{\beta}$, and let the stepsize satisfy $0<\gamma\le\min\{\gamma_L,\gamma_0,\gamma_1,\gamma_2\}$, where}
\begin{align*}
    \gamma_L&:=\frac{1}{80\sqrt2\,L\ell},\\
    \gamma_0&:=\frac{R}{\left(c\,40^{\alpha-1}(K+1)\,\ell^{\alpha-1}\right)^{1/\alpha}\sigma_0},\\
    \gamma_1&:=\frac{1}{\left(c\,2^\alpha 40^{\alpha-1}\ell^{\alpha-1}\right)^{2/\alpha}\sigma_1^2L\,(K+1)^{\frac{2-\alpha}{\alpha}}},\\
    \gamma_2&:=\frac{1}{\left(c\,2^{\alpha/2}40^{\alpha-1}\ell^{\alpha-1}\right)^{1/\alpha}\sigma_1L},
\end{align*}
with $\gamma_0:=+\infty$ if $\sigma_0=0$ and $\gamma_1=\gamma_2:=+\infty$ if $\sigma_1=0$. Then, with probability at least $1-\beta$, the iterates produced by Algorithm~\ref{alg:clip-sgd} with $\lambda_k\equiv\lambda:=\frac{R}{40\gamma\ell}$ satisfy $x^k\in B_{\sqrt2R}(x^*)$ for all $k=0,1,\dots,K+1$ and
\begin{equation}\label{eq:main-bound}
    f(\bar x^K)-f^*\le\frac{2R^2}{\gamma(K+1)}.
\end{equation}
\end{theorem}

\begin{corollary}[Rate and complexity]\label{cor:rate}
Let the conditions of Theorem~\ref{thm:main} hold and $\gamma=\min\{\gamma_L,\gamma_0,\gamma_1,\gamma_2\}$. Then, with probability at least $1-\beta$, we have $f(\bar x^K)-f^*\le\max\{\Delta_L,\Delta_0,\Delta_1\}$, where
\begin{align*}
    \Delta_L&:=\frac{160\sqrt2\,LR^2\ell}{K+1},\\
    \Delta_0&:=\frac{2\left(c\,40^{\alpha-1}\right)^{1/\alpha}\sigma_0R\,\ell^{(\alpha-1)/\alpha}}{(K+1)^{(\alpha-1)/\alpha}},\\
    \Delta_1&:=\frac{2\left(c\,2^\alpha40^{\alpha-1}\right)^{2/\alpha}\sigma_1^2LR^2\ell^{2(\alpha-1)/\alpha}}{(K+1)^{2(\alpha-1)/\alpha}}.
\end{align*}
That is, \clipsgd achieves $f(\bar x^K)-f^*\le\varepsilon$ with probability at least $1-\beta$ after
\begin{align*}
    K=\widetilde O\Bigg(\max\Bigg\{\frac{LR^2}{\varepsilon},\left(\frac{\sigma_0R}{\varepsilon}\right)^{\frac{\alpha}{\alpha-1}},\left(\frac{\sigma_1^2LR^2}{\varepsilon}\right)^{\frac{\alpha}{2(\alpha-1)}}\Bigg\}\Bigg)
\end{align*}
iterations, where $\widetilde O$ hides factors depending only on $\alpha$ and logarithmic factors in $1/\beta$ and the problem parameters (including $c=c_\alpha(\Lambda_K)$ under Assumption~\ref{as:weak}).
\end{corollary}

We notice that the threshold $\gamma_2$ from Theorem~\ref{thm:main} is redundant: with $\Delta_2:=2R^2/((K+1)\gamma_2)$, the proof of Corollary~\ref{cor:rate} gives $\Delta_2\le\max\{\Delta_L,\Delta_1\}/15$, i.e., $\gamma_2\ge15\min\{\gamma_L,\gamma_1\}$, and, therefore, it never determines the stepsize. The logarithmic dependence on the horizon in $\ell$ comes from the union bound over the $K+1$ steps in the proof: under Assumption~\ref{as:moment}, and under Assumption~\ref{as:weak} with $\alpha<2$ and $v=v_\alpha(\infty)$, Theorem~\ref{thm:main} holds with $\ell$ replaced by $\ln(4/\beta)$ (see Theorem~\ref{thm:maximal} in Appendix~\ref{app:maximal}, which relies on a maximal inequality, as in \citet{nguyen2023improved} for the moment case), and the resulting dependence on $\beta$ is optimal (see Proposition~\ref{prop:confidence-lower}). The same holds for Theorem~\ref{thm:sstm}; moreover, Theorem~\ref{thm:fixed} and Corollary~\ref{cor:fixed} admit the replacement under either assumption and for every $\alpha\in(1,2]$ (see item~1 of Theorem~\ref{thm:maximal}), since their clipping level is prescribed and $\Lambda_\lambda$ does not depend on $\ell$. With the horizon-dependent constant $v_\alpha(\Lambda_K)$, in particular, at $\alpha=2$, the logarithm remains; the endpoint is reached without it via the moment route (see Corollary~\ref{cor:lorentz-upper}).

\begin{corollary}[Pareto noise]\label{cor:pareto}
{Consider the problem from Appendix~\ref{app:pareto-example} with $\sigma>0$. Then Assumptions~\ref{as:smooth}--\ref{as:convex} and~\ref{as:weak} hold with $\sigma_0=\sigma$ and $\sigma_1=0$, although $\E|\xi|^\alpha=\infty$.} Therefore, by Corollary~\ref{cor:rate}, with probability at least $1-\beta$, we have
\[
    f(\bar x^K)-f^*=O_\alpha\left(\frac{LR^2\ell}{K}+\frac{c^{1/\alpha}\sigma R\,\ell^{(\alpha-1)/\alpha}}{K^{(\alpha-1)/\alpha}}\right),
\]
where $c=c_\alpha(\Lambda_K)\le c_\alpha(\infty)$ for $\alpha<2$, and $c=O(1+\ln(K+1)+\ln(1+LR/\sigma))$ for $\alpha=2$.
In contrast, Assumption~\ref{as:moment} holds for this problem only with $\alpha$ replaced by some $q\in(1,\alpha)$ and $\sigma_0=\left(\frac{\alpha}{\alpha-q}\right)^{1/q}\sigma$. For a fixed $q$, Theorem~\ref{thm:main} gives the slower rate $K^{-(q-1)/q}$; if $q$ approaches $\alpha$ with the horizon, one recovers the exponent $(\alpha-1)/\alpha$ only up to an additional factor of the order $(\ln K)^{1/\alpha}$, which comes from the constant $(\alpha/(\alpha-q))^{1/q}$.
\end{corollary}

For example, for $\alpha=3/2$, the weak-$L^{3/2}$ analysis gives the exponent $1/3$ with a constant independent of $K$, while the moment-based analysis gives only $1/4$ for the fixed $q=4/3$, and $1/3$ up to the factor $(\ln K)^{2/3}$ for a horizon-dependent $q$. For $\alpha=2$, the noise has infinite variance, and the weak-$L^2$ analysis gives the exponent $1/2$ of the finite-variance case up to a logarithmic factor, in the original weak-$L^2$ class and with explicit finite-horizon parameters. We also notice that, for $\alpha<2$ close to $2$, the logarithmic branch of $v_\alpha$ improves the constant for a fixed horizon as well: $c_\alpha(\Lambda_K)\le c_\alpha(\infty)$ for every $K$, with a strict gain when $\ln\Lambda_K\ll(2-\alpha)^{-1}$ (since $\Lambda_K\ge80$, the logarithmic branch can be active only for $\alpha>1.68$); since $c_\alpha(\Lambda_K)\to c_\alpha(\infty)$ as $K\to\infty$, this improves the bound for a fixed horizon but not uniformly in $K$.

\begin{theorem}[Prescribed clipping level: full statement]\label{thm:fixed}
{Let Assumptions~\ref{as:smooth}--\ref{as:convex} hold for $Q=B_{2R}(x^*)$, where $R\ge\|x^0-x^*\|>0$, and let either Assumption~\ref{as:moment} or Assumption~\ref{as:weak} hold with $\alpha\in(1,2]$. Let $\beta\in(0,1)$, $K\in\{0,1,\dots\}$, $\ell:=\ln\frac{4(K+1)}{\beta}$, let $\lambda\ge2\sqrt2LR$ be arbitrary, and let}
\begin{gather*}
    S_0:=\sigma_0^\alpha(K+1),\qquad S_2:=2^{\alpha/2}\sigma_1^\alpha(LR)^\alpha,\\
    S_1:=2^\alpha\sigma_1^\alpha L^{\alpha/2}R^\alpha(K+1)^{\frac{2-\alpha}{2}} .
\end{gather*}
{Let the stepsize $0<\gamma\le\frac{R}{40\lambda\ell}$ satisfy the following bias and variance conditions:}
\begin{align}
    60\,b\,\gamma\lambda^{1-\alpha}\Sigma_{K+1}(\gamma)&\le3R,\label{eq:fixed-bias}\\
    7200\,\ell\,v\,\gamma^2\lambda^{2-\alpha}\Sigma_{K+1}(\gamma)&\le3R^2,\label{eq:fixed-var}
\end{align}
where $\Sigma_{K+1}(\gamma):=S_0+S_1\gamma^{-\alpha/2}+S_2$. Then, with probability at least $1-\beta$, the iterates produced by Algorithm~\ref{alg:clip-sgd} with $\lambda_k\equiv\lambda$ satisfy $x^k\in B_{\sqrt2R}(x^*)$ for all $k=0,1,\dots,K+1$ and
\[
    f(\bar x^K)-f^*\le\frac{2R^2}{\gamma(K+1)}.
\]
\end{theorem}

Both left-hand sides are non-decreasing in $\gamma$. Requiring each of the three terms of $\Sigma_{K+1}(\gamma)$ to contribute at most one third, we get the explicit thresholds $\gamma\le\min\{\gamma_b^\lambda,\gamma_v^\lambda\}$ from~\eqref{eq:fixed-thresholds} below. For $\alpha=2$, the $S_1$-part of~\eqref{eq:fixed-bias} does not depend on $\gamma$ and requires $\lambda\ge240\,b\,\sigma_1^2LR$.

\begin{corollary}[Neighborhood and horizon-dependent levels: full statement of Theorem~\ref{thm:fixed-main}]\label{cor:fixed}
Let the conditions of Theorem~\ref{thm:fixed} hold, let either $\alpha<2$, or $\alpha=2$ and $\lambda\ge240\,b\,\sigma_1^2LR$ (in the latter case, the $S_1$-term of $\gamma_b^\lambda$ and the second term of $\mathcal N_\lambda$ are absent), and let $\gamma:=\min\{\frac{R}{40\lambda\ell},\gamma_b^\lambda,\gamma_v^\lambda\}$ with the thresholds from~\eqref{eq:fixed-thresholds}. Then, with probability at least $1-\beta$, we have
\begin{equation}\label{eq:fixed-bound}
    f(\bar x^K)-f^*\le\max\left\{\frac{80\lambda R\ell}{K+1},\,\mathcal N_\lambda,\,\mathcal V_{\lambda,K}\right\},
\end{equation}
where the neighborhood
\begin{equation}\label{eq:neighborhood}
    \mathcal N_\lambda:=\max\left\{\frac{120\,bR\sigma_0^\alpha}{\lambda^{\alpha-1}},2\left(\frac{60\cdot2^\alpha bRL^{\frac\alpha2}\sigma_1^\alpha}{\lambda^{\alpha-1}}\right)^{\frac2{2-\alpha}}\right\}
\end{equation}
does not depend on $K$, while
\begin{align}
    \mathcal V_{\lambda,K}:=\max\Bigg\{&120\sqrt2\,R\sqrt{\frac{v\lambda^{2-\alpha}\sigma_0^\alpha\ell}{K+1}},\ \frac{120\sqrt2\,b\,\sigma_1^\alpha LR^2}{K+1},\ \frac{120\sqrt2\,R\sqrt{v\lambda^{2-\alpha}S_2\,\ell}}{K+1},\notag\\
    &\frac{2R^2}{K+1}\left(\frac{7200\,\ell\,v\lambda^{2-\alpha}S_1}{R^2}\right)^{\frac{2}{4-\alpha}}\Bigg\}\label{eq:V-terms}
\end{align}
vanishes as $K\to\infty$ (its terms are of the orders $K^{-1/2}$, $K^{-1}$, $K^{-1}$, and $K^{-2/(4-\alpha)}$, up to logarithmic factors). Moreover, if $\sigma_1=0$ and $\lambda=\lambda_0(K+1)^{1/\alpha}$ with $\lambda_0>0$ and $K+1\ge(2\sqrt2LR/\lambda_0)^\alpha$, then, with probability at least $1-\beta$, we have
\begin{equation}\label{eq:horizon-bound}
    f(\bar x^K)-f^*\le\frac{\Xi R}{(K+1)^{(\alpha-1)/\alpha}},
\end{equation}
where $\Xi:=80\lambda_0\ell+120\,b\,\sigma_0^\alpha\lambda_0^{1-\alpha}+120\sqrt{2v\ell}\,\sigma_0^{\alpha/2}\lambda_0^{1-\alpha/2}$.
For $\sigma_1>0$, the same choice of $\lambda$ adds the terms decaying faster than $(K+1)^{-(\alpha-1)/\alpha}$.
\end{corollary}

\paragraph{Discussion.} If we write $\lambda_0=\rho\sigma_0$ in~\eqref{eq:horizon-bound}, then the price of misspecifying the noise level is the factor $80\rho\ell+120b\rho^{1-\alpha}+120\sqrt{2v\ell}\,\rho^{1-\alpha/2}$ in front of $\sigma_0R(K+1)^{-(\alpha-1)/\alpha}$, which grows linearly in the case of overestimation ($\rho\to\infty$) and as $\rho^{-(\alpha-1)}$ in the case of underestimation ($\rho\to0$). We emphasize that Theorems~\ref{thm:main} and~\ref{thm:fixed} both follow from one general bound (Proposition~\ref{prop:master} below), where $\gamma$ and $\lambda$ are coupled only through the bias condition, involving $b$, and the variance condition, involving $v$; under Assumption~\ref{as:moment}, Theorem~\ref{thm:main} is the special case $\lambda=R/(40\gamma\ell)$ of Theorem~\ref{thm:fixed}, while under Assumption~\ref{as:weak} the variance constants $v_\alpha(\Lambda_K)$ and $v_\alpha(\Lambda_\lambda)$ are not comparable in general, meaning that neither of the theorems contains the other one. For $\sigma_1=0$ and the moment condition, {the bound of \citet{vatankhah2026differentially}, for $\lambda>4LR$ and the best iterate, reduces to the same three terms as in Corollary~\ref{cor:fixed} after applying the initial-point bound $f(x^0)-f^*\le LR^2/2$ (see the discussion of Table~\ref{tab:prescribed}):} a $\lambda R\ell/K$ term, a neighborhood of the order $R\sigma_0^\alpha/\lambda^{\alpha-1}$, and a cross term of the order $R\lambda^{1-\alpha/2}\sigma_0^{\alpha/2}\sqrt{\ell/K}$; Corollary~\ref{cor:fixed} gives them for the averaged iterate, under affine weak-$L^\alpha$ noise and with explicit constants, and, at $\alpha=2$ under Assumption~\ref{as:weak}, only the cross term has the logarithm of $v_2$.

\subsection{A General Bound and the Proofs}\label{app:main-proofs}

Throughout this section, we use the notation $\ell:=\ln\frac{4(K+1)}{\beta}$, $R_k:=\|x^k-x^*\|$, and $\delta_k:=f(x^k)-f^*\ge0$, and the decomposition~\eqref{eq:theta}. Note that $\ell>\ln4>1$ since $4(K+1)/\beta>4$. For $\gamma,\lambda>0$ and $T\ge1$, we define
\begin{equation}\label{eq:Sigma-def}
    \Sigma_T:=\sigma_0^\alpha T+\sigma_1^\alpha2^\alpha L^{\alpha/2}R^\alpha\gamma^{-\alpha/2}T^{\frac{2-\alpha}{2}}+\sigma_1^\alpha2^{\alpha/2}L^\alpha R^\alpha,\qquad
    \Phi_T:=\gamma\lambda^{1-\alpha}\Sigma_T .
\end{equation}
Both quantities are non-decreasing in $T$, while $\Sigma_T/T$ is non-increasing in $T$. Next, let $(b,v):=(b_\alpha^{\rm mom},v_\alpha^{\rm mom})$ under Assumption~\ref{as:moment} and $(b,v):=(b_\alpha,v_\alpha(\Lambda_{\gamma,\lambda}))$ under Assumption~\ref{as:weak}, where
\begin{equation}\label{eq:Lambda-gl}
    \Lambda_{\gamma,\lambda}:=\frac{(K+1)(3\lambda/2)^\alpha}{\Sigma_{K+1}}\qquad(\Lambda_{\gamma,\lambda}:=0\ \text{ if }\sigma_0=\sigma_1=0),
\end{equation}
and let $c:=180\max\{v,b/3\}$, meaning that
\begin{equation}\label{eq:c-props}
    c\ge180v\qquad\text{and}\qquad c\ge60b .
\end{equation}
Since $\Sigma_{K+1}\ge(K+1)\sigma_0^\alpha+2^{\alpha/2}\sigma_1^\alpha(LR)^\alpha$, we have $\Lambda_{\gamma,\lambda}\le\Lambda_\lambda$, where $\Lambda_\lambda$ is defined in~\eqref{eq:Lambda}; hence, under Assumption~\ref{as:weak}, this $c$ is at most $c_\alpha(\Lambda_\lambda)$.

\begin{proposition}[A general high-probability bound]\label{prop:master}
{Let Assumptions~\ref{as:smooth}--\ref{as:convex} hold for $Q=B_{2R}(x^*)$, where $R\ge\|x^0-x^*\|>0$, and let either Assumption~\ref{as:moment} or Assumption~\ref{as:weak} hold with $\alpha\in(1,2]$. Let $\beta\in(0,1)$, $K\in\{0,1,\dots\}$, and let $\gamma,\lambda>0$ satisfy, with $(b,v)$ defined above,}
\begin{equation}\label{eq:master-conds}
    \lambda\ge2\sqrt2LR,\qquad\gamma\lambda\le\frac{R}{40\ell},\qquad b\,\Phi_{K+1}\le\frac{R}{20},\qquad v\gamma^2\lambda^{2-\alpha}\Sigma_{K+1}\le\frac{R^2}{2400\,\ell}.
\end{equation}
Then, with probability at least $1-\beta$, the iterates produced by Algorithm~\ref{alg:clip-sgd} with $\lambda_k\equiv\lambda$ satisfy $x^k\in B_{\sqrt2R}(x^*)$ for all $k=0,1,\dots,K+1$ and $f(\bar x^K)-f^*\le\frac{2R^2}{\gamma(K+1)}$.
\end{proposition}

The third condition in~\eqref{eq:master-conds} controls the accumulated clipping bias, and the fourth one controls the accumulated variance. Both of them follow from the single condition $\Phi_{K+1}\le3R/c$: by~\eqref{eq:c-props}, $b\Phi_{K+1}\le3bR/c\le R/20$, and $v\gamma^2\lambda^{2-\alpha}\Sigma_{K+1}=v\,\gamma\lambda\,\Phi_{K+1}\le\frac{R}{40\ell}\cdot\frac{3vR}{c}\le\frac{R^2}{2400\ell}$. This is exactly how we derive Theorem~\ref{thm:main} below. For a prescribed $\lambda$, keeping the two conditions separate is what makes the neighborhood in Theorem~\ref{thm:fixed} independent of the variance constant $v$ (and of its logarithm at $\alpha=2$).

\paragraph{Proof of Proposition~\ref{prop:master}.} Our proof is induction-based: by induction, we show that the iterates do not leave the ball $B_{\sqrt2R}(x^*)$ with high probability. We start with the consequences of~\eqref{eq:master-conds}. By the first two conditions and $\ell>1$, we have
\begin{equation}\label{eq:gamma-L}
    \gamma\le\frac{R}{40\ell\lambda}\le\frac{1}{80\sqrt2L}<\frac1{2L},\qquad L\gamma\le\sqrt2-1,\qquad \gamma\lambda\le\frac{R}{40\ell}\le\frac R{40},\qquad \sqrt2LR\le\frac\lambda2 .
\end{equation}

\paragraph{Events and induction hypothesis.} To formulate the statement rigorously, we introduce probability event $E_k$ for each $k=0,1,\dots,K+1$ as follows: $E_0:=\Omega$ and, for $k=1,\dots,K+1$,
\begin{equation}\label{eq:events}
    E_k:=\left\{\forall t=1,\dots,k:\ -2\gamma\sum_{l=0}^{t-1}\abr{x^l-x^*-\gamma\nabla f(x^l),\theta_l}+\gamma^2\sum_{l=0}^{t-1}\|\theta_l\|^2\le R^2\ \text{ and }\ R_t\le\sqrt2R\right\}.
\end{equation}
We will prove by induction that $\PP(E_k)\ge1-\frac{k\beta}{K+1}$ for all $k=0,1,\dots,K+1$. The base of the induction ($k=0$) is trivial. Next, we assume that the statement holds for $k=T-1$ with some $T\in\{1,\dots,K+1\}$, i.e., $\PP(E_{T-1})\ge1-\frac{(T-1)\beta}{K+1}$. Given this, we need to prove that $\PP(E_T)\ge1-\frac{T\beta}{K+1}$. Throughout the proof, empty sums are equal to zero.

\paragraph{Localization.} First, we show that $E_{T-1}$ implies that the iterates $x^0,\dots,x^T$ belong to $Q$. Indeed, on $E_{T-1}$, we have $R_t\le\sqrt2R$ for $t=1,\dots,T-1$ and $R_0\le R$, and, therefore, $x^t\in Q$ for $t\le T-1$. Moreover, by~\eqref{eq:gamma-L}, $R_T\le R_{T-1}+\gamma\|\tg^{T-1}\|\le\sqrt2R+\gamma\lambda\le2R$, i.e., $x^T\in Q$. Since $x^*$ is an interior minimizer of the differentiable convex function $f$ on $Q$, we have $\nabla f(x^*)=0$. Thus, $E_{T-1}$ implies that, for $t=0,\dots,T-1$,
\begin{equation}\label{eq:grad-small}
    \|\nabla f(x^t)\|\le LR_t\le\sqrt2LR\le\frac\lambda2,\qquad
    \|x^t-x^*-\gamma\nabla f(x^t)\|\le\sqrt2R(1+L\gamma)\le2R .
\end{equation}
Before we proceed, we introduce a new notation: for $t=0,\dots,T-1$, we define the $\G_t$-measurable vectors
\[
    \eta_t:=\begin{cases}x^t-x^*-\gamma\nabla f(x^t),&\text{if }\|x^t-x^*-\gamma\nabla f(x^t)\|\le2R,\\ 0,&\text{otherwise}.\end{cases}
\]
By construction, $\|\eta_t\|\le2R$ almost surely, and $E_{T-1}$ implies $\eta_t=x^t-x^*-\gamma\nabla f(x^t)$ for $t\le T-1$.

\paragraph{Basic inequality.} By~\eqref{eq:gamma-L}, Lemma~\ref{lem:basic} can be applied on $E_{T-1}$ with $t=T$. Since $f(\bar x^{T-1})\ge f^*$, it gives
\begin{equation}\label{eq:RT}
    R_T^2\le R^2-2\gamma\sum_{t=0}^{T-1}\abr{x^t-x^*-\gamma\nabla f(x^t),\theta_t}+\gamma^2\sum_{t=0}^{T-1}\|\theta_t\|^2 .
\end{equation}

\paragraph{Clipping bounds.} By construction, we have $\|\theta_t^u\|\le2\lambda$ almost surely. Next, on $E_{T-1}$, for $t\le T-1$, we have $\|\nabla f(x^t)\|\le\lambda/2$ by~\eqref{eq:grad-small}. Formally, let $\mathcal C_t:=\{x^t\in Q,\ \|\nabla f(x^t)\|\le\lambda/2\}\in\G_t$; then $E_{T-1}\subseteq\mathcal C_t$, and let $\check x^t:=x^t\ind_{\mathcal C_t}+x^*\ind_{\mathcal C_t^c}\in Q$. By Lemma~\ref{lem:adaptive}, Corollary~\ref{cor:conditional} can be applied with $\G=\G_t$, $x=\nabla f(\check x^t)$, and $X=g(\check x^t,\xi^t)$; here we use Lemma~\ref{lem:moment-clip} under Assumption~\ref{as:moment} and Lemma~\ref{lem:weak-clip} under Assumption~\ref{as:weak}. The clipped vector coincides with $\tg^t$ on $\mathcal C_t$; therefore, on $\mathcal C_t$, the conditional bias and second moment of $\tg^t$ satisfy the same bounds (see Step~2 of Appendix~\ref{app:sstm-proof} for the details). Here we always condition on $\G_t$ and never on $E_{T-1}$. Let $A_t:=\sigma_0^\alpha+\sigma_1^\alpha(2L\delta_t)^{\alpha/2}$, and let $\phi_\lambda(s):=v_\alpha^{\rm mom}s$ under Assumption~\ref{as:moment} and $\phi_\lambda(s):=s\,v_\alpha((3\lambda/2)^\alpha/s)$ under Assumption~\ref{as:weak} (with $\phi_\lambda(0):=0$). By Assumption~\ref{as:grad-bound}, we have $\|\nabla f(x^t)\|^\alpha\le(2L\delta_t)^{\alpha/2}$, meaning that the noise scale at $x^t$ is at most $A_t$. Since the right-hand sides of the clipping lemmas are non-decreasing in $A(x)$ (see Lemma~\ref{lem:log-sum} for $\phi_\lambda$), we get
\begin{equation}\label{eq:clip-bounds}
    \|\theta_t^b\|\le\frac{b\,A_t}{\lambda^{\alpha-1}},\qquad
    \E\left[\|\theta_t^u\|^2\mid\G_t\right]\le\lambda^{2-\alpha}\phi_\lambda(A_t).
\end{equation}

\paragraph{Upper bound for $\sum_t\delta_t^{\alpha/2}$.} {On $E_{T-1}$, $L$-smoothness on $Q$ and $\nabla f(x^*)=0$ give $\delta_{T-1}\le\frac L2R_{T-1}^2\le LR^2$. If $T=1$, the sum over $t=0,\dots,T-2$ is empty. If $T\ge2$, Lemma~\ref{lem:basic} with $t=T-1$ and the first condition in~\eqref{eq:events} give $\gamma\sum_{t=0}^{T-2}\delta_t\le R_0^2+R^2\le2R^2$. Applying Lemma~\ref{lem:jensen-power} with $r=\alpha/2$ and $n=T-1$, we obtain}
\[
    \sum_{t=0}^{T-2}\delta_t^{\alpha/2}\le(T-1)^{\frac{2-\alpha}{2}}\left(\sum_{t=0}^{T-2}\delta_t\right)^{\alpha/2}\le T^{\frac{2-\alpha}{2}}\left(\frac{2R^2}{\gamma}\right)^{\alpha/2}.
\]
{The final upper bound also holds for $T=1$, since the empty sum is zero. Therefore, on $E_{T-1}$, with $\Sigma_T$ defined in~\eqref{eq:Sigma-def}, we have}
\begin{equation}\label{eq:Sigma}
    \sum_{t=0}^{T-1}A_t=\sum_{t=0}^{T-1}\left(\sigma_0^\alpha+\sigma_1^\alpha(2L\delta_t)^{\alpha/2}\right)
    \le\Sigma_T .
\end{equation}
Moreover, $E_{T-1}$ implies
\begin{equation}\label{eq:var-sum}
    \sum_{t=0}^{T-1}\E\left[\|\theta_t^u\|^2\mid\G_t\right]\le v\lambda^{2-\alpha}\Sigma_T .
\end{equation}
Indeed, under Assumption~\ref{as:moment}, this follows from~\eqref{eq:clip-bounds} and~\eqref{eq:Sigma}. Under Assumption~\ref{as:weak}, if $\Sigma_T=0$, then all $A_t=0$ and both sides vanish. Otherwise, Lemma~\ref{lem:log-sum} with $C=(3\lambda/2)^\alpha$ and $\bar B=\Sigma_T$, the inequality $T/\Sigma_T\le(K+1)/\Sigma_{K+1}$, and the monotonicity of $v_\alpha$ give $\sum_{t<T}\phi_\lambda(A_t)\le\Sigma_Tv_\alpha(TC/\Sigma_T)\le\Sigma_Tv_\alpha(\Lambda_{\gamma,\lambda})=v\Sigma_T$.

\paragraph{The key inequalities.} Since $\Sigma_T$ and $\Phi_T$ are non-decreasing in $T$ and $T\le K+1$, the last two conditions in~\eqref{eq:master-conds} imply
\begin{equation}\label{eq:master}
    b\,\Phi_T=b\,\gamma\lambda^{1-\alpha}\Sigma_T\le\frac{R}{20}
\end{equation}
and
\begin{equation}\label{eq:master2}
    v\gamma^2\lambda^{2-\alpha}\Sigma_T\le\frac{R^2}{2400\,\ell}.
\end{equation}

\paragraph{Decomposition into five terms.} Using $\theta_t=\theta_t^u+\theta_t^b$, $\|\theta_t\|^2\le2\|\theta_t^u\|^2+2\|\theta_t^b\|^2$, and $\eta_t=x^t-x^*-\gamma\nabla f(x^t)$ on $E_{T-1}$, we derive from~\eqref{eq:RT} that $E_{T-1}$ implies
\begin{equation}\label{eq:five-terms}
    -2\gamma\sum_{t=0}^{T-1}\abr{x^t-x^*-\gamma\nabla f(x^t),\theta_t}+\gamma^2\sum_{t=0}^{T-1}\|\theta_t\|^2\le\circledOne+\circledTwo+\circledThree+\circledFour+\circledFive,
    \qquad R_T^2\le R^2+\circledOne+\circledTwo+\circledThree+\circledFour+\circledFive,
\end{equation}
where
\begin{gather*}
    \circledOne:=-2\gamma\sum_{t=0}^{T-1}\abr{\theta_t^u,\eta_t},\qquad
    \circledTwo:=-2\gamma\sum_{t=0}^{T-1}\abr{\theta_t^b,\eta_t},\qquad
    \circledThree:=2\gamma^2\sum_{t=0}^{T-1}\left(\|\theta_t^u\|^2-\E\left[\|\theta_t^u\|^2\mid\G_t\right]\right),\\
    \circledFour:=2\gamma^2\sum_{t=0}^{T-1}\E\left[\|\theta_t^u\|^2\mid\G_t\right],\qquad
    \circledFive:=2\gamma^2\sum_{t=0}^{T-1}\|\theta_t^b\|^2 .
\end{gather*}
That is, it remains to derive the upper bounds for $\circledOne$, $\circledTwo$, $\circledThree$, $\circledFour$, $\circledFive$. More precisely, we show that each of these terms is at most $R^2/5$ in absolute value on $E_{T-1}$ intersected with two events of probability at least $1-\frac{\beta}{2(K+1)}$ each.

\paragraph{Upper bound for $\circledTwo$.} On $E_{T-1}$, by~\eqref{eq:clip-bounds}, \eqref{eq:Sigma}, and~\eqref{eq:master}, we have
\[
    |\circledTwo|\le2\gamma\sum_{t=0}^{T-1}\|\theta_t^b\|\|\eta_t\|\le4\gamma R\sum_{t=0}^{T-1}\|\theta_t^b\|\le4bR\,\gamma\lambda^{1-\alpha}\Sigma_T=4R\,b\,\Phi_T\le\frac{R^2}{5}.
\]

\paragraph{Upper bound for $\circledFour$.} On $E_{T-1}$, by~\eqref{eq:var-sum} and~\eqref{eq:master2}, we have
\[
    \circledFour\le2v\gamma^2\lambda^{2-\alpha}\Sigma_T\le\frac{R^2}{1200\,\ell}\le\frac{R^2}{5}.
\]

\paragraph{Upper bound for $\circledFive$.} On $E_{T-1}$, since the summands are non-negative, by~\eqref{eq:clip-bounds} and~\eqref{eq:master}, we have
\[
    \circledFive\le2\left(\gamma\sum_{t=0}^{T-1}\|\theta_t^b\|\right)^2\le2\left(b\,\Phi_T\right)^2\le\frac{R^2}{200}\le\frac{R^2}{5}.
\]

\paragraph{Upper bound for $\circledOne$.} To estimate this sum, we will use Bernstein's inequality in the martingale form (Lemma~\ref{lem:freedman}). Let $X_t:=-2\gamma\abr{\theta_t^u,\eta_t}$. The summands have conditional expectations equal to zero: since $\eta_t$ is $\G_t$-measurable and $\E[\theta_t^u\mid\G_t]=0$, the sequence $\{X_t\}_{t=0}^{T-1}$ is a martingale difference sequence w.r.t.\ $\{\G_{t+1}\}_{t=0}^{T-1}$. Next, the summands are bounded:
\[
    |X_t|\le2\gamma\|\theta_t^u\|\|\eta_t\|\le8\gamma\lambda R\le\frac{R^2}{5\ell}=:c_1 .
\]
Finally, conditional variances $\sigma_t^2:=\E[X_t^2\mid\G_t]\le4\gamma^2\|\eta_t\|^2\E[\|\theta_t^u\|^2\mid\G_t]\le16\gamma^2R^2\E[\|\theta_t^u\|^2\mid\G_t]$ are bounded on $E_{T-1}$: by~\eqref{eq:var-sum} and~\eqref{eq:master2},
\begin{equation}\label{eq:var-T1}
    \sum_{t=0}^{T-1}\sigma_t^2\le16R^2\,v\gamma^2\lambda^{2-\alpha}\Sigma_T\le\frac{16R^4}{2400\ell}=\frac{R^4}{150\ell}=:G .
\end{equation}
Applying Lemma~\ref{lem:freedman} with $b_1=R^2/5$, $c_1$, and $G$, we get
\[
    \PP\left\{|\circledOne|>\frac{R^2}{5}\ \text{and}\ \sum_{t=0}^{T-1}\sigma_t^2\le G\right\}\le2\exp\left(-\frac{b_1^2}{2G+2c_1b_1/3}\right)=2\exp(-\ell)=\frac{\beta}{2(K+1)}.
\]
The above is equivalent to $\PP(E_{\circledOne})\ge1-\frac{\beta}{2(K+1)}$ for the event
\[
    E_{\circledOne}:=\left\{\text{either }\sum_{t=0}^{T-1}\sigma_t^2>G\ \text{ or }\ |\circledOne|\le\frac{R^2}{5}\right\}.
\]
Moreover, by~\eqref{eq:var-T1}, $E_{T-1}\cap E_{\circledOne}$ implies $|\circledOne|\le R^2/5$.

\paragraph{Upper bound for $\circledThree$.} Let $Y_t:=2\gamma^2\left(\|\theta_t^u\|^2-\E[\|\theta_t^u\|^2\mid\G_t]\right)$. Then $\{Y_t\}_{t=0}^{T-1}$ is a martingale difference sequence w.r.t.\ $\{\G_{t+1}\}_{t=0}^{T-1}$, and, since $\|\theta_t^u\|^2\le4\lambda^2$, the summands are bounded:
\[
    |Y_t|\le2\gamma^2\left(\|\theta_t^u\|^2+\E[\|\theta_t^u\|^2\mid\G_t]\right)\le16\gamma^2\lambda^2\le\frac{R^2}{100\ell^2}\le\frac{R^2}{5\ell}=c_1 .
\]
Next, for a random variable $0\le Z\le B$, we have $\E(Z-\E Z)^2\le\E Z^2\le B\,\E Z$. With $Z=\|\theta_t^u\|^2$ and $B=4\lambda^2$, the conditional variances satisfy $\widetilde\sigma_t^2:=\E[Y_t^2\mid\G_t]\le16\gamma^4\lambda^2\E[\|\theta_t^u\|^2\mid\G_t]$. On $E_{T-1}$, by~\eqref{eq:var-T1}, we have
\[
    \sum_{t=0}^{T-1}\widetilde\sigma_t^2\le16\gamma^2\lambda^2\cdot\gamma^2\sum_{t=0}^{T-1}\E[\|\theta_t^u\|^2\mid\G_t]\le\frac{R^2}{100\ell^2}\cdot\frac{1}{16R^2}\cdot\frac{R^4}{150\ell}\le G .
\]
Similarly to the case of $\circledOne$, Lemma~\ref{lem:freedman} implies that the event
\[
    E_{\circledThree}:=\left\{\text{either }\sum_{t=0}^{T-1}\widetilde\sigma_t^2>G\ \text{ or }\ |\circledThree|\le\frac{R^2}{5}\right\}
\]
satisfies $\PP(E_{\circledThree})\ge1-\frac{\beta}{2(K+1)}$, and $E_{T-1}\cap E_{\circledThree}$ implies $|\circledThree|\le R^2/5$.

\paragraph{Completing the induction.} That is, we derived the upper bounds for $\circledOne$--$\circledFive$. More precisely, $E_{T-1}\cap E_{\circledOne}\cap E_{\circledThree}$ implies that all five bounds hold. By~\eqref{eq:five-terms}, the left-hand side of the first condition in~\eqref{eq:events} with $t=T$ is at most $R^2$, and $R_T^2\le R^2+5\cdot\frac{R^2}{5}=2R^2$. Together with the conditions for $t\le T-1$ contained in $E_{T-1}$, this means that $E_{T-1}\cap E_{\circledOne}\cap E_{\circledThree}\subseteq E_T$, and the union bound gives
\[
    \PP(E_T)\ge1-\PP(E_{T-1}^c)-\PP(E_{\circledOne}^c)-\PP(E_{\circledThree}^c)\ge1-\frac{(T-1)\beta}{K+1}-\frac{\beta}{K+1}=1-\frac{T\beta}{K+1}.
\]
This finishes the induction. In particular, $\PP(E_{K+1})\ge1-\beta$.

\paragraph{Conclusion of the proof of Proposition~\ref{prop:master}.} Finally, $E_{K+1}$ implies that all iterates $x^0,\dots,x^{K+1}$ belong to $B_{\sqrt2R}(x^*)$. Lemma~\ref{lem:basic} with $t=K+1$ and the first condition in~\eqref{eq:events} with $t=K+1$ give $\gamma(K+1)\left(f(\bar x^K)-f^*\right)\le R_0^2+R^2\le2R^2$, which concludes the proof.
\hfill$\square$

\paragraph{Proof of Theorem~\ref{thm:main}.} Let $c_K$ denote the constant $c$ from Theorem~\ref{thm:main}, i.e., $c_K=c_\alpha^{\rm mom}$ under Assumption~\ref{as:moment} and $c_K=c_\alpha(\Lambda_K)$ under Assumption~\ref{as:weak}, and let $\gamma_*:=\min\{\gamma_L,\gamma_0,\gamma_1,\gamma_2\}$, computed with $c_K$. Let $\gamma\le\gamma_*$ and $\lambda=\frac{R}{40\gamma\ell}$. Then $\gamma\lambda=\frac{R}{40\ell}$, and $\lambda\ge\frac{80\sqrt2L\ell R}{40\ell}=2\sqrt2LR$ since $\gamma\le\gamma_L$. It remains to verify that $\Phi_{K+1}\le3R/c$ for the constant $c=180\max\{v,b/3\}$ from Appendix~\ref{app:main}, which implies the last two conditions in~\eqref{eq:master-conds} (see the discussion after Proposition~\ref{prop:master}).

\emph{Step 1: the bound with $c_K$.} Raising $\gamma\le\gamma_0$, $\gamma\le\gamma_1$ and $\gamma\le\gamma_2$ to the powers $\alpha$, $\alpha/2$ and $\alpha$, respectively, we get
\begin{equation}\label{eq:P123}
\begin{aligned}
    &\gamma^\alpha(40\ell)^{\alpha-1}\sigma_0^\alpha (K+1)\le\frac{R^\alpha}{c_K},\qquad
    \gamma^{\alpha/2}(40\ell)^{\alpha-1}\sigma_1^\alpha2^\alpha L^{\alpha/2}(K+1)^{\frac{2-\alpha}{2}}\le\frac1{c_K},\\
    &\gamma^\alpha(40\ell)^{\alpha-1}\sigma_1^\alpha2^{\alpha/2}L^\alpha\le\frac1{c_K} .
\end{aligned}
\end{equation}
Using $\lambda^{1-\alpha}=(40\gamma\ell/R)^{\alpha-1}$ and~\eqref{eq:P123}, we get
\begin{align*}
    \Phi_{K+1}
    &=R^{1-\alpha}\gamma^\alpha(40\ell)^{\alpha-1}\sigma_0^\alpha(K+1)
    +R\,\gamma^{\alpha/2}(40\ell)^{\alpha-1}\sigma_1^\alpha2^\alpha L^{\alpha/2}(K+1)^{\frac{2-\alpha}{2}}\\
    &\quad+R\,\gamma^\alpha(40\ell)^{\alpha-1}\sigma_1^\alpha2^{\alpha/2}L^\alpha
    \le\frac{3R}{c_K}.
\end{align*}
Under Assumption~\ref{as:moment}, we have $c=c_K$; thus, \eqref{eq:master-conds} holds, and Proposition~\ref{prop:master} gives the result, including~\eqref{eq:main-bound}.

\emph{Step 2: Assumption~\ref{as:weak} and $\gamma=\gamma_*$.} Let $\lambda_*:=\frac{R}{40\gamma_*\ell}$ and $\Lambda_*:=\Lambda_{\gamma_*,\lambda_*}$. We show that $c\le c_K$ at $\gamma=\gamma_*$. If $\sigma_0=\sigma_1=0$, then $\Sigma_{K+1}=0$, $\Phi_{K+1}=0$ and, by the convention in~\eqref{eq:Lambda-gl}, $c=c_\alpha(0)\le c_K$; thus, Steps~2 and~3 hold trivially, and, therefore, we assume that $\Sigma_{K+1}>0$ in both steps. If $\alpha\le1.1$, then $b_\alpha/3\ge b_{1.1}/3>3.5>v_{1.1}(\infty)\ge v_\alpha(\infty)\ge v_\alpha(s)$ for all $s$, since $b_\alpha$ is decreasing and $v_\alpha(\infty)=\frac{2}{2-\alpha}(\frac32)^{2-\alpha}$ is increasing in $\alpha$. In this case, $c_\alpha(\cdot)\equiv60b_\alpha$, and there is nothing to prove. Next, let $\alpha>1.1$, implying $b_\alpha<b_{1.1}<10.8$; we show that $\Lambda_*\le\Lambda_K$. With $\lambda_j:=\frac{R}{40\gamma_j\ell}$ for $j\in\{L,0,1,2\}$, we have $\lambda_*=\max_j\lambda_j$. If $\lambda_*=\lambda_L=2\sqrt2LR$, then $\Lambda_*\le\Lambda_{2\sqrt2LR}\le\Lambda_K$. If $\lambda_*=\lambda_0$, then $\Sigma_{K+1}\ge(K+1)\sigma_0^\alpha$ and the definition of $\gamma_0$ give
\[
    \Lambda_*\le\left(\frac{3\lambda_0}{2\sigma_0}\right)^\alpha=\left(\frac32\right)^\alpha\frac{c_K(K+1)}{40\ell}.
\]
If $\lambda_*=\lambda_2$, the same bound follows from $\Sigma_{K+1}\ge2^{\alpha/2}\sigma_1^\alpha(LR)^\alpha$ and the definition of $\gamma_2$. If $\lambda_*=\lambda_1$, then $\gamma_*=\gamma_1$, and the same bound follows from $\Sigma_{K+1}\ge\sigma_1^\alpha2^\alpha L^{\alpha/2}R^\alpha\gamma_1^{-\alpha/2}(K+1)^{\frac{2-\alpha}{2}}$ and the definition of $\gamma_1$. It remains to show that $(\frac32)^\alpha\frac{c_K(K+1)}{40\ell}\le\Lambda_K$. The map $\Lambda\mapsto\Lambda/c_\alpha(\Lambda)$ is non-decreasing, since $\Lambda\mapsto\Lambda/(1+\frac2\alpha\ln(1+\Lambda))$ is increasing and the other branches of $c_\alpha$ are constant. Since $\Lambda_K\ge80(K+1)$, it is sufficient to show that $(\frac32)^\alpha c_\alpha(80(K+1))\le3200\ell$. We have $(\frac32)^\alpha\cdot180v_\alpha(s)\le405(1+2\ln(1+s))$, and $405(1+2\ln(1+80(K+1)))<3200\ln(4(K+1))<3200\ell$, since the difference of the two sides of the first inequality is increasing in $K$ and positive for $K=0$. Moreover, $(\frac32)^\alpha\cdot60b_\alpha<2.25\cdot60\cdot10.8<3200\ln4<3200\ell$. Thus, $\Lambda_*\le\Lambda_K$, i.e., $c\le c_K$, and $\Phi_{K+1}\,c\le3R$ at $\gamma=\gamma_*$ by Step~1.

\emph{Step 3: Assumption~\ref{as:weak} and $\gamma\le\gamma_*$.} Let $u(\gamma):=\gamma^\alpha\Sigma_{K+1}$, which is non-decreasing in $\gamma$, and $C':=(K+1)(\frac{3R}{80\ell})^\alpha$. For $\lambda=\frac{R}{40\gamma\ell}$, we have $\Phi_{K+1}=(\frac{40\ell}{R})^{\alpha-1}u(\gamma)$ and $\Lambda_{\gamma,\lambda}=C'/u(\gamma)$, which gives
\[
    \Phi_{K+1}\,c=180\left(\frac{40\ell}{R}\right)^{\alpha-1}\max\left\{u\,v_\alpha\!\left(\frac{C'}{u}\right),\ \frac{u\,b_\alpha}3\right\}\Bigg|_{u=u(\gamma)} .
\]
By Lemma~\ref{lem:log-sum}, $u\mapsto u\,v_\alpha(C'/u)$ is non-decreasing, and, therefore, $\Phi_{K+1}\,c$ is non-decreasing in $\gamma$. Since $\Phi_{K+1}\,c\le3R$ at $\gamma=\gamma_*$ by Step~2, the same holds for every $\gamma\le\gamma_*$. Thus, \eqref{eq:master-conds} holds, and Proposition~\ref{prop:master} gives the result, including~\eqref{eq:main-bound}, which concludes the proof.
\hfill$\square$

\paragraph{Proof of Corollary~\ref{cor:rate}.} For $\gamma=\min\{\gamma_L,\gamma_0,\gamma_1,\gamma_2\}$, we have $\frac{2R^2}{\gamma(K+1)}=\max_{j\in\{L,0,1,2\}}\frac{2R^2}{\gamma_j(K+1)}$, and direct computations give $\frac{2R^2}{\gamma_j(K+1)}=\Delta_j$ for $j\in\{L,0,1\}$ and
\[
    \frac{2R^2}{\gamma_2(K+1)}=\Delta_2:=\frac{2\left(c\,2^{\alpha/2}40^{\alpha-1}\right)^{1/\alpha}\sigma_1LR^2\ell^{(\alpha-1)/\alpha}}{K+1}.
\]
For example, $\frac{2R^2}{\gamma_1(K+1)}=2\left(c\,2^\alpha40^{\alpha-1}\right)^{2/\alpha}\sigma_1^2LR^2\ell^{2(\alpha-1)/\alpha}(K+1)^{\frac{2-\alpha}{\alpha}-1}$ with $\frac{2-\alpha}{\alpha}-1=-\frac{2(\alpha-1)}{\alpha}$. {The term $\Delta_2$ is redundant. If $\sigma_1=0$, then $\Delta_1=\Delta_2=0$ and the claim is immediate. If $\sigma_1>0$, then $\Delta_L\Delta_1>0$, and, with $r:=\frac{\alpha-1}{\alpha}\le\frac12$, we have}
\begin{gather*}
    \sqrt{\Delta_L\Delta_1}=\sqrt{320\sqrt2}\left(c\,2^\alpha40^{\alpha-1}\right)^{1/\alpha}\sigma_1LR^2\,\frac{\ell^{r+1/2}}{(K+1)^{r+1/2}},\\
    \frac{\Delta_2}{\sqrt{\Delta_L\Delta_1}}=\frac{2\cdot2^{1/2}}{2\sqrt{320\sqrt2}}\,\ell^{-1/2}(K+1)^{r-1/2}\le\frac{1}{15},
\end{gather*}
since $\ell>1$ and $r\le\frac12$; hence, $\Delta_2\le\sqrt{\Delta_L\Delta_1}\le\max\{\Delta_L,\Delta_1\}$, and $\frac{2R^2}{\gamma(K+1)}=\max\{\Delta_L,\Delta_0,\Delta_1\}$. To derive the complexity bound, we use the fact that each of the three terms has the form $\mathrm{const}\cdot\ln^p(K/\beta)/K^q$ with $p\ge0$ and $q>0$. Such a term is at most $\varepsilon$ as soon as $K\ge A\left(\frac{\mathrm{const}}{\varepsilon}\right)^{1/q}\ln^{p/q}\frac{\mathrm{const}}{\varepsilon\beta}$ for a sufficiently large constant $A=A(p,q)$, which depends only on $\alpha$ through the exponents $(p,q)$ (and blows up as $\alpha\downarrow1$, where $q\downarrow0$), since in this case $\ln(K/\beta)=O\left(\ln\frac{\mathrm{const}}{\varepsilon\beta}\right)$. Under Assumption~\ref{as:weak}, $c=c_\alpha(\Lambda_K)\le\max\{270(1+2\ln(1+80(K+1))+2\ln(1+\Lambda_{2\sqrt2LR})),60b_\alpha\}$ adds a factor that is polylogarithmic in $K$ and in the problem parameters, which can be handled in the same way. Taking the maximum over the three terms with $(q,p)=(1,1)$, $\left(\frac{\alpha-1}{\alpha},\frac{\alpha-1}{\alpha}\right)$, and $\left(\frac{2(\alpha-1)}{\alpha},\frac{2(\alpha-1)}{\alpha}\right)$, we get the stated complexity, which concludes the proof.
\hfill$\square$

\paragraph{Proof of Theorem~\ref{thm:fixed}.} The first condition in~\eqref{eq:master-conds} is assumed, and the second one follows from $\gamma\le\frac{R}{40\lambda\ell}$. Under Assumption~\ref{as:weak}, the variance constant from Proposition~\ref{prop:master} is $v_\alpha(\Lambda_{\gamma,\lambda})\le v_\alpha(\Lambda_\lambda)=v$, since $\Lambda_{\gamma,\lambda}\le\Lambda_\lambda$ and $v_\alpha(\cdot)$ is non-decreasing. By~\eqref{eq:Sigma-def}, we have $\Sigma_{K+1}=S_0+S_1\gamma^{-\alpha/2}+S_2=\Sigma_{K+1}(\gamma)$, meaning that \eqref{eq:fixed-bias} is the third condition in~\eqref{eq:master-conds}, $b\Phi_{K+1}\le R/20$, and~\eqref{eq:fixed-var} is the fourth one, $v\gamma^2\lambda^{2-\alpha}\Sigma_{K+1}\le R^2/(2400\ell)$. Therefore, the result follows from Proposition~\ref{prop:master}.
\hfill$\square$

\paragraph{Explicit thresholds.} Let
\begin{equation}\label{eq:fixed-thresholds}
\begin{aligned}
    \gamma_b^\lambda&:=\min\left\{\frac{R\lambda^{\alpha-1}}{60\,b\,S_0},\ \left(\frac{R\lambda^{\alpha-1}}{60\,b\,S_1}\right)^{\frac{2}{2-\alpha}},\ \frac{R\lambda^{\alpha-1}}{60\,b\,S_2}\right\},\\
    \gamma_v^\lambda&:=\min\left\{\frac{R}{\sqrt{7200\,\ell\,v\lambda^{2-\alpha}S_0}},\ \left(\frac{R^2}{7200\,\ell\,v\lambda^{2-\alpha}S_1}\right)^{\frac{2}{4-\alpha}},\ \frac{R}{\sqrt{7200\,\ell\,v\lambda^{2-\alpha}S_2}}\right\},
\end{aligned}
\end{equation}
with the convention that a term with $S_j=0$ equals $+\infty$; for $\alpha=2$, the second term of $\gamma_b^\lambda$ is replaced by $+\infty$ and the condition $\lambda\ge60\,b\,S_1/R=240\,b\,\sigma_1^2LR$ is imposed instead. Then every $\gamma\le\min\{\gamma_b^\lambda,\gamma_v^\lambda\}$ satisfies~\eqref{eq:fixed-bias} and~\eqref{eq:fixed-var}. Indeed, \eqref{eq:fixed-bias} holds if each of the three terms of $b\gamma\lambda^{1-\alpha}\Sigma_{K+1}(\gamma)$ is at most $R/60$:
\begin{align*}
    b\gamma\lambda^{1-\alpha}S_0\le\frac R{60}\quad&\Longleftrightarrow\quad\gamma\le\frac{R\lambda^{\alpha-1}}{60\,b\,S_0},\\
    b\gamma^{1-\alpha/2}\lambda^{1-\alpha}S_1\le\frac R{60}\quad&\Longleftrightarrow\quad\gamma^{\frac{2-\alpha}{2}}\le\frac{R\lambda^{\alpha-1}}{60\,b\,S_1}\quad\Longleftrightarrow\quad\gamma\le\left(\frac{R\lambda^{\alpha-1}}{60\,b\,S_1}\right)^{\frac2{2-\alpha}},\\
    b\gamma\lambda^{1-\alpha}S_2\le\frac R{60}\quad&\Longleftrightarrow\quad\gamma\le\frac{R\lambda^{\alpha-1}}{60\,b\,S_2},
\end{align*}
where the middle line is for $\alpha<2$; for $\alpha=2$, the middle term equals $bS_1/\lambda=4b\sigma_1^2LR^2/\lambda$, which is at most $R/60$ since $\lambda\ge240\,b\,\sigma_1^2LR$. Similarly, \eqref{eq:fixed-var} holds if each of the three terms of $v\gamma^2\lambda^{2-\alpha}\Sigma_{K+1}(\gamma)$ is at most $R^2/(7200\ell)$:
\begin{align*}
    v\gamma^2\lambda^{2-\alpha}S_0\le\frac{R^2}{7200\ell}\quad&\Longleftrightarrow\quad\gamma\le\frac{R}{\sqrt{7200\,\ell\,v\lambda^{2-\alpha}S_0}},\\
    v\gamma^{2-\alpha/2}\lambda^{2-\alpha}S_1\le\frac{R^2}{7200\ell}\quad&\Longleftrightarrow\quad\gamma\le\left(\frac{R^2}{7200\,\ell\,v\lambda^{2-\alpha}S_1}\right)^{\frac{2}{4-\alpha}},\\
    v\gamma^2\lambda^{2-\alpha}S_2\le\frac{R^2}{7200\ell}\quad&\Longleftrightarrow\quad\gamma\le\frac{R}{\sqrt{7200\,\ell\,v\lambda^{2-\alpha}S_2}},
\end{align*}
(for $\alpha=2$, the middle line reads $\gamma\le R^2/(7200\ell vS_1)$, which is the stated term).

\paragraph{Proof of Corollary~\ref{cor:fixed}.} By the previous paragraph, $\gamma=\min\{\frac{R}{40\lambda\ell},\gamma_b^\lambda,\gamma_v^\lambda\}$ is admissible in Theorem~\ref{thm:fixed}, and $\frac{2R^2}{\gamma(K+1)}$ is the maximum of $\frac{2R^2}{(K+1)}$ divided by each of the seven thresholds (six for $\alpha=2$, where the $S_1$ bias threshold is $+\infty$ and the corresponding term below is absent). The first one gives $\frac{2R^2}{\frac{R}{40\lambda\ell}(K+1)}=\frac{80\lambda R\ell}{K+1}$. The three bias thresholds give
\begin{gather*}
    \frac{2R^2\cdot60\,b\,S_0}{R\lambda^{\alpha-1}(K+1)}=\frac{120\,bR\sigma_0^\alpha}{\lambda^{\alpha-1}},\qquad
    \frac{2R^2}{K+1}\left(\frac{60\,b\,S_1}{R\lambda^{\alpha-1}}\right)^{\frac2{2-\alpha}}=2\left(\frac{60\cdot2^\alpha bRL^{\alpha/2}\sigma_1^\alpha}{\lambda^{\alpha-1}}\right)^{\frac2{2-\alpha}},\\
    \frac{2R^2\cdot60\,b\,S_2}{R\lambda^{\alpha-1}(K+1)}=\frac{120\cdot2^{\alpha/2}b\,\sigma_1^\alpha L^\alpha R^{\alpha+1}}{\lambda^{\alpha-1}(K+1)}\le\frac{120\cdot2^{(3-2\alpha)/2}b\,\sigma_1^\alpha LR^2}{K+1}\le\frac{120\sqrt2\,b\,\sigma_1^\alpha LR^2}{K+1},
\end{gather*}
where the second identity uses $S_1^{2/(2-\alpha)}=(2^\alpha\sigma_1^\alpha L^{\alpha/2}R^\alpha)^{2/(2-\alpha)}(K+1)$ and $R^2(R^{\alpha-1})^{2/(2-\alpha)}=R^{2/(2-\alpha)}$, and the last line uses $\lambda\ge2^{3/2}LR$ and $\alpha>1$. The three variance thresholds give, with $\sqrt{7200}=60\sqrt2$,
\begin{gather*}
    \frac{2R\sqrt{7200\,\ell\,v\lambda^{2-\alpha}S_0}}{K+1}=120\sqrt2\,R\sqrt{\frac{v\lambda^{2-\alpha}\sigma_0^\alpha\ell}{K+1}},\\
    \frac{2R^2}{K+1}\left(\frac{7200\,\ell\,v\lambda^{2-\alpha}S_1}{R^2}\right)^{\frac{2}{4-\alpha}},\qquad
    \frac{120\sqrt2\,R\sqrt{v\lambda^{2-\alpha}S_2\,\ell}}{K+1}.
\end{gather*}
This proves~\eqref{eq:fixed-bound}. The orders in $K$ follow from $S_1\propto(K+1)^{(2-\alpha)/2}$: the fourth term of $\mathcal V_{\lambda,K}$ is of the order $(K+1)^{\frac{2-\alpha}{4-\alpha}-1}=(K+1)^{-\frac{2}{4-\alpha}}$. Next, if $\sigma_1=0$ and $\lambda=\lambda_0(K+1)^{1/\alpha}$, then $K+1\ge(2\sqrt2LR/\lambda_0)^\alpha$ is equivalent to $\lambda\ge2\sqrt2LR$, and, substituting $\lambda$, we get $\frac{80\lambda R\ell}{K+1}=\frac{80\lambda_0R\ell}{(K+1)^{(\alpha-1)/\alpha}}$, $\frac{120bR\sigma_0^\alpha}{\lambda^{\alpha-1}}=\frac{120b\,\sigma_0^\alpha\lambda_0^{1-\alpha}R}{(K+1)^{(\alpha-1)/\alpha}}$, and
\[
    120\sqrt2\,R\sqrt{\frac{v\lambda^{2-\alpha}\sigma_0^\alpha\ell}{K+1}}=120\sqrt{2v\ell}\,\sigma_0^{\alpha/2}\lambda_0^{1-\alpha/2}R\,(K+1)^{\frac{2-\alpha}{2\alpha}-\frac12}=\frac{120\sqrt{2v\ell}\,\sigma_0^{\alpha/2}\lambda_0^{1-\alpha/2}R}{(K+1)^{(\alpha-1)/\alpha}} .
\]
Bounding the maximum of the three quantities by their sum, we get~\eqref{eq:horizon-bound}. Finally, for $\sigma_1>0$, the remaining four terms are of the orders $(K+1)^{-\frac{2(\alpha-1)}{\alpha(2-\alpha)}}$ (the second part of $\mathcal N_\lambda$), $(K+1)^{-1}$, $(K+1)^{\frac{2-\alpha}{2\alpha}-1}$, and $(K+1)^{\frac{2(2-\alpha)}{\alpha(4-\alpha)}-\frac{2}{4-\alpha}}=(K+1)^{-\frac{4(\alpha-1)}{\alpha(4-\alpha)}}$, and each exponent is smaller than $-\frac{\alpha-1}{\alpha}$ for $\alpha\in(1,2)$, which concludes the proof.
\hfill$\square$

\section{Clip-SGD with an Arbitrary Clipping Level: Proof of Theorem~\ref{thm:any-lambda}}\label{app:any-lambda}

\subsection{Full Statement}\label{app:any-statement}

\begin{theorem}[Any clipping level: full statement of Theorem~\ref{thm:any-main}]\label{thm:any-lambda}
{Let Assumptions~\ref{as:smooth}--\ref{as:convex} hold for $Q=B_{2R}(x^*)$, where $R\ge\|x^0-x^*\|>0$, and let Assumption~\ref{as:moment} (with $\alpha\in(1,2]$) or Assumption~\ref{as:weak} (with $\alpha\in(1,2)$) hold with $\sigma_1\le\bar\sigma_1$. Let $\lambda>0$ be arbitrary, $\beta\in(0,1)$, $K\in\{0,1,\dots\}$, $\ell:=\ln\frac{4(K+1)}{\beta}$, $\bar A:=\sigma_0^\alpha+\sigma_1^\alpha(\sqrt2LR)^\alpha$, $V_\lambda:=\min\{4\lambda^2,\kappa_2(2\lambda)^{2-\alpha}\bar A\}$, and let the stepsize satisfy}
\[
    {0<\gamma\le\min\left\{\frac{1}{48L},\,\frac{R}{40\lambda\ell},\,\frac{R}{\sqrt{1200\,\ell\,V_\lambda(K+1)}}\right\},}
\]
where the last term is $+\infty$ if $\bar A=0$. Then, with probability at least $1-\beta$, the iterates produced by Algorithm~\ref{alg:clip-sgd} with $\lambda_k\equiv\lambda$ satisfy
\begin{align*}
    \min_{0\le t\le K}f(x^t)-f^*\le\max\Bigg\{\varepsilon_N,\ \varepsilon_M,\ \frac{4R^2}{\gamma(K+1)},\ \frac{4608LR^4}{\gamma^2\lambda^2(K+1)^2}\Bigg\},
\end{align*}
where $\varepsilon_N:=48\sqrt2\,c_\perp R\,\bar b_\lambda$ and $\varepsilon_M:=384\,c_\perp\sigma_1^\alpha LR^2\min\{1,\,96\,c_\perp\sigma_1^\alpha\}$.
\end{theorem}

For $\gamma$ equal to the minimum, the last two terms are
\begin{align*}
    O\Bigg(\frac{LR^2+\lambda R\ell}{K}+R\sqrt{\frac{\ell V_\lambda}{K}}+\frac{L^3R^4}{\lambda^2K^2}+\frac{LR^2\ell^2}{K^2}+\frac{\ell LR^2V_\lambda}{\lambda^2K}\Bigg)
\end{align*}
with an absolute constant, and they vanish as $K\to\infty$ (note that $V_\lambda\le\max\{4,2^{2-\alpha}\kappa_2\}\lambda(\bar b_\lambda+\sigma_1^\alpha(\sqrt2LR)^\alpha\lambda^{1-\alpha})$). This is exactly the term $\mathcal V'_{\lambda,K}$ from Theorem~\ref{thm:any-main}.

Throughout this section, $\sigma_1\le\bar\sigma_1=1/(2\cdot28^{1/\alpha})$, and $(\kappa_1,\kappa_2)$, $c_\perp$, $\bar b_\lambda$, $\bar A$, and $V_\lambda$ are defined as in Section~\ref{sec:planned}; $A(x)=\sigma_0^\alpha+\sigma_1^\alpha\|x\|^\alpha$ is defined as in Section~\ref{sec:clipping}, and $(b,v)$ are the clipping constants. We will use the following three consequences of the noise assumptions. Let $z$ be a random vector with $\PP\{\|z\|>u\}\le Au^{-\alpha}$ for all $u>0$ (under Assumption~\ref{as:weak}) or with $\E\|z\|^\alpha\le A$ (under Assumption~\ref{as:moment}). Then, for every $c>0$, we have
\begin{equation}\label{eq:tail-facts}
    \PP\{\|z\|>c\}\le Ac^{-\alpha},\qquad
    \E\left[\|z\|\ind_{\{\|z\|>c\}}\right]\le\kappa_1Ac^{1-\alpha},\qquad
    \E\min\{\|z\|,c\}^2\le\kappa_2Ac^{2-\alpha}.
\end{equation}
Indeed, under Assumption~\ref{as:weak}, the second bound follows from $\E[Y\ind_{\{Y>c\}}]=c\PP\{Y>c\}+\int_c^\infty\PP\{Y>u\}\dif u$, and the third one follows from $\E\min\{Y,c\}^2=2\int_0^cu\PP\{Y>u\}\dif u$. Under Assumption~\ref{as:moment}, the first bound is Markov's inequality, and the other two follow from $Y\ind_{\{Y>c\}}\le Y^\alpha c^{1-\alpha}$ and $\min\{Y,c\}^2\le Y^\alpha c^{2-\alpha}$, where $Y:=\|z\|$ and $\alpha\le2$.

\subsection{Clipping with an Arbitrary Level}

\begin{lemma}[Clipping with an arbitrary level]\label{lem:arb-clip}
Let $a\in\R^d$ with $G:=\|a\|>0$, let $\lambda>0$, $D:=\max\{G,\lambda\}$, $\hat a:=a/G$ and $P_\perp:=I-\hat a\hat a^\top$. Let $z$ be a random vector with $\E z=0$ satisfying~\eqref{eq:tail-facts} with some $A\ge0$, and let $\tX:=\clip(a+z,\lambda)$. Then the following statements hold:
\begin{enumerate}
    \item[(i)] $\E\|\tX-\clip(a,\lambda)\|^2\le\E\min\{\|z\|,2\lambda\}^2\le\min\{4\lambda^2,\kappa_2(2\lambda)^{2-\alpha}A\}$;
    \item[(ii)] $\|P_\perp\E\tX\|\le c_\perp\lambda AD^{-\alpha}$;
    \item[(iii)] if $G\ge\lambda/2$ and $\PP\{\|z\|>G/2\}\le\frac1{14}$, then $\abr{\E\tX,\hat a}\ge\frac16\min\{G,\lambda\}$.
\end{enumerate}
\end{lemma}

\begin{proof}
(i) The map $\clip(\cdot,\lambda)$ is the Euclidean projection onto $B_\lambda(0)$; therefore, it is $1$-Lipschitz, and both $\tX$ and $\clip(a,\lambda)$ belong to $B_\lambda(0)$. Thus, $\|\tX-\clip(a,\lambda)\|\le\min\{\|z\|,2\lambda\}$, and the claim follows from~\eqref{eq:tail-facts} with $c=2\lambda$.

(ii) Let $y:=a+z$, $M(v):=\max\{\|v\|,\lambda\}$, and $s(v):=\lambda/M(v)=\min\{1,\lambda/\|v\|\}$; then $\tX=s(y)y$. Since $P_\perp a=0$ and $\E P_\perp z=0$, we have $P_\perp\E\tX=\E[s(y)P_\perp z]=\E[(s(y)-s(a))P_\perp z]$. The map $v\mapsto M(v)$ is $1$-Lipschitz and $M(a)=D$; hence, $|s(y)-s(a)|=\lambda|M(a)-M(y)|/(M(a)M(y))\le\lambda\|z\|/(DM(y))$. On the event $\{\|z\|\le D/2\}$, we have $M(y)\ge D/2$ (if $G\ge\lambda$, then $\|y\|\ge G-\|z\|\ge G/2$; otherwise, $M(y)\ge\lambda=D$). Thus, $\|(s(y)-s(a))P_\perp z\|\le2\lambda\|z\|^2/D^2$ on this event, and, by~\eqref{eq:tail-facts},
\[
    \E\left[\|(s(y)-s(a))P_\perp z\|\ind_{\{\|z\|\le D/2\}}\right]\le\frac{2\lambda}{D^2}\E\min\left\{\|z\|,\frac D2\right\}^2\le2^{\alpha-1}\kappa_2\lambda AD^{-\alpha}.
\]
On the event $\{\|z\|>D/2\}$, we use $\|s(y)P_\perp z\|=\|P_\perp\tX\|\le\lambda<2\lambda\|z\|/D$ and $s(a)\|P_\perp z\|\le\lambda\|z\|/D$. Thus, $\|(s(y)-s(a))P_\perp z\|\le3\lambda\|z\|/D$ on this event, and
\[
    \E\left[\|(s(y)-s(a))P_\perp z\|\ind_{\{\|z\|>D/2\}}\right]\le\frac{3\lambda}{D}\E\left[\|z\|\ind_{\{\|z\|>D/2\}}\right]\le3\cdot2^{\alpha-1}\kappa_1\lambda AD^{-\alpha}.
\]
Summing up the two bounds, we get (ii).

(iii) Let $p:=\PP\{\|z\|>G/2\}\le\frac1{14}$. On the event $\{\|z\|\le G/2\}$, we have $\abr{y,\hat a}=G+\abr{z,\hat a}\ge G/2$ and $\|y\|\le3G/2$; hence, $s(y)\ge\min\{1,\frac{2\lambda}{3G}\}$ and $\abr{\tX,\hat a}=s(y)\abr{y,\hat a}\ge\min\{\frac G2,\frac\lambda3\}\ge\frac13\min\{G,\lambda\}$. On the complementary event, we have $\abr{\tX,\hat a}\ge-\|\tX\|\ge-\lambda$. Since $G\ge\lambda/2$, we have $\lambda\le2\min\{G,\lambda\}$, and, therefore,
\[
    \abr{\E\tX,\hat a}\ge\frac{1-p}{3}\min\{G,\lambda\}-p\lambda\ge\left(\frac13-\frac{7p}{3}\right)\min\{G,\lambda\}\ge\frac16\min\{G,\lambda\}.
\]
\end{proof}

In the proof of Theorem~\ref{thm:any-lambda}, we apply these deterministic statements pointwise at random points via the freezing formula (see below). This allows us to avoid conditional versions with random thresholds such as $G/2$ and $D/2$.

\Needspace{14\baselineskip}
\subsection{The Mean Clipped Field}

\begin{lemma}[Mean clipped field]\label{lem:drift}
Let the assumptions of Theorem~\ref{thm:any-lambda} hold, and let $x\in B_{\sqrt2R}(x^*)$, $u:=x-x^*$, $\delta:=f(x)-f^*$, $G:=\|\nabla f(x)\|$, and $X:=g(x,\xi)$ with $\xi\sim\D$. Let
\[
    w:=1\ \ \text{if}\ G\le\frac\lambda2,\qquad w:=\frac{\min\{G,\lambda\}}{6G}\ \ \text{if}\ G>\frac\lambda2,
\]
and let $\varepsilon_N=48\sqrt2\,c_\perp R\,\bar b_\lambda$ and $\varepsilon_M=384\,c_\perp\sigma_1^\alpha LR^2\min\{1,96\,c_\perp\sigma_1^\alpha\}$ be defined as in Theorem~\ref{thm:any-lambda}. If $\delta>\max\{\varepsilon_N,\varepsilon_M\}$, then $\abr{\E\clip(X,\lambda),u}\ge\frac34w\delta$.
\end{lemma}

\begin{proof}
Let $\tX:=\clip(X,\lambda)$ and $z:=X-\nabla f(x)$; then $\E z=0$ and~\eqref{eq:tail-facts} holds with $A=A(\nabla f(x))=\sigma_0^\alpha+\sigma_1^\alpha G^\alpha$. By convexity, we have $\abr{\nabla f(x),u}\ge\delta$, and $\delta\le G\|u\|\le\sqrt2RG$. In particular, $G>0$. By Assumption~\ref{as:grad-bound}, we also have $\delta\ge G^2/(2L)$. We start with the following consequence of $\delta>\varepsilon_M$: for every $C\in(0,96\sqrt2\,c_\perp]$, we have
\begin{equation}\label{eq:eps-M-use}
    C\sigma_1^\alpha RG\le\delta .
\end{equation}
Indeed, if $\varepsilon_M=384c_\perp\sigma_1^\alpha LR^2$, then either $G\ge2CL\sigma_1^\alpha R$, and then $C\sigma_1^\alpha RG\le G^2/(2L)\le\delta$, or $G<2CL\sigma_1^\alpha R$, and then $\delta\le\sqrt2RG<2\sqrt2C\sigma_1^\alpha LR^2\le384c_\perp\sigma_1^\alpha LR^2=\varepsilon_M$, which is excluded. If $\varepsilon_M=36864c_\perp^2\sigma_1^{2\alpha}LR^2$, then $C\sigma_1^\alpha RG\le C\sigma_1^\alpha R\sqrt{2L\delta}\le\delta$, since $\delta>\varepsilon_M\ge2C^2\sigma_1^{2\alpha}LR^2$ (note that $2(96\sqrt2c_\perp)^2=36864c_\perp^2$). Next, we consider three cases.

\emph{Case $G\le\lambda/2$.} By Lemma~\ref{lem:weak-clip} or Lemma~\ref{lem:moment-clip}, we have $\|\E\tX-\nabla f(x)\|\le bA\lambda^{1-\alpha}=b\sigma_0^\alpha\lambda^{1-\alpha}+b\sigma_1^\alpha G^\alpha\lambda^{1-\alpha}$, and also $\|\E\tX-\nabla f(x)\|\le\lambda+G\le\frac32\lambda$. If $\lambda\ge\sigma_0$, then the first bound contains $b\sigma_0^\alpha\lambda^{1-\alpha}=b\bar b_\lambda$; otherwise, the second one is at most $\frac32\sigma_0=\frac32\bar b_\lambda$. In both cases, since $G\le\lambda/2$ implies $G^\alpha\lambda^{1-\alpha}\le2^{1-\alpha}G$, we have
\[
    \|\E\tX-\nabla f(x)\|\le\max\left\{b,\tfrac32\right\}\bar b_\lambda+2^{1-\alpha}b\,\sigma_1^\alpha G .
\]
Thus, by $\|u\|\le\sqrt2R$, $\max\{b,\frac32\}\le c_\perp$, and~\eqref{eq:eps-M-use} with $C=8\sqrt2\cdot2^{1-\alpha}b\le8\sqrt2c_\perp$, we get
\[
    \abr{\E\tX,u}\ge\delta-\sqrt2R\,c_\perp\bar b_\lambda-\sqrt2\cdot2^{1-\alpha}b\,\sigma_1^\alpha RG\ge\delta-\frac{\varepsilon_N}{48}-\frac\delta8\ge\frac34\delta=\frac34w\delta .
\]
Here $\max\{b,\frac32\}\le c_\perp$ since $c_\perp\ge2^{\alpha-1}\cdot4\ge4$; under Assumption~\ref{as:weak}, $c_\perp\ge3\cdot2^{\alpha-1}\kappa_1=3\alpha b_\alpha\ge b_\alpha$; and under Assumption~\ref{as:moment}, $c_\perp=2^{\alpha+1}\ge4>1\ge b$.

\emph{Case $G>\lambda/2$ and $G\ge G_1:=2\cdot28^{1/\alpha}\sigma_0$.} By~\eqref{eq:tail-facts}, $G\ge G_1$, and $\sigma_1\le\bar\sigma_1$, we have
\[
    \PP\{\|z\|>G/2\}\le2^\alpha AG^{-\alpha}=\frac{2^\alpha\sigma_0^\alpha}{G^\alpha}+2^\alpha\sigma_1^\alpha\le\frac1{28}+\frac1{28}=\frac1{14},
\]
and Lemma~\ref{lem:arb-clip}(iii) gives $\abr{\E\tX,\hat a}\ge\frac16\min\{G,\lambda\}=wG$, where $\hat a:=\nabla f(x)/G$. Since $\abr{\hat a,u}\ge\delta/G\ge0$, the component of $\E\tX$ along $\hat a$ contributes at least $w\delta$ to $\abr{\E\tX,u}$. Next, let $D:=\max\{G,\lambda\}$; then $w=\lambda/(6D)$ in this case, $D\ge\max\{\lambda,\sigma_0\}$ (since $G_1\ge\sigma_0$), and $D<2G$ (since $\lambda<2G$). By Lemma~\ref{lem:arb-clip}(ii), the orthogonal component contributes at least
\begin{align*}
    -\sqrt2R\,c_\perp\lambda AD^{-\alpha}&=-\sqrt2R\,c_\perp\lambda\sigma_0^\alpha D^{-\alpha}-\sqrt2R\,c_\perp\lambda\sigma_1^\alpha\left(\frac GD\right)^\alpha\\
    &\ge-w\cdot6\sqrt2\,c_\perp R\sigma_0^\alpha D^{1-\alpha}-w\cdot6\sqrt2\,c_\perp R\sigma_1^\alpha D .
\end{align*}
The first term is at least $-w\cdot6\sqrt2c_\perp R\bar b_\lambda=-w\varepsilon_N/8\ge-w\delta/8$, since $D\ge\max\{\lambda,\sigma_0\}$. The second one is at least $-w\cdot12\sqrt2c_\perp R\sigma_1^\alpha G\ge-w\delta/8$, due to $D<2G$ and~\eqref{eq:eps-M-use} with $C=96\sqrt2c_\perp$. Thus, $\abr{\E\tX,u}\ge w(\delta-\delta/8-\delta/8)=\frac34w\delta$.

\emph{Case $\lambda/2<G<G_1$.} This case is impossible when $\delta>\varepsilon_N$. Indeed, it requires $\lambda<2G_1=4\cdot28^{1/\alpha}\sigma_0$. Then $\max\{\lambda,\sigma_0\}\le4\cdot28^{1/\alpha}\sigma_0$, and, therefore, $\sigma_0\le(4\cdot28^{1/\alpha})^{\alpha-1}\bar b_\lambda$. Therefore,
\[
    \delta\le\sqrt2RG<\sqrt2RG_1=2\cdot28^{1/\alpha}\sqrt2R\sigma_0\le2\cdot28\cdot4^{\alpha-1}\sqrt2R\,\bar b_\lambda=14\cdot4^\alpha\sqrt2R\,\bar b_\lambda\le\varepsilon_N,
\]
where the last step uses $c_\perp\ge2^{\alpha-1}\cdot4=2^{\alpha+1}$ (since $\kappa_1,\kappa_2\ge1$) and $14\cdot4^\alpha\le48\cdot2^{\alpha+1}$, i.e., $14\cdot2^\alpha\le96$, for $\alpha\le2$. This concludes the proof.
\end{proof}

\subsection{Proof of Theorem~\ref{thm:any-lambda}}

Let $\varepsilon_N$ and $\varepsilon_M$ be defined as in Lemma~\ref{lem:drift}, let $\varepsilon_*:=\max\{\varepsilon_N,\varepsilon_M\}$, let us introduce the notation $R_t:=\|x^t-x^*\|$, $u_t:=x^t-x^*$, $\delta_t:=f(x^t)-f^*$, $G_t:=\|\nabla f(x^t)\|$, and let $w_t$ be the weight from Lemma~\ref{lem:drift} at $x^t$. Let $\tg^t:=\clip(g^t,\lambda)$, $\theta_t^u:=\tg^t-\E[\tg^t\mid\G_t]$ (as in~\eqref{eq:theta}), and $\widetilde m_t:=\|\tg^t-\clip(\nabla f(x^t),\lambda)\|^2$. Let
\[
    \tau:=\inf\{t\ge0:\ \delta_t\le\varepsilon_*\},\qquad \iota_t:=\ind_{\{t<\tau\}},\qquad \eta_t:=u_t\ind_{\{R_t\le\sqrt2R\}} .
\]
Since $\delta_t$ is $\G_t$-measurable, $\tau$ is a stopping time, and $\iota_t$, $w_t$, and $\eta_t$ are $\G_t$-measurable.

\paragraph{One-step inequality.} Let $t$ be such that $\iota_t=1$ and $R_t\le\sqrt2R$. Then we have
\begin{equation}\label{eq:one-step}
    R_{t+1}^2\le R_t^2-\gamma w_t\delta_t-2\gamma\abr{\theta_t^u,\eta_t}+2\gamma^2\widetilde m_t .
\end{equation}
Indeed, $R_{t+1}^2=R_t^2-2\gamma\abr{\tg^t,u_t}+\gamma^2\|\tg^t\|^2$, $u_t=\eta_t$, and $\|\tg^t\|^2\le2\min\{G_t,\lambda\}^2+2\widetilde m_t$, since $\|\clip(\nabla f(x^t),\lambda)\|=\min\{G_t,\lambda\}$. Let $F(x):=\E_\xi\clip(g(x,\xi),\lambda)$ and $\Phi(x):=\E_\xi\|\clip(g(x,\xi),\lambda)-\clip(\nabla f(x),\lambda)\|^2$. These functions are bounded and measurable. Since $x^t$ is $\G_t$-measurable and $\xi^t$ is independent of $\G_t$, the freezing formula (as in the proof of Lemma~\ref{lem:adaptive}) gives $\E[\tg^t\mid\G_t]=F(x^t)$ and $\E[\widetilde m_t\mid\G_t]=\Phi(x^t)$ almost surely. Thus, Lemma~\ref{lem:drift}, applied at the point $x^t(\omega)\in B_{\sqrt2R}(x^*)$, gives $\abr{\E[\tg^t\mid\G_t],u_t}\ge\frac34w_t\delta_t$. Finally, $2\gamma^2\min\{G_t,\lambda\}^2\le\frac12\gamma w_t\delta_t$. For $G_t\le\lambda/2$, this follows from $G_t^2\le2L\delta_t$ and $\gamma\le1/(8L)$. For $G_t>\lambda/2$, it is equivalent to $24\gamma\min\{G_t,\lambda\}G_t\le\delta_t$, which follows from $\min\{G_t,\lambda\}G_t\le G_t^2\le2L\delta_t$ and $\gamma\le1/(48L)$.

\paragraph{Variance bounds.} If $R_t\le\sqrt2R$, then $\|\nabla f(x^t)\|\le LR_t\le\sqrt2LR$, implying $A(\nabla f(x^t))\le\bar A$, and Lemma~\ref{lem:arb-clip}(i), applied at the point $x^t(\omega)$ with $A=A(\nabla f(x^t))$ via the freezing formula, gives $\E[\|\theta_t^u\|^2\mid\G_t]\le\E[\widetilde m_t\mid\G_t]\le V_\lambda$, where we use that the bound in Lemma~\ref{lem:arb-clip}(i) is non-decreasing in $A$. Moreover, we always have $\|\theta_t^u\|\le2\lambda$ and $\widetilde m_t\le4\lambda^2$.

\paragraph{Events and induction.} As in the proof of Proposition~\ref{prop:master}, our proof is induction-based. For $T\ge1$, let $S_T:=-2\gamma\sum_{t<T}\iota_t\abr{\theta_t^u,\eta_t}+2\gamma^2\sum_{t<T}\iota_t\widetilde m_t$. We introduce probability event $E_k$ for each $k=0,1,\dots,K+1$ as follows: $E_0:=\Omega$ and, for $k=1,\dots,K+1$,
\[
    E_k:=\left\{\forall T\in\{1,\dots,k\}:\ S_T\le R^2\ \text{ and }\ \left(T\le\tau\ \Rightarrow\ R_T\le\sqrt2R\right)\right\}.
\]
We will prove by induction that $\PP(E_k)\ge1-\frac{k\beta}{K+1}$ for all $k=0,1,\dots,K+1$. The base of the induction is trivial. Next, we assume that the statement holds for $k=T-1$ with some $T\in\{1,\dots,K+1\}$ and prove it for $k=T$. On $E_{T-1}$, every $t<T$ with $\iota_t=1$ satisfies $R_t\le\sqrt2R$ (for $t=0$, $R_0\le R$), meaning that $x^t\in Q$ and the variance bounds hold for all active $t<T$.

Let $X_t:=-2\gamma\iota_t\abr{\theta_t^u,\eta_t}$. This is a martingale difference sequence with $|X_t|\le4\sqrt2\gamma\lambda R\le\frac{R^2}{5\ell}=:c_1$, since $\gamma\lambda\le R/(40\ell)$. On $E_{T-1}$, its conditional variances satisfy $\sum_{t<T}\E[X_t^2\mid\G_t]\le8\gamma^2R^2V_\lambda(K+1)\le\frac{R^4}{150\ell}=:G_F$, since $\gamma^2V_\lambda(K+1)\le\frac{R^2}{1200\ell}$. Next, let $Y_t:=2\gamma^2\iota_t(\widetilde m_t-\E[\widetilde m_t\mid\G_t])$. Then $Y_t$ is a martingale difference with $|Y_t|\le8\gamma^2\lambda^2\le\frac{R^2}{200\ell^2}\le c_1$, and, since $\iota_t$ is $\G_t$-measurable with $\iota_t^2=\iota_t$, $\E[Y_t^2\mid\G_t]=4\gamma^4\iota_t\Var(\widetilde m_t\mid\G_t)\le4\gamma^4\iota_t\E[\widetilde m_t^2\mid\G_t]\le16\gamma^4\lambda^2\iota_t\E[\widetilde m_t\mid\G_t]$. On $E_{T-1}$, every active $t<T$ satisfies $x^t\in Q$, and, therefore, $\iota_t\E[\widetilde m_t\mid\G_t]\le\iota_tV_\lambda$, while for inactive $t$ both sides vanish; hence, on $E_{T-1}$, $\sum_{t<T}\E[Y_t^2\mid\G_t]\le16\gamma^4\lambda^2V_\lambda(K+1)=16(\gamma\lambda)^2\gamma^2V_\lambda(K+1)\le G_F$. Lemma~\ref{lem:freedman} with $b_1=R^2/5$, $c_1$, and $G_F$, applied in the same way as in the proof of Proposition~\ref{prop:master}, gives events $E_X$ and $E_Y$, each of probability at least $1-\frac{\beta}{2(K+1)}$, such that $|\sum_{t<T}X_t|\le\frac{R^2}{5}$ and $|\sum_{t<T}Y_t|\le\frac{R^2}5$ on $E_{T-1}\cap E_X\cap E_Y$. Since, in addition, $2\gamma^2\sum_{t<T}\iota_t\E[\widetilde m_t\mid\G_t]\le2\gamma^2V_\lambda(K+1)\le\frac{R^2}{600}$ on $E_{T-1}$, we get $S_T\le\frac{R^2}5+\frac{R^2}5+\frac{R^2}{600}<R^2$ on this event. If, moreover, $T\le\tau$, then all $t<T$ are active, and, summing up~\eqref{eq:one-step} over $t<T$, we get
\begin{equation}\label{eq:any-lambda-sum}
    R_T^2+\gamma\sum_{t<T}w_t\delta_t\le R_0^2+S_T\le2R^2 .
\end{equation}
Thus, $E_{T-1}\cap E_X\cap E_Y\subseteq E_T$, and the union bound gives $\PP(E_T)\ge1-\frac{T\beta}{K+1}$, which finishes the induction.

\paragraph{Conclusion.} Finally, we consider the event $E_{K+1}$, which has probability at least $1-\beta$. If $\tau\le K$, then $\min_{t\le K}\delta_t\le\delta_\tau\le\varepsilon_*$. Otherwise, \eqref{eq:any-lambda-sum} with $T=K+1$ gives $\sum_{t\le K}w_t\delta_t\le2R^2/\gamma$. Let $\mathcal T_1:=\{t\le K:\ G_t\le\lambda/2\}$ and $\mathcal T_2:=\{0,\dots,K\}\setminus\mathcal T_1$. If $|\mathcal T_1|\ge\frac{K+1}2$, then $\frac{K+1}{2}\min_t\delta_t\le\sum_{t\in\mathcal T_1}\delta_t\le\frac{2R^2}\gamma$, since $w_t=1$ on $\mathcal T_1$. Otherwise, $|\mathcal T_2|>\frac{K+1}2$. For $t\in\mathcal T_2$, we have $\min\{G_t,\lambda\}\ge\lambda/2$ and $G_t\le\sqrt{2L\delta_t}$; hence, $w_t\delta_t\ge\frac{\lambda\delta_t}{12G_t}\ge\frac{\lambda\sqrt{\delta_t}}{12\sqrt{2L}}$, and
\[
    \frac{K+1}{2}\cdot\frac{\lambda\sqrt{\min_t\delta_t}}{12\sqrt{2L}}\le\frac{2R^2}{\gamma},\qquad\text{i.e.,}\qquad \min_t\delta_t\le\frac{4608LR^4}{\gamma^2\lambda^2(K+1)^2}.
\]
In all cases, we have $\min_{t\le K}\delta_t\le\max\{\varepsilon_N,\varepsilon_M,\frac{4R^2}{\gamma(K+1)},\frac{4608LR^4}{\gamma^2\lambda^2(K+1)^2}\}$, which concludes the proof.
\hfill$\square$

\section{Lower Bounds for a Fixed Clipping Level: Full Statement, Comparison and Proofs}\label{app:lower}

\subsection{Full Statement and Comparison with Existing Results}\label{app:lower-statement}

{In Table~\ref{tab:prescribed}, we compare Theorems~\ref{thm:fixed-main}, \ref{thm:any-main}, and~\ref{thm:lower-main} with \citet{koloskova2023revisiting} and \citet{vatankhah2026differentially}; see also the comparison in Section~\ref{sec:lower}. The table separates the output criterion, the neighborhood, and the terms that vanish with the horizon. We then give the full statement of Theorem~\ref{thm:lower-main}, including the growth of the expected error with the total stepsize.}

{Here $R\ge\|x^0-x^*\|$, $x^t$ are the iterates with $\lambda_k\equiv\lambda$, and $\bar x^K$ is their average. The noise is additive ($\sigma_1=0$), with level $\sigma_0$ and index $\alpha\in(1,2]$, under Assumption~\ref{as:moment} or~\ref{as:weak}. Unless a note states otherwise, the upper bounds hold with probability at least $1-\beta$, with a stepsize chosen for the horizon $K$. The criterion is bounded by the neighborhood plus the vanishing terms. The vanishing terms suppress constants depending on $\alpha$ and logarithmic factors in $K$ and $1/\beta$.}

{Our neighborhood terms retain their constants: $b=b_\alpha^{\rm mom}\in[1/2,1]$ under Assumption~\ref{as:moment}, and $b=b_\alpha=2^{\alpha-1}/(\alpha-1)$ under Assumption~\ref{as:weak}. The constant $c_\perp=2^{\alpha-1}(\kappa_2+3\kappa_1)$ is defined in Section~\ref{sec:planned} and depends only on $\alpha$.}

\begin{table}[tp]
\caption{{Bounds for \clipsgd with a prescribed clipping level under additive noise. Our results are highlighted in light blue.}}\label{tab:prescribed}
\begin{center}
\scriptsize
\begin{threeparttable}
\renewcommand{\arraystretch}{1.6}
\setlength{\tabcolsep}{3pt}
\begin{tabular}{lllll}
\textbf{Result} & \textbf{Level} & \textbf{Criterion} & \textbf{Neighborhood term} & \textbf{Vanishing terms}\\
\hline
\stk{\citet{koloskova2023revisiting},\\Thm.~3.3\tnote{(1)}} & any $\lambda>0$ & $\min_{t\le K}\E\|\nabla f(x^t)\|$ & $O(\min\{\sigma_0,\sigma_0^2/\lambda\})$ & \stk{$(LF_0\sigma_0^2/K)^{1/4}$\\${}+\max\{1,\sigma_0/\lambda\}\sqrt{LF_0/K}$\\${}+LF_0/(\lambda K)$}\\
\citet{vatankhah2026differentially}\tnote{(2)} & $\lambda>4LR$ & $\min_{t\le K}f(x^t)-f^*$ & $O(R\sigma_0^\alpha/\lambda^{\alpha-1})$ & $\lambda R/K+R\sqrt{\lambda^{2-\alpha}\sigma_0^\alpha/K}$\\
\rowcolor{blue!10} Theorem~\ref{thm:fixed-main} & $\lambda\ge2\sqrt2LR$ & $f(\bar x^K)-f^*$ & $120\,b\,R\sigma_0^\alpha/\lambda^{\alpha-1}$ & $\lambda R/K+R\sqrt{\lambda^{2-\alpha}\sigma_0^\alpha/K}$\\
\rowcolor{blue!10} Theorem~\ref{thm:any-main}\tnote{(3)} & any $\lambda>0$ & $\min_{t\le K}f(x^t)-f^*$ & \stk{$48\sqrt2\,c_\perp R\times{}$\\$\min\{\sigma_0,\sigma_0^\alpha/\lambda^{\alpha-1}\}$} & \stk{$(LR^2+\lambda R)/K$\\${}+R\sqrt{\lambda^{2-\alpha}\sigma_0^\alpha/K}$\\${}+L^3R^4/(\lambda^2K^2)$}\\
\stk{\citet{koloskova2023revisiting},\\Thms.~3.1--3.2\tnote{(4)}} & any $\lambda>0$ & \stk{lower bound: $\|\nabla f(\hat x)\|$\\at the fixed points $\hat x$} & $\ge\min\{\sigma_0/12,\,\sigma_0^2/(6\lambda)\}$ & ---\\
\rowcolor{blue!10} Theorem~\ref{thm:lower}\tnote{(5)} & any $\lambda>0$ & \stk{lower bound: $\E f(\hat x)-f^*$,\\$\hat x$ an iterate or a\\deterministic average;\\deterministic stepsizes} & \stk{$\ge\frac{1}{48\cdot4^\alpha}\times{}$\\$\min\{R\sigma_0^\alpha/\lambda^{\alpha-1},R\sigma_0,LR^2\}$} & ---\\
\rowcolor{blue!10} \stk{Thm.~\ref{thm:fixed-main} with\tnote{(6)}\\$\lambda=\lambda_0(K+1)^{1/\alpha}$} & any $\lambda_0>0$ & $f(\bar x^K)-f^*$ & --- & \stk{$\big(\lambda_0+\sigma_0^\alpha\lambda_0^{1-\alpha}$\\${}+\sigma_0^{\alpha/2}\lambda_0^{1-\alpha/2}\big)$\\${}\times R\,K^{-(\alpha-1)/\alpha}$}\\
\end{tabular}
{\footnotesize{\begin{tablenotes}[flushleft]
\item[(1)] {Expected gradient norm for non-convex smooth functions under bounded variance ($\alpha=2$).}
\item[(2)] {Under Assumption~\ref{as:moment}; simplified using the initial-point bound, as explained in the text.}
\item[(3)] Under Assumption~\ref{as:weak}, for $\alpha<2$.
\item[(4)] {Gradient-norm lower bound at fixed points of the mean clipped field under bounded variance.}
\item[(5)] {Expected objective gap for every iterate and every convex average with deterministic weights; the text also gives the simultaneous probability comparison.}
\item[(6)] {The level depends on the horizon and stays fixed within each run; the horizon condition is stated in the text.}
\end{tablenotes}}}
\end{threeparttable}
\end{center}
\end{table}

{The first row specializes Theorem~3.3 of \citet{koloskova2023revisiting}, which treats non-convex $(L_0,L_1)$-smooth functions under $\E\|g(x,\xi)-\nabla f(x)\|^2\le\sigma_0^2$, to $L_0=L$ and $L_1=0$. It uses $\eta=\min\{1/(9L),(F_0/(L\sigma_0^2K))^{1/2}\}$, where $F_0:=f(x^0)-f^*\le LR^2/2$ in our setting. Their upper and lower bounds characterize the gradient-norm neighborhood $\min\{\sigma_0,\sigma_0^2/\lambda\}=\bar b_\lambda$ up to absolute constants.}

{The second row uses the full non-private bound of \citet{vatankhah2026differentially}. With $u:=R^2/(\gamma(K+1))$, their bound is $4u+64Lu^2/\lambda^2$. The best-iterate gap is also at most $f(x^0)-f^*\le LR^2/2$. For $\lambda>4LR$,}
\[
    {\min\left\{\frac{LR^2}{2},\,4u+\frac{64Lu^2}{\lambda^2}\right\}\le8u.}
\]
{Indeed, if $u\le LR^2$, then $64Lu^2/\lambda^2\le4u$; otherwise, the first term in the minimum suffices. Choosing $\gamma$ as the minimum of their full stepsize bounds, with no privacy noise and $\zeta_\lambda=0$, gives}
\[
    {u=O\left(\frac{\lambda R\ell}{K+1}+\frac{R\sigma_0^\alpha}{\lambda^{\alpha-1}}+R\sqrt{\frac{\lambda^{2-\alpha}\sigma_0^\alpha\ell}{K+1}}\right).}
\]
{This yields the displayed row, including its finite-horizon terms. In particular, the additional quadratic neighborhood term $LR^2\sigma_0^{2\alpha}/\lambda^{2\alpha}$ in their unsimplified bound is absorbed by this argument.}

{Theorems~3.1--3.2 of \citet{koloskova2023revisiting} construct, for every $\lambda>0$, a one-dimensional strongly convex quadratic with $L=1$ and two-point noise of variance $\sigma_0^2$, for which every fixed point $\hat x$ of the mean clipped field, $\E\clip(g(\hat x,\xi),\lambda)=0$, satisfies the displayed gradient bound. This is a statement about the fixed points governing the small-stepsize limit. The finite-horizon upper bound concerns $\min_{t\le K}\E\|\nabla f(x^t)\|$. The factor $R$ in our function-value neighborhood reflects the linear growth of the Huber objective away from its quadratic region. For $f(x)=Lx^2/2$, a gradient norm of magnitude $B$ corresponds to a function gap $B^2/(2L)$.}

{Theorem~\ref{thm:lower} gives the expected lower bound for every iterate and every convex combination with deterministic weights, under any positive deterministic stepsize sequence. If $0<\gamma_k\le R/(16\lambda)$ for all $k$, then with probability at least $3/4$ all iterates and all their convex averages simultaneously have error at least half the displayed bound. This also bounds the best iterate on the same event. We use this probability statement to compare with the high-probability upper bounds, whose stepsizes satisfy the required cap. Consequently, the neighborhood $R\min\{\sigma_0,\sigma_0^\alpha/\lambda^{\alpha-1}\}$ is unavoidable up to a factor depending only on $\alpha$ until it reaches the natural scale $LR^2$. Above that scale, returning the initial point already gives $f(x^0)-f^*\le LR^2/2$; the lower bound therefore includes the truncation by $LR^2$.}

{In the final row, $\lambda=\lambda_0(K+1)^{1/\alpha}$ grows with the horizon but remains constant during a run. The condition $K+1\ge(2\sqrt2LR/\lambda_0)^\alpha$ ensures that Theorem~\ref{thm:fixed-main} applies. This choice recovers the rate in Table~\ref{tab:complexity}.}

% Revision: allow the lower-bound theorem to follow the comparison without a forced page break.
\begin{theorem}[Lower bound for a fixed clipping level: full statement of Theorem~\ref{thm:lower-main}]\label{thm:lower}
Let $L,R,\sigma_0,\lambda>0$ and $\alpha\in(1,2]$, and let
\begin{gather*}
    \varepsilon_\lambda:=\frac{1}{48\cdot4^\alpha}\min\left\{\frac{R\sigma_0^\alpha}{\lambda^{\alpha-1}},\,R\sigma_0,\,LR^2\right\},\\
    s_\lambda:=\frac{1}{4^{\alpha+1}}\min\left\{\frac{\sigma_0^\alpha}{\lambda^{\alpha-1}},\,\lambda\right\}.
\end{gather*}
For $\lambda\ge\sigma_0$, the first minimum is $\min\{R\sigma_0^\alpha/\lambda^{\alpha-1},LR^2\}$ and $s_\lambda=\sigma_0^\alpha/(4^{\alpha+1}\lambda^{\alpha-1})$; in the notation of~\eqref{eq:bias-scale}, $\varepsilon_\lambda=\frac{1}{48\cdot4^\alpha}\min\{R\bar b_\lambda,LR^2\}$ and $s_\lambda=\min\{\sigma_0,\lambda\}^\alpha/(4^{\alpha+1}\lambda^{\alpha-1})$.
There exist a convex $L$-smooth $f:\R\to\R$ with minimizer $x^*=0$, a point $x^0$ with $|x^0-x^*|=R$, and an oracle $g(x,\xi)=f'(x)+\zeta$ with $\E\zeta=0$ and $\E|\zeta|^\alpha\le\sigma_0^\alpha$ (i.e., Assumptions~\ref{as:smooth}--\ref{as:weak} hold on $\R$ with $\sigma_1=0$) such that \clipsgd with $\lambda_k\equiv\lambda$ has the following properties.
\begin{enumerate}
    \item {For any deterministic stepsizes $\gamma_k>0$ (used instead of $\gamma$ at iteration $k$) and all $k,K\in\{0,1,\dots\}$,}
    \[
        \E f(x^k)-f^*\ge\varepsilon_\lambda\left(1+\frac{s_\lambda}{R}\sum_{j<k}\gamma_j\right),
    \]
    and $\E f(\hat x)-f^*\ge\varepsilon_\lambda$ for every convex combination $\hat x$ of $x^0,\dots,x^K$ with deterministic weights. For a constant stepsize $\gamma$, moreover, $\E f(\bar x^K)-f^*\ge\varepsilon_\lambda\left(1+\frac{s_\lambda\gamma K}{2R}\right)$.
    \item {For any deterministic stepsizes $0<\gamma_k\le\frac{R}{16\lambda}$, with probability at least $3/4$, $f(x^k)-f^*\ge\varepsilon_\lambda/2$ for all $k\in\{0,1,\dots\}$ and $f(\bar x^K)-f^*\ge\varepsilon_\lambda/2$ for all $K\in\{0,1,\dots\}$.}
\end{enumerate}
\end{theorem}

\paragraph{Two further consequences.} First, due to the linear growth of the expected error (part~1 of Theorem~\ref{thm:lower}), to keep the expected error of the last iterate, or of the averaged iterate with a constant stepsize, of the order $\varepsilon_\lambda$, one needs the total stepsize to be $O_\alpha(R\lambda^{\alpha-1}/\sigma_0^\alpha)$ when $\sigma_0\le\lambda$ (for the averaged iterate under a general deterministic schedule, the proof controls the weighted sum $\frac{1}{K+1}\sum_{j<K}(K-j)\gamma_j$ instead of the total stepsize, since a late step moves only a few of the averaged iterates), which is exactly the scale of the bias condition in Theorem~\ref{thm:fixed-main} (there, the bias threshold for the stepsize is $R\lambda^{\alpha-1}/(60b\sigma_0^\alpha(K+1))$ for $\sigma_1=0$, and a total stepsize $\gamma(K+1)$ much larger than $R\lambda^{\alpha-1}/(60b\sigma_0^\alpha)$ would increase the error); the best iterate in Theorem~\ref{thm:any-main} does not require this condition, and part~2 of Theorem~\ref{thm:lower} covers its stepsizes. Second, the bias is not caused by the linear growth of the Huber function: for $f(x)=\frac L2x^2$ and the noise from the first construction, \clipsgd with $\lambda_k\equiv\lambda$ and any stepsizes with $\sum_k\gamma_k=\infty$ and $\sum_k\gamma_k^2<\infty$ converges almost surely to a point $x_\lambda\ne x^*$ with $f(x_\lambda)-f^*\ge\min\{\sigma_0^{2\alpha}/\lambda^{2(\alpha-1)},\lambda^2\}/(8\cdot16^\alpha L)$ (see Proposition~\ref{prop:quadratic} in Section~\ref{app:quadratics}).

\subsection{Proof of Theorem~\ref{thm:lower}}

\begin{proof}
\emph{Huber functions.} For $s>0$, let
\[
    f_s(x):=\begin{cases}\frac L2x^2,&|x|\le s/L,\\ s|x|-\frac{s^2}{2L},&|x|>s/L.\end{cases}
\]
Then $f_s$ is convex and $L$-smooth on $\R$, and $f_s'(x)=\max\{-s,\min\{Lx,s\}\}$, implying $|f_s'|\le s$. Moreover, $x^*=0$ is the unique minimizer, $f^*=0$, and $f_s$ is even and non-decreasing in $|x|$. Assumption~\ref{as:grad-bound} holds on $\R$: $(f_s'(x))^2=2Lf_s(x)$ for $|x|\le s/L$, and $s^2\le2Ls|x|-s^2=2Lf_s(x)$ otherwise. Moreover, for every $r\ge0$, we have
\begin{equation}\label{eq:huber-value}
    f_s(-r)=f_s(r)\ge\min\left\{\frac{sr}{2},\frac{Lr^2}{2}\right\},
\end{equation}
since $sr-\frac{s^2}{2L}\ge\frac{sr}{2}$ for $r\ge s/L$. We use two instances depending on $\lambda$. In both of them, $f=f_s$ for a suitable $s$, and the noise satisfies $\E\zeta=0$ and $\E|\zeta|^\alpha\le\sigma_0^\alpha$. Thus, Assumption~\ref{as:moment} holds on $\R$ with $(\sigma_0,0)$, and so does Assumption~\ref{as:weak} by Proposition~\ref{prop:embeddings}. The iterates use independent copies $\zeta_k$ of $\zeta$.

\emph{Instance~A: $\lambda\ge\sigma_0/(2\cdot3^{1/\alpha})$, $x^0=-R$.} Let $\sigma_\lambda:=\min\{\sigma_0,\lambda\}$ and
\[
    p:=\frac12\left(\frac{\sigma_\lambda}{4\lambda}\right)^\alpha,\qquad \nu:=\frac{2p\lambda}{1-p},\qquad s:=\frac{p\lambda}{2}.
\]
Since $\sigma_\lambda\le\lambda$ and $\alpha>1$, we have $p\le\frac18$. Let $\zeta$ be equal to $\nu$ with probability $1-p$ and to $-2\lambda$ with probability $p$. Then $\E\zeta=(1-p)\nu-2p\lambda=0$ and $p(2\lambda)^\alpha=\sigma_\lambda^\alpha/2^{\alpha+1}\le\sigma_\lambda^\alpha/4$. Since $1-p\ge\frac78$, $4^\alpha\ge4$, and $\sigma_\lambda\le\lambda$, we have
\[
    \nu\le\frac{16}{7}p\lambda=\frac87\cdot\frac{\sigma_\lambda}{4^\alpha}\left(\frac{\sigma_\lambda}{\lambda}\right)^{\alpha-1}\le\frac27\sigma_\lambda ,
\]
hence, $\E|\zeta|^\alpha=(1-p)\nu^\alpha+p(2\lambda)^\alpha\le\left(\frac27+\frac14\right)\sigma_\lambda^\alpha\le\sigma_0^\alpha$. Since $s\le\lambda/16$ and $\nu\le2\lambda/7$, for every $x\in\R$ we have $f'(x)+\nu\in(-\lambda,\lambda)$ and $f'(x)-2\lambda<-\lambda$. Thus, $\clip(f'(x)+\zeta,\lambda)$ equals $f'(x)+\nu$ with probability $1-p$ and $-\lambda$ with probability $p$. Using $(1-p)\nu=2p\lambda$ and $|f'(x)|\le s$, we get
\begin{equation}\label{eq:mean-field}
    h(x):=\E\clip(f'(x)+\zeta,\lambda)=(1-p)(f'(x)+\nu)-p\lambda=(1-p)f'(x)+p\lambda\ge p\lambda-s=s .
\end{equation}
Moreover, the variance of $\clip(f'(x)+\zeta,\lambda)$ is $p(1-p)(f'(x)+\nu+\lambda)^2\le p\left(\frac{1}{16}+\frac27+1\right)^2\lambda^2\le2p\lambda^2$. Finally, $s=\frac{\sigma_\lambda^\alpha}{4^{\alpha+1}\lambda^{\alpha-1}}=s_\lambda$ and $\frac{sR}2=\frac{R\sigma_\lambda^\alpha}{8\cdot4^\alpha\lambda^{\alpha-1}}$. If $\lambda\ge\sigma_0$, then this equals $\frac{R\bar b_\lambda}{8\cdot4^\alpha}$. Otherwise, $\sigma_\lambda=\lambda\ge\sigma_0/6$ and $\bar b_\lambda=\sigma_0$, and, therefore, it is at least $\frac{R\bar b_\lambda}{48\cdot4^\alpha}$. In both cases, $\frac{sR}{2}\ge\frac{R\bar b_\lambda}{48\cdot4^\alpha}$.

\emph{Instance~B: $\lambda<\sigma_0/(2\cdot3^{1/\alpha})$, $x^0=R$.} Let $a:=\sigma_0/3^{1/\alpha}$ (meaning that $\lambda<a/2$) and $s:=a-\lambda>a/2$. Let $\zeta$ be equal to $-a$ with probability $3/4$ and to $3a$ with probability $1/4$. Then $\E\zeta=0$ and $\E|\zeta|^\alpha=\frac{3+3^\alpha}{4}a^\alpha\le3a^\alpha=\sigma_0^\alpha$. For every $x\in\R$, we have $f'(x)-a\le s-a=-\lambda$ and $f'(x)+3a\ge3a-s=2a+\lambda\ge\lambda$. Thus, the clipped oracle equals $-\lambda$ with probability $3/4$ and $\lambda$ with probability $1/4$, independently of $x$, and $x^{k+1}=x^k+\gamma_k\lambda\chi_k$ with i.i.d.\ $\chi_k\in\{-1,1\}$, $\PP\{\chi_k=1\}=\frac34$. In particular, $\E[x^{k+1}\mid\G_k]=x^k+\gamma_k\lambda/2$, and $\lambda/2\ge\lambda/4^{\alpha+1}=s_\lambda$. Finally, $\bar b_\lambda=\sigma_0$ and $\frac{sR}2\ge\frac{aR}4\ge\frac{R\sigma_0}{12}\ge\frac{R\bar b_\lambda}{48\cdot4^\alpha}$.

\emph{Value of the bound.} In both instances, $\frac{sR}{2}\ge\frac{R\bar b_\lambda}{48\cdot4^\alpha}$ and $\frac{LR^2}{2}\ge\frac{LR^2}{48\cdot4^\alpha}$, and, therefore, \eqref{eq:huber-value} gives $f(r)\ge\varepsilon_\lambda$ for every $r\ge R$, and $f(R/2)\ge\min\{\frac{sR}{4},\frac{LR^2}{8}\}\ge\frac{\varepsilon_\lambda}{2}$.

\emph{Part~1.} Let $\Gamma_k:=\sum_{j<k}\gamma_j$. Since $|x^{k+1}-x^k|\le\gamma_k\lambda$, for each fixed $k$ we have $|x^k|\le|x^0|+\lambda\Gamma_k$; thus, each $x^k$ is bounded by a deterministic constant and, therefore, integrable, and the same holds for every finite convex combination of the iterates. In instance~A, \eqref{eq:mean-field} gives $\E[x^{k+1}\mid\G_k]\le x^k-\gamma_ks$; hence, $\E x^k\le-R-s_\lambda\Gamma_k$. In instance~B, we have $\E x^k=R+\lambda\Gamma_k/2\ge R+s_\lambda\Gamma_k$. In both cases, $|\E x^k|\ge R+s_\lambda\Gamma_k$ with $\E x^k$ on the side of $x^0$. Thus, $|\E\hat x|\ge R$ for every convex combination $\hat x$ of the iterates with deterministic weights, and $|\E\bar x^K|\ge R+s_\lambda\gamma K/2$ for a constant stepsize $\gamma$, since $\frac1{K+1}\sum_{k=0}^K\Gamma_k=\gamma K/2$. Since $f$ is convex, even, and non-decreasing in $|x|$, Jensen's inequality gives $\E f(\hat x)\ge f(R)\ge\varepsilon_\lambda$, $\E f(x^k)\ge f(R+s_\lambda\Gamma_k)$, and, for a constant stepsize, $\E f(\bar x^K)\ge f(R+s_\lambda\gamma K/2)$. By~\eqref{eq:huber-value} with $r=R(1+s_\lambda\Gamma_k/R)$ (and, similarly, with $r=R+s_\lambda\gamma K/2$), we get
\[
    f(R+s_\lambda\Gamma_k)\ge\min\left\{\frac{sR}{2},\frac{LR^2}{2}\right\}\left(1+\frac{s_\lambda\Gamma_k}{R}\right)\ge\varepsilon_\lambda\left(1+\frac{s_\lambda\Gamma_k}{R}\right).
\]

\emph{Part~2.} Let $\bar\gamma:=\sup_k\gamma_k\le R/(16\lambda)$. We show that, with probability at least $1-e^{-2}\ge\frac34$, all iterates satisfy $|x^k|\ge R/2$ and belong to the side of $x^0$. Then all averages $\bar x^K$ belong to it as well, and $f(x^k)-f^*\ge f(R/2)\ge\varepsilon_\lambda/2$ and $f(\bar x^K)-f^*\ge\varepsilon_\lambda/2$.

In instance~A, let $e_k:=\clip(f'(x^k)+\zeta_k,\lambda)-h(x^k)$ and $D_k:=\sum_{j<k}\gamma_je_j$. Then $x^k=-R-\sum_{j<k}\gamma_jh(x^j)-D_k\le-R-s\Gamma_k-D_k$. Let $\vartheta:=1/(4\bar\gamma\lambda)$. Since $|\vartheta\gamma_je_j|\le\frac12$ and $e^v\le1+v+v^2$ for $|v|\le1$, we have $\E[e^{-\vartheta\gamma_je_j}\mid\G_j]\le1+\vartheta^2\gamma_j^2\E[e_j^2\mid\G_j]\le\exp(2p\vartheta^2\gamma_j^2\lambda^2)$. Thus, $M_k:=\exp\left(-\vartheta D_k-2p\vartheta^2\lambda^2\sum_{j<k}\gamma_j^2\right)$ is a non-negative supermartingale with $M_0=1$. On the event $\{x^k>-R/2\}$, we have $-D_k>\frac R2+s\Gamma_k$. Since $\sum_{j<k}\gamma_j^2\le\bar\gamma\Gamma_k$ and $2p\vartheta\lambda^2\bar\gamma=\frac{p\lambda}2=s$, on this event we have
\[
    M_k\ge\exp\left(\frac{\vartheta R}2+\vartheta\Gamma_k\left(s-2p\vartheta\lambda^2\bar\gamma\right)\right)=e^{\vartheta R/2}.
\]
By Ville's maximal inequality, $\PP\{\exists k:\ x^k>-R/2\}\le e^{-\vartheta R/2}=e^{-R/(8\bar\gamma\lambda)}\le e^{-2}$.

In instance~B, let $\kappa:=\ln3/(\bar\gamma\lambda)$. The function $\varphi(v):=\frac34e^{-v}+\frac14e^{v}$ is convex with $\varphi(0)=\varphi(\ln3)=1$; hence, $\E e^{-\kappa\gamma_k\lambda\chi_k}=\varphi(\kappa\gamma_k\lambda)\le1$. Therefore, $N_k:=\exp(-\kappa(x^k-R/2))$ is a non-negative supermartingale, and $\PP\{\exists k:\ x^k<R/2\}\le\PP\{\sup_kN_k>1\}\le N_0=3^{-R/(2\bar\gamma\lambda)}\le3^{-8}\le e^{-2}$, which concludes the proof.
\end{proof}

\subsection{Optimality of the Stochastic Term}\label{app:lower-rate}

\begin{proposition}[Lower bound for the stochastic term]\label{prop:lower-rate}
{Let $L,R,\sigma_0>0$, $\alpha\in(1,2]$, $K\in\{0,1,\dots\}$, and}
\[
    s:=\min\left\{\frac{\sigma_0}{2(K+1)^{(\alpha-1)/\alpha}},\ LR\right\}.
\]
Consider any algorithm that makes at most $K+1$ adaptive, possibly randomized, queries to a stochastic gradient oracle and outputs a point $\hat x\in\R$; the algorithm accesses the unknown instance only via this oracle, without access to exact function values, exact gradients, or an analytical description of the instance. Then there exist two instances $v\in\{-1,+1\}$ of problem~\eqref{eq:problem} on $\R$, with $f_v$ convex and $L$-smooth on $\R$, minimizer $x^*_v=vR$, $f_v^*=0$, common starting point $x^0=0$ (i.e., $|x^0-x^*_v|=R$), and unbiased oracles $\nabla f_{v,\xi}$ satisfying Assumption~\ref{as:moment}, and, therefore, Assumption~\ref{as:weak}, on $\R$ with the parameters $(\sigma_0,0)$ such that, for at least one $v$,
\[
    \PP\left\{f_v(\hat x)-f_v^*\ge\frac{sR}{2}\right\}\ge\frac14,\qquad\text{where}\qquad \frac{sR}{2}=\frac14\min\left\{\frac{\sigma_0R}{(K+1)^{(\alpha-1)/\alpha}},\ 2LR^2\right\}.
\]
\end{proposition}

\begin{proof}
Let $p:=\frac{1}{2(K+1)}$ and let $f_s$ be the Huber function from the proof of Theorem~\ref{thm:lower}, with $f_s'(y)=\max\{-s,\min\{Ly,s\}\}$. For $v\in\{-1,+1\}$, let $f_v(x):=f_s(x-vR)$, which is convex and $L$-smooth on $\R$, has the unique minimizer $vR$ with $f_v^*=0$, and satisfies $|f_v'|\le s$. We write $f_v'=m+v\delta$ with $m:=\frac12(f_+'+f_-')$ and $\delta:=\frac12(f_+'-f_-')$, where $|\delta|\le s$, and, similarly, $f_v=M+vD$ with $M:=\frac12(f_++f_-)$ and $D:=\frac12(f_+-f_-)$, where $M'=m$ and $D'=\delta$. Let $\xi=U\sim\mathrm{Uniform}(0,1)$ and
\[
    f_{v,U}(x):=M(x)+vD(x)\,\frac{\ind_{\{U\le p\}}}{p},\qquad
    \nabla f_{v,U}(x)=m(x)+v\delta(x)\,\frac{\ind_{\{U\le p\}}}{p}.
\]
Then $\E_Uf_{v,U}=f_v$ and $\E_U\nabla f_{v,U}(x)=f_v'(x)$. The noise $\nabla f_{v,U}(x)-f_v'(x)=v\delta(x)\big(\frac{\ind_{\{U\le p\}}}{p}-1\big)$ equals $-v\delta(x)$ with probability $1-p$ and $v\delta(x)\frac{1-p}{p}$ with probability $p$; hence,
\begin{align*}
    \E_U\left|\nabla f_{v,U}(x)-f_v'(x)\right|^\alpha&=|\delta(x)|^\alpha\left[(1-p)+p\left(\frac{1-p}{p}\right)^\alpha\right]\le|\delta(x)|^\alpha\left(1+p^{1-\alpha}\right)\\
    &\le2s^\alpha p^{1-\alpha}=2s^\alpha\big(2(K+1)\big)^{\alpha-1}\le\sigma_0^\alpha,
\end{align*}
where $1+p^{1-\alpha}\le2p^{1-\alpha}$ since $p<1<\alpha$, and the last step follows from the choice of $s$. Therefore, both instances satisfy Assumption~\ref{as:moment} on $\R$ with $(\sigma_0,0)$, and Assumption~\ref{as:weak} by Markov's inequality.

Next, we run the algorithm on both instances with the same internal randomness and the same i.i.d.\ uniforms $U_0,\dots,U_K$ (one per query). On the event $E:=\{U_k>p\ \text{for all }k\le K\}$, every query at any point $x$ returns $m(x)$ under both instances, meaning that the two runs produce the same query points and the same output $\hat x$. By Bernoulli's inequality, $\PP(E)=(1-p)^{K+1}\ge1-(K+1)p=\frac12$. Since $\PP(E\cap\{\hat x\le0\})+\PP(E\cap\{\hat x>0\})=\PP(E)\ge\frac12$, one of the two probabilities is at least $\frac14$. If $\PP(E\cap\{\hat x\le0\})\ge\frac14$, we take $v=+1$: for $x\le0$, we have $x-R\le-R\le-s/L$, since $s\le LR$, and, therefore, $f_+(x)=s(R-x)-\frac{s^2}{2L}\ge sR-\frac{s^2}{2L}\ge\frac{sR}{2}$. Otherwise, $\PP(E\cap\{\hat x>0\})\ge\frac14$, and, for $v=-1$ and $x>0$, we have $f_-(x)=s(x+R)-\frac{s^2}{2L}\ge\frac{sR}{2}$. In both cases, the event in question is determined by the coupled run, which coincides with the actual run of instance $v$ on $E$, which concludes the proof.
\end{proof}

Proposition~\ref{prop:lower-rate} establishes the optimal polynomial dependence of the additive stochastic term $\Delta_0$ from Corollary~\ref{cor:rate} on $K$, $\sigma_0$, and $R$ in the regime $\sigma_0R(K+1)^{-(\alpha-1)/\alpha}\lesssim LR^2$, at a fixed confidence level with failure probability below $1/4$, within the class of algorithms using $K+1$ oracle calls; \clipsgd with $K+1$ iterations belongs to this class. For the moment class, the upper and lower bounds match up to constants and the factor $\ell^{(\alpha-1)/\alpha}$; at the weak-$L^2$ endpoint, the upper bound also contains the tail-dependent logarithmic factor $c_2(\Lambda_K)^{1/2}$. It is worth mentioning that the proposition itself does not establish the necessity of these logarithms, the optimal dependence on the failure probability, or the optimality of the multiplicative term; the dependence on the failure probability is studied in Proposition~\ref{prop:confidence-lower} below, the necessity of the logarithm at $\alpha=2$ at fixed confidence is shown in Proposition~\ref{prop:lorentz-opt-lower}, and nothing is claimed for the multiplicative term. The hard instances have bounded gradients, $|f_v'|\le s$, and state-dependent noise (its law depends on $x$ through $\delta(x)$, which is allowed by Assumption~\ref{as:moment}); therefore, the lower bound also holds within the subclass of problems with bounded gradients. That is, the proposition concerns a uniformly bounded noise scale ($\sigma_1=0$) rather than state-independent additive noise; in contrast, the Lorentz lower bound from Proposition~\ref{prop:lorentz-opt-lower} (see Appendix~\ref{app:lorentz}) is realized by state-independent noise. Related lower bounds are classical in expectation \citep{nemirovski1983problem,vural2022mirror}, and \citet{liu2026clipped} recently proves high-probability lower bounds for non-smooth convex problems under bounded $\alpha$-th moments, of the order $(\sigma_s^{2/\alpha-1}\sigma_l^{2-2/\alpha}+\sigma_s\ln^{1-1/\alpha}(1/\delta))D/T^{1-1/\alpha}$ for $T$ oracle calls and failure probability $\delta$, where $\sigma_s$ and $\sigma_l$ bound the directional and the full $\alpha$-th moment of the noise and $D$ is the initial distance; in dimension one, where $\sigma_s=\sigma_l$, this coincides with the dependence on $K$, $\sigma_0$, $R$, and $\beta$ in Propositions~\ref{prop:lower-rate} and~\ref{prop:confidence-lower}, which are derived for the smooth unconstrained problem and have explicit constants, while the dimension-dependent improvements via $\sigma_s$ have no counterpart in our setting.

The same construction with a confidence-dependent revelation probability gives the dependence on the confidence level.

\begin{proposition}[Dependence on the failure probability]\label{prop:confidence-lower}
{Let $L,R,\sigma_0>0$, $\alpha\in(1,2]$, $\beta\in(0,1/8]$, $\ell_\beta:=\ln\frac1{4\beta}$, and let $K\in\{0,1,\dots\}$ satisfy $K+1\ge\ell_\beta$. For every algorithm as in Proposition~\ref{prop:lower-rate}, there exist two instances $v\in\{-1,+1\}$ with the properties listed there (Assumption~\ref{as:moment} with $(\sigma_0,0)$ on $\R$, $|x^0-x_v^*|=R$) such that, for at least one $v$,}
\[
    \PP\left\{f_v(\hat x)-f_v^*\ge\frac14\min\left\{LR^2,\ \sigma_0R\left(\frac{\ell_\beta}{K+1}\right)^{\frac{\alpha-1}{\alpha}}\right\}\right\}>2\beta .
\]
That is, a guarantee holding with probability at least $1-\beta'$ for some $\beta'\le1/4$ and $K+1\ge\ln\frac{1}{2\beta'}$ must allow an error of at least $\frac14\min\{LR^2,\sigma_0R(\ln\frac1{2\beta'}/(K+1))^{(\alpha-1)/\alpha}\}$ on one of these instances; in other words, the factor $(\ln(1/\beta)/K)^{(\alpha-1)/\alpha}$ in the additive term in~\eqref{eq:improved-rates} cannot be improved in this range.
\end{proposition}

\begin{proof}
We repeat the proof of Proposition~\ref{prop:lower-rate} with the parameters $p:=\frac{\ell_\beta}{2(K+1)}\le\frac12$ and $s:=\min\{LR,\sigma_0(p^{\alpha-1}/2)^{1/\alpha}\}$; then $2s^\alpha p^{1-\alpha}\le\sigma_0^\alpha$ and both instances satisfy Assumption~\ref{as:moment} with $(\sigma_0,0)$. On the event $E=\{U_k>p\ \forall k\le K\}$, the two coupled runs coincide, and $\PP(E)=(1-p)^{K+1}>e^{-2p(K+1)}=e^{-\ell_\beta}=4\beta$, since $\ln(1-p)+2p$ vanishes at $0$ and is increasing on $(0,\frac12)$. Thus, $\PP(E\cap\{\hat x\le0\})>2\beta$ or $\PP(E\cap\{\hat x>0\})>2\beta$, and, on the respective instance, the error is at least $\frac{sR}2$, as before. Finally, $\frac{sR}2\ge\frac R2\cdot2^{-\frac1\alpha-\frac{\alpha-1}{\alpha}}\min\{LR,\sigma_0(\ell_\beta/(K+1))^{(\alpha-1)/\alpha}\}=\frac14\min\{LR^2,\sigma_0R(\ell_\beta/(K+1))^{(\alpha-1)/\alpha}\}$. The consequence follows with $\beta=\beta'/2$, which concludes the proof.
\end{proof}

\subsection{Strongly Convex Quadratics}\label{app:quadratics}

\begin{proposition}[Bias under strong convexity]\label{prop:quadratic}
Let $L,\sigma_0,\lambda>0$, $\alpha\in(1,2]$, and let $\zeta$, $\sigma_\lambda$, $p$, and $\nu$ be defined as in instance~A from the proof of Theorem~\ref{thm:lower} (these definitions make sense for every $\lambda>0$). Consider $f(x)=\frac L2x^2$ on $\R$ with the oracle $g(x,\xi)=f'(x)+\zeta$, which satisfies Assumptions~\ref{as:smooth}--\ref{as:weak} on $\R$ with $(\sigma_0,0)$. Let $x_\lambda:=-\frac{p\lambda}{(1-p)L}$. Then, for any $x^0\in\R$ and any deterministic stepsizes $\gamma_k>0$ with $\sum_k\gamma_k=\infty$ and $\sum_k\gamma_k^2<\infty$, \clipsgd with $\lambda_k\equiv\lambda$ satisfies $x^k\to x_\lambda$ almost surely, and
\[
    \lim_{k\to\infty}f(x^k)-f^*=\frac L2x_\lambda^2\ge\frac{\sigma_\lambda^{2\alpha}}{8\cdot16^\alpha L\lambda^{2(\alpha-1)}}=\frac{1}{8\cdot16^\alpha L}\min\left\{\frac{\sigma_0^{2\alpha}}{\lambda^{2(\alpha-1)}},\lambda^2\right\}\quad\text{almost surely}.
\]
\end{proposition}

The proof relies on the following standard result from stochastic approximation.

\begin{lemma}[{Scalar mean field with a unique attracting root}]\label{lem:scalar-sa}
Let $V:\R\times\mathcal Z\to\R$ be measurable with $|V|\le B$, let $\zeta_0,\zeta_1,\dots$ be i.i.d., and let $h(x):=\E V(x,\zeta_0)$ be continuous with $(x-z)h(x)>0$ for all $x\ne z$. Let $\gamma_k>0$ be deterministic with $\sum_k\gamma_k=\infty$ and $\sum_k\gamma_k^2<\infty$. Then, for every $x^0\in\R$, the recursion $x^{k+1}=x^k-\gamma_kV(x^k,\zeta_k)$ converges to $z$ almost surely.
\end{lemma}

\begin{proof}
{Continuity and the sign condition imply $h(z)=0$.} Let $\G_k:=\sigma(\zeta_0,\dots,\zeta_{k-1})$ and $W_k:=(x^k-z)^2$. For each fixed $k$, the bound $|x^k|\le|x^0|+B\sum_{j<k}\gamma_j$ implies that $W_k$ is integrable. Since $x^k$ is $\G_k$-measurable and $\zeta_k$ is independent of $\G_k$, we have
\[
    \E[W_{k+1}\mid\G_k]\le W_k-2\gamma_k(x^k-z)h(x^k)+\gamma_k^2B^2 .
\]
The last term is summable, and $(x^k-z)h(x^k)\ge0$. By the Robbins--Siegmund theorem \citep{robbins1971convergence}, $W_k$ converges almost surely to a finite limit $W_\infty$, and $\sum_k\gamma_k(x^k-z)h(x^k)<\infty$ almost surely; in particular, the whole trajectory is bounded almost surely. On the event $\{W_\infty>0\}$, for all sufficiently large $k$, the iterates stay in the compact set $\{x:\sqrt{W_\infty}/2\le|x-z|\le2\sqrt{W_\infty}\}$. On this set, the continuous positive function $(x-z)h(x)$ is bounded from below by some $\delta>0$, and, therefore, $\sum_k\gamma_k(x^k-z)h(x^k)=\infty$ since $\sum_k\gamma_k=\infty$, which is a contradiction. Thus, $W_\infty=0$ almost surely, which concludes the proof.
\end{proof}

\begin{proof}[Proof of Proposition~\ref{prop:quadratic}]
The noise coincides with the one from instance~A; therefore, the assumptions hold as shown there. Let $V(x,\zeta):=\clip(Lx+\zeta,\lambda)$, which satisfies $|V|\le\lambda$, and let $h(x):=\E V(x,\zeta)=(1-p)\clip(Lx+\nu,\lambda)+p\,\clip(Lx-2\lambda,\lambda)$. The function $h$ is continuous and non-decreasing, since both clipped terms are. We recall that $p\le\frac18$, $\nu\le\frac27\lambda$, and $(1-p)\nu=2p\lambda$. If $-\lambda-\nu\le Lx\le\lambda-\nu$, then $Lx+\nu\in[-\lambda,\lambda]$ and $Lx-2\lambda<-\lambda$, and, therefore,
\[
    h(x)=(1-p)(Lx+\nu)-p\lambda=(1-p)Lx+p\lambda .
\]
On this interval, $h$ is strictly increasing and vanishes exactly at $x_\lambda$, which belongs to the interval since $\frac{p\lambda}{1-p}\le\frac\lambda7<\lambda-\nu$. By monotonicity, $h(x)\ge h\left(\frac{\lambda-\nu}{L}\right)>0$ for $Lx>\lambda-\nu$, and $h(x)\le h\left(\frac{-\lambda-\nu}{L}\right)=-\lambda<0$ for $Lx<-\lambda-\nu$. Thus, $(x-x_\lambda)h(x)>0$ for all $x\ne x_\lambda$, and Lemma~\ref{lem:scalar-sa} gives $x^k\to x_\lambda$ almost surely. By continuity, $f(x^k)-f^*\to\frac L2x_\lambda^2\ge\frac{(p\lambda)^2}{2L}$. It remains to notice that $p\lambda=\frac{\sigma_\lambda^\alpha}{2\cdot4^\alpha\lambda^{\alpha-1}}$ and $\sigma_\lambda^{2\alpha}/\lambda^{2(\alpha-1)}=\min\{\sigma_0^{2\alpha}/\lambda^{2(\alpha-1)},\lambda^2\}$, which concludes the proof.
\end{proof}

\subsection{Strong Multiplicative Noise}

The oracle below is the multiplicative oracle from \citet[Proposition~13]{faw2023beyond} with $\epsilon=2$ and $\sigma_1=9$ in their notation, for which their Lemma~35 shows that clipped and normalized SGD with any constant stepsize make no progress with constant probability on a quadratic; the following proposition extends the clipping-failure argument to bounded deterministic stepsize sequences, with an explicit probability and divergence on the event.

\begingroup\mdfsetup{nobreak=true}
\begin{proposition}[Strong multiplicative noise]\label{prop:diverge}
Let $L,\lambda>0$, $\alpha\in(1,2]$, $f(x)=\frac L2x^2$ on $\R$, and $g(x,\xi):=Lx(1+3Z)$, where $\PP\{Z=-1\}=0.9$ and $\PP\{Z=9\}=0.1$. Then $\E g(x,\xi)=f'(x)$ and $\E|g(x,\xi)-f'(x)|^\alpha=\sigma_1^\alpha|f'(x)|^\alpha$ with $\sigma_1:=(0.9\cdot3^\alpha+0.1\cdot27^\alpha)^{1/\alpha}\in(5.4,9]$, meaning that Assumption~\ref{as:moment} holds on $\R$ with $\sigma_0=0$. {Let $x^0\ge\lambda/L$ and let $(\gamma_k)_{k\ge0}$ be any deterministic stepsizes with $0<\gamma_k\le1/L$ for all $k$.} Then, with probability at least $2/3$, \clipsgd with $\lambda_k\equiv\lambda$ and stepsizes $\gamma_k$ satisfies $x^k\ge\frac{\lambda}{2L}$, and, therefore, $f(x^k)-f^*\ge\frac{\lambda^2}{8L}$, for all $k\ge0$. If, in addition, $\sum_k\gamma_k=\infty$, then $x^k\to+\infty$ almost surely on this event.
\end{proposition}
\endgroup

{The stepsize caps in Theorems~\ref{thm:main}, \ref{thm:fixed}, and~\ref{thm:any-lambda} imply $\gamma\le1/L$. Their other assumptions exclude the above instance: the first two theorems require $\lambda\ge2\sqrt2LR$ with $R\ge|x^0|$, which is incompatible with $x^0\ge\lambda/L$, while Theorem~\ref{thm:any-lambda} requires $\sigma_1\le\bar\sigma_1$.}

\begin{proof}
Unbiasedness follows from $\E Z=0$, and $|g(x,\xi)-f'(x)|=3|Z||f'(x)|$ gives the moment identity; $\sigma_1=\|3Z\|_{L^\alpha}$ is non-decreasing in $\alpha$, equals $5.4$ at $\alpha=1$, and equals $9$ at $\alpha=2$. If $x\ge\frac{\lambda}{2L}$, then $g(x,\xi)\in\{-2Lx,28Lx\}$ with $2Lx\ge\lambda$, and, therefore, $\clip(g(x,\xi),\lambda)$ equals $-\lambda$ if $Z=-1$ and $\lambda$ if $Z=9$. Let $\chi_k:=1$ if $Z_k=-1$ and $\chi_k:=-1$ otherwise; these random variables are i.i.d.\ with $\PP\{\chi_k=1\}=0.9$. Let $\tau:=\inf\{k\ge0:x^k<\frac{\lambda}{2L}\}$. For $k<\tau$, we have $x^{k+1}=x^k+\gamma_k\lambda\chi_k$.

Next, let $\bar\gamma:=\sup_k\gamma_k\le1/L$ and $\kappa:=\ln9/(\bar\gamma\lambda)$. The function $\varphi(u):=0.9e^{-u}+0.1e^{u}$ is convex with $\varphi(0)=\varphi(\ln9)=1$; hence, $\varphi\le1$ on $[0,\ln9]$. Since $\kappa\gamma_k\lambda\in(0,\ln9]$, we have $\E e^{-\kappa\gamma_k\lambda\chi_k}=\varphi(\kappa\gamma_k\lambda)\le1$. Thus, $M_k:=\exp\left(-\kappa\left(x^{\min\{k,\tau\}}-\frac{\lambda}{2L}\right)\right)$ is a non-negative supermartingale, and $M_\tau>1$ on $\{\tau<\infty\}$. By the maximal inequality for non-negative supermartingales, we get
\[
    \PP\{\tau<\infty\}\le\PP\left\{\sup_kM_k\ge1\right\}\le M_0\le\exp\left(-\frac{\kappa\lambda}{2L}\right)=9^{-1/(2\bar\gamma L)}\le\frac13 .
\]
On $\{\tau=\infty\}$, we have $x^k=x^0+\lambda\sum_{j<k}\gamma_j\chi_j$ for all $k$. Let $\Gamma_k:=\sum_{j<k}\gamma_j\to\infty$. Since $\sum_{j\ge0}\gamma_j^2/\Gamma_{j+1}^2\le\bar\gamma\left(\frac1{\gamma_0}+\sum_{j\ge1}\left(\frac1{\Gamma_j}-\frac1{\Gamma_{j+1}}\right)\right)\le\frac{2\bar\gamma}{\gamma_0}$, the martingale $\sum_{j<k}\gamma_j(\chi_j-0.8)/\Gamma_{j+1}$ is bounded in $L^2$ and converges almost surely. By Kronecker's lemma, $\frac{1}{\Gamma_k}\sum_{j<k}\gamma_j(\chi_j-0.8)\to0$; hence, $\frac{1}{\Gamma_k}\sum_{j<k}\gamma_j\chi_j\to0.8$ and $x^k\to+\infty$ almost surely on $\{\tau=\infty\}$, which concludes the proof.
\end{proof}

\section{Proofs for Clip-SSTM}\label{app:sstm}

\subsection{Full Statements}\label{app:sstm-statements}

For a parameter $a\ge1$, the weights $\omega_{k+1}=\frac{k+2}{2aL}$ and $A_{k+1}=A_k+\omega_{k+1}$ from Algorithm~\ref{alg:clip-sstm} satisfy
\begin{equation}\label{eq:Ak}
    A_k=\frac{k(k+3)}{4aL},\qquad \frac{L\omega_{k+1}^2}{A_{k+1}}=\frac{(k+2)^2}{a(k+1)(k+4)}\le\frac1a ,
\end{equation}
since $(k+1)(k+4)-(k+2)^2=k\ge0$. We analyze \clipsstm under either noise model with $\alpha\in(1,2]$, with the clipping constants $(b,v)$ from Section~\ref{sec:results}: under Assumption~\ref{as:weak}, $v=v_\alpha(\Lambda_K^{\rm S})$, where
\begin{equation}\label{eq:Lambda-sstm}
    \Lambda_K^{\rm S}:=\max\left\{200\,b\,K,\ \Lambda_K^0\right\},\qquad
    \Lambda_K^0:=\left[\left(\frac{\sigma_0(K+1)}{17640\,LRH}\right)^{\!\alpha}+\left(\frac{\sigma_1}{3}\right)^{\!\alpha}\right]^{-1}
\end{equation}
($\Lambda_K^0:=0$ if $\sigma_0=\sigma_1=0$) plays the role of $\Lambda_K$ from Theorem~\ref{thm:main}, and $H:=\ln\frac{4K}{\beta}$ is defined as in Theorem~\ref{thm:sstm} below.

\begin{theorem}[\clipsstm: full statement of Theorem~\ref{thm:sstm-main}]\label{thm:sstm}
Let Assumptions~\ref{as:smooth}--\ref{as:convex} hold for $Q=B_{3R}(x^*)$, where $R\ge\|x^0-x^*\|>0$, and let either Assumption~\ref{as:moment} or Assumption~\ref{as:weak} hold with $\alpha\in(1,2]$. Let $K\ge1$ be an integer, $\beta\in(0,1]$, $H:=\ln\frac{4K}{\beta}$, $r:=\frac{\alpha-1}{\alpha}$, let $(b,v)$ be defined as above,
\begin{gather*}
    q_\alpha:=\min\left\{\frac{1}{1350\,v\,30^{\alpha-2}},\ \frac{1}{15\,b\,30^{\alpha-1}}\right\},\\
    c_{0,\alpha}:=\left(\frac{2}{q_\alpha}\right)^{1/\alpha},\qquad c_{1,\alpha}:=49\left(\frac{4}{q_\alpha}\right)^{2/\alpha},
\end{gather*}
and
\begin{align*}
    a_*(K,\beta):=\max\Big\{176400H^2,\ c_{0,\alpha}\frac{\sigma_0}{LR}(K+1)K^{\frac1\alpha}H^r,\ c_{1,\alpha}\sigma_1^2K^{\frac2\alpha}H^{2r}\Big\}.
\end{align*}
Then, for any $a\ge a_*(K,\beta)$ and $\lambda_k=\frac{R}{30\,\omega_{k+1}H}$, $0\le k<K$, there exists an event of probability at least $1-\beta$ on which all points $x^k,y^k,z^k$, $k=0,\dots,K$, produced by Algorithm~\ref{alg:clip-sstm} belong to $B_{2R}(x^*)$ and
\begin{equation}\label{eq:sstm-a}
    f(y^K)-f^*\le\frac{6aLR^2}{K(K+3)}.
\end{equation}
Moreover, for $a=a_*(K,\beta)$, on the same event, we have
\begin{align}
    f(y^K)-f^*\le\max\Big\{\frac{1058400\,LR^2H^2}{K^2},\ \frac{12c_{0,\alpha}\sigma_0RH^r}{K^r},\ \frac{6c_{1,\alpha}\sigma_1^2LR^2H^{2r}}{K^{2r}}\Big\}.\label{eq:sstm-rate}
\end{align}
The constants $c_{0,\alpha},c_{1,\alpha}$ depend only on $\alpha$ under Assumption~\ref{as:moment} (at $\alpha=2$, $q_2=\frac1{1350}$, $c_{0,2}=\sqrt{2700}$, $c_{1,2}=49\cdot5400$), or under Assumption~\ref{as:weak} with $\alpha<2$ if $v$ is replaced with $v_\alpha(\infty)$ (which is allowed), and are $O(\sqrt v)$ and $O(v)$ with $v=1+\ln(1+\Lambda_K^{\rm S})=O({\ln(eK)}+\ln(1+\Lambda_K^0))$ at $\alpha=2$ under Assumption~\ref{as:weak}.
\end{theorem}

\begin{corollary}[Oracle complexity of \clipsstm]\label{cor:sstm-complexity}
Let the assumptions of Theorem~\ref{thm:sstm} hold with $\alpha<2$ under Assumption~\ref{as:weak}, in which case $q_\alpha$, $c_{0,\alpha}$, $c_{1,\alpha}$ and $a_*(K,\beta)$ are computed with $v=v_\alpha(\infty)$; let $C_\alpha:=\max\{1058400,12c_{0,\alpha},6c_{1,\alpha}\}$, $\varepsilon>0$, and
\begin{align*}
    M_\varepsilon:=\max\Bigg\{1,\sqrt{\frac{C_\alpha LR^2}{\varepsilon}},\left(\frac{C_\alpha\sigma_0R}{\varepsilon}\right)^{\frac1r},\left(\frac{C_\alpha\sigma_1^2LR^2}{\varepsilon}\right)^{\frac1{2r}}\Bigg\}.
\end{align*}
If $K\ge4M_\varepsilon\ln\frac{4M_\varepsilon}{\beta}$, then Algorithm~\ref{alg:clip-sstm} with $a=a_*(K,\beta)$ and $\lambda_k$ defined as in Theorem~\ref{thm:sstm} guarantees $f(y^K)-f^*\le\varepsilon$ with probability at least $1-\beta$. In particular,
\begin{align*}
    \widetilde O_\alpha\Bigg(\max\Bigg\{1,\sqrt{\frac{LR^2}{\varepsilon}},\left(\frac{\sigma_0R}{\varepsilon}\right)^{\frac{\alpha}{\alpha-1}},\left(\frac{\sigma_1^2LR^2}{\varepsilon}\right)^{\frac{\alpha}{2(\alpha-1)}}\Bigg\}\Bigg)
\end{align*}
oracle calls are sufficient, where $\widetilde O_\alpha$ hides constants depending only on $\alpha$ and logarithmic factors, including those in $1/\beta$. At $\alpha=2$ under Assumption~\ref{as:weak}, the same complexity holds with an additional polylogarithmic factor, as in Corollary~\ref{cor:rate}.
\end{corollary}

Throughout Sections~\ref{app:sstm-potential}--\ref{app:sstm-uniform}, we assume that Assumptions~\ref{as:smooth}--\ref{as:convex} hold on $Q=B_{3R}(x^*)$, either Assumption~\ref{as:moment} or Assumption~\ref{as:weak} holds with $\alpha\in(1,2]$, and $(b,v)$, $H$, $r$, $q_\alpha$, $c_{0,\alpha}$, $c_{1,\alpha}$, $a_*(K,\beta)$, $\Lambda_K^{\rm S}$, $\Lambda_K^0$ are defined as in Theorem~\ref{thm:sstm} and~\eqref{eq:Lambda-sstm}; $\nabla f:=0$ outside $Q$, as in Appendix~\ref{app:aux}. As in Appendix~\ref{app:main}, let $\phi_\lambda(s):=v_\alpha^{\rm mom}s$ under Assumption~\ref{as:moment} and $\phi_\lambda(s):=s\,v_\alpha((3\lambda/2)^\alpha/s)$ under Assumption~\ref{as:weak}, with $\phi_\lambda(0):=0$. For $a\ge1$, we define
\begin{equation}\label{eq:Psibar}
    \bar\Psi_K(a):=\frac{\sigma_0^\alpha H^{\alpha-1}K(K+1)^\alpha}{2^\alpha a^\alpha L^\alpha R^\alpha}+\frac{7^\alpha\sigma_1^\alpha H^{\alpha-1}K}{a^{\alpha/2}},\quad
    \Lambda(a):=\frac{K}{20^\alpha H\,\bar\Psi_K(a)}\quad(\Lambda(a):=0\ \text{if}\ \bar\Psi_K(a)=0),
\end{equation}
which are the deterministic bound on the affine budget from Step~3 below and the corresponding logarithmic ratio; $\bar\Psi_K$ is non-increasing and $\Lambda$ is non-decreasing in $a$. For the iterates of Algorithm~\ref{alg:clip-sstm}, we define the errors
\begin{equation}\label{eq:theta-sstm}
    \theta_{k+1}:=\tg_k-\nabla f(x^{k+1})=\theta_{k+1}^u+\theta_{k+1}^b,\qquad
    \theta_{k+1}^u:=\tg_k-\E[\tg_k\mid\G_k],\qquad
    \theta_{k+1}^b:=\E[\tg_k\mid\G_k]-\nabla f(x^{k+1}).
\end{equation}
Since $\|\tg_k\|\le\lambda_k$, both conditional expectations are defined on the whole probability space, $\E[\theta_{k+1}^u\mid\G_k]=0$, and $\|\theta_{k+1}^u\|\le2\lambda_k$ almost surely. By induction over the recursions and the joint measurability of the oracle, $y^k$, $z^k$, and $x^{k+1}$ are $\G_k$-measurable. We never need the unconditional mean of the unclipped oracle at random points outside $Q$.

\subsection{Adaptive Query Points}\label{app:adaptive}

\begin{lemma}[Adaptive query points]\label{lem:adaptive}
Let Assumption~\ref{as:smooth} hold on $Q=B_{\rho R}(x^*)$ for some $\rho>0$, and let Assumption~\ref{as:moment} or Assumption~\ref{as:weak} hold on $Q$. Let $X$ be a $\G$-measurable random vector, let $\xi\sim\D$ be independent of $\G$, and let $A_X:=\{X\in Q\}$. Then $\ind_{A_X}g(X,\xi)\in L^1$ and
\begin{equation}\label{eq:adaptive-mean}
    \E\left[\ind_{A_X}g(X,\xi)\mid\G\right]=\ind_{A_X}\nabla f(X).
\end{equation}
Moreover, almost surely, under Assumption~\ref{as:weak}, for every $u>0$,
\begin{equation}\label{eq:adaptive-tail}
    \ind_{A_X}\PP\left\{\|g(X,\xi)-\nabla f(X)\|>u\mid\G\right\}\le\ind_{A_X}\frac{\sigma_0^\alpha+\sigma_1^\alpha\|\nabla f(X)\|^\alpha}{u^\alpha},
\end{equation}
and under Assumption~\ref{as:moment}, $\ind_{A_X}\E\left[\|g(X,\xi)-\nabla f(X)\|^\alpha\mid\G\right]\le\ind_{A_X}\left(\sigma_0^\alpha+\sigma_1^\alpha\|\nabla f(X)\|^\alpha\right)$.
\end{lemma}

\begin{proof}
Since $x^*$ is an unconstrained minimizer of the differentiable function $f$, we have $\nabla f(x^*)=0$, and Assumption~\ref{as:smooth} gives
\begin{equation}\label{eq:grad-lin}
    \|\nabla f(x)\|\le L\|x-x^*\|\le\rho LR,\qquad x\in Q .
\end{equation}
Let $M_Q:=(\sigma_0^\alpha+\sigma_1^\alpha(\rho LR)^\alpha)^{1/\alpha}$. For every $x\in Q$, we have
\begin{equation}\label{eq:oracle-L1}
    \E_\xi\|g(x,\xi)\|\le\rho LR+\frac{\alpha}{\alpha-1}M_Q .
\end{equation}
Indeed, under Assumption~\ref{as:moment}, this follows from Lyapunov's inequality, $\E_\xi\|g(x,\xi)-\nabla f(x)\|\le M_Q$. Under Assumption~\ref{as:weak}, for $Y\ge0$ and $M>0$ with $\PP\{Y>u\}\le\min\{1,M^\alpha u^{-\alpha}\}$ for all $u>0$, we have $\E Y\le\int_0^M\dif u+M^\alpha\int_M^\infty u^{-\alpha}\dif u=\frac{\alpha}{\alpha-1}M$; for $M=0$, $Y=0$ almost surely.

Next, the vector $X_Q:=X\ind_{A_X}+x^*\ind_{A_X^c}$ is $\G$-measurable and belongs to $Q$. By independence and joint measurability, for every jointly measurable $\varphi$ with $\E|\varphi(X_Q,\xi)|<\infty$, we have
\[
    \E\left[\varphi(X_Q,\xi)\mid\G\right]=\E_{\xi\sim\D}\,\varphi(x,\xi)\big|_{x=X_Q}.
\]
This formula is verified first for indicators of measurable rectangles in $(x,\xi)$, then for simple functions, and, finally, for integrable functions by monotone convergence and splitting into positive and negative parts. By~\eqref{eq:oracle-L1}, $g(X_Q,\xi)$ is integrable, uniformly in the value of $X_Q$, and $\E[g(X_Q,\xi)\mid\G]=\nabla f(X_Q)$. Multiplying by $\ind_{A_X}$, on which $X_Q=X$, we get~\eqref{eq:adaptive-mean}. For the tail indicator $\varphi=\ind\{\|g(x,\xi)-\nabla f(x)\|>u\}$ and for the non-negative function $\varphi=\|g(x,\xi)-\nabla f(x)\|^\alpha$, the same formula does not require the integrability of $g$. {For the tail probabilities, we choose the version supplied by the freezing formula; the pointwise tail assumption then gives the inequality simultaneously for every $u>0$.} Together with Assumption~\ref{as:weak} (resp.\ Assumption~\ref{as:moment}), it gives the remaining claims, which concludes the proof.
\end{proof}

Without additional integrability of $g(X,\xi)$, the conditional expectation $\E[g(X,\xi)\mid\G]$ outside $Q$ is never used. For example, if $g(x,\xi)=x$ and $X$ has a Cauchy distribution, then the vector $g(X,\xi)$ is not integrable, while its restriction to $\{X\in Q\}$ is. In contrast, the clipped oracle is always bounded and has an ordinary conditional expectation on the whole probability space.

\subsection{Potential Inequality}\label{app:sstm-potential}

\begin{lemma}[Potential inequality]\label{lem:potential}
Let $a\ge1$ and let the points $x^{k+1},y^k,y^{k+1},z^k,z^{k+1}$, $0\le k<N$, produced by Algorithm~\ref{alg:clip-sstm} belong to $Q$. Then, for every $z\in Q$, we have
\begin{equation}\label{eq:potential}
\begin{aligned}
    A_N\left(f(y^N)-f(z)\right)\le{}&\frac12\|z^0-z\|^2-\frac12\|z^N-z\|^2\\
    &+\sum_{k=0}^{N-1}\omega_{k+1}\abr{\theta_{k+1},z-z^k+\omega_{k+1}\nabla f(x^{k+1})}+\sum_{k=0}^{N-1}\omega_{k+1}^2\|\theta_{k+1}\|^2 .
\end{aligned}
\end{equation}
\end{lemma}

\begin{proof}
We fix one step and introduce the notation $v:=\nabla f(x^{k+1})$ and $g:=\tg_k=v+\theta_{k+1}$. The recursions give $y^{k+1}-x^{k+1}=-\frac{\omega_{k+1}^2}{A_{k+1}}g$. The descent lemma holds on the segment $[x^{k+1},y^{k+1}]\subset Q$, since it follows by integrating Assumption~\ref{as:smooth} along the segment. Thus, by~\eqref{eq:Ak} and $a\ge1$, we have
\[
    A_{k+1}f(y^{k+1})\le A_{k+1}f(x^{k+1})-\omega_{k+1}^2\abr{v,g}+\frac{L\omega_{k+1}^4}{2A_{k+1}}\|g\|^2
    \le A_{k+1}f(x^{k+1})-\frac{\omega_{k+1}^2}{2}\|g\|^2+\omega_{k+1}^2\abr{\theta_{k+1},g}.
\]
The identity $A_k(x^{k+1}-y^k)=\omega_{k+1}(z^k-x^{k+1})$ and convexity give
\[
    A_{k+1}f(x^{k+1})-A_kf(y^k)\le\omega_{k+1}\left(f(x^{k+1})+\abr{v,z^k-x^{k+1}}\right)\le\omega_{k+1}f(z)+\omega_{k+1}\abr{v,z^k-z}.
\]
Finally, the update rule for $z$ gives $\omega_{k+1}\abr{g,z^k-z}-\frac{\omega_{k+1}^2}{2}\|g\|^2=\frac12\|z^k-z\|^2-\frac12\|z^{k+1}-z\|^2$. Substituting $v=g-\theta_{k+1}$ and combining the three relations, we get
\[
    A_{k+1}f(y^{k+1})-A_kf(y^k)\le\omega_{k+1}f(z)+\frac12\|z^k-z\|^2-\frac12\|z^{k+1}-z\|^2+\omega_{k+1}\abr{\theta_{k+1},z-z^{k+1}} .
\]
It remains to use $z-z^{k+1}=z-z^k+\omega_{k+1}\nabla f(x^{k+1})+\omega_{k+1}\theta_{k+1}$ and to sum up the above inequality over $k=0,\dots,N-1$. The terms $A_kf(y^k)$ and the squared distances telescope, and $A_0=0$, which concludes the proof.
\end{proof}

\subsection{Weighted Bound on the True Gradients}\label{app:sstm-budget}

For $t\ge0$, we define
\begin{equation}\label{eq:Bt}
    B_t:=\sum_{l=0}^{t-1}\omega_{l+1}\abr{\theta_{l+1},x^*-z^l+\omega_{l+1}\nabla f(x^{l+1})}+\sum_{l=0}^{t-1}\omega_{l+1}^2\|\theta_{l+1}\|^2,\qquad B_0:=0 .
\end{equation}
These are finite measurable random variables; we do not need their unconditional integrability.

\begin{lemma}[Gradient budget]\label{lem:budget}
Let $a\ge1$ and $T\in\{1,\dots,K\}$, and assume that $B_t\le R^2$ and $\|z^t-x^*\|\le2R$ for $0\le t<T$. Then, for every $0\le t<T$, we have
\begin{equation}\label{eq:budget-pointwise}
    \omega_{t+1}\|\nabla f(x^{t+1})\|\le\frac{7R}{\sqrt a},
    \qquad\text{and hence}\qquad
    \sum_{t=0}^{T-1}\omega_{t+1}^\alpha\|\nabla f(x^{t+1})\|^\alpha\le\frac{7^\alpha R^\alpha T}{a^{\alpha/2}} .
\end{equation}
\end{lemma}

\begin{proof}
By the recursions and induction, $y^t$ is a convex combination of $z^0,\dots,z^t$. Thus, $y^t,x^{t+1}\in B_{2R}(x^*)$ for $t<T$. For $1\le t<T$, all points from the first $t$ steps belong to $Q$, and Lemma~\ref{lem:potential} with $N=t$ and $z=x^*$ gives
\begin{equation}\label{eq:pot-bound}
    A_t\left(f(y^t)-f^*\right)\le\frac12R^2+B_t\le\frac32R^2 .
\end{equation}
Since $x^{t+1}-y^t=\omega_{t+1}(z^t-y^t)/A_{t+1}$ and $\|z^t-y^t\|\le4R$, Assumptions~\ref{as:smooth} and~\ref{as:grad-bound} give
\[
    \|\nabla f(x^{t+1})\|\le L\|x^{t+1}-y^t\|+\|\nabla f(y^t)\|\le\frac{4LR\,\omega_{t+1}}{A_{t+1}}+\sqrt{2L(f(y^t)-f^*)}\le\frac{4LR\,\omega_{t+1}}{A_{t+1}}+R\sqrt{\frac{3L}{A_t}} .
\]
Therefore, $\omega_{t+1}\|\nabla f(x^{t+1})\|\le4R\frac{L\omega_{t+1}^2}{A_{t+1}}+R\sqrt{\frac{3L\omega_{t+1}^2}{A_t}}$. The first ratio is at most $1/a$ by~\eqref{eq:Ak}. For the second one, we have $\frac{L\omega_{t+1}^2}{A_t}=\frac{(t+2)^2}{a\,t(t+3)}\le\frac{9}{4a}$, since $9t(t+3)-4(t+2)^2=(t-1)(5t+16)\ge0$. Consequently, $\omega_{t+1}\|\nabla f(x^{t+1})\|\le\frac{4R}{a}+\frac{3\sqrt3R}{2\sqrt a}<\frac{7R}{\sqrt a}$. For $t=0$, we have $x^1=x^0$ and $\omega_1=1/(aL)$, and \eqref{eq:grad-lin} gives $\omega_1\|\nabla f(x^1)\|\le R/a\le R/\sqrt a$. Raising to the power $\alpha$ and summing up, we get the second claim, which concludes the proof.
\end{proof}

\begin{remark}\label{rem:crude-budget}
The crude bound $\|\nabla f(x^{t+1})\|\le2LR$ would give only a budget of the order $R^\alpha K^{\alpha+1}/a^\alpha$, while Lemma~\ref{lem:budget} gives $R^\alpha K/a^{\alpha/2}$. Replacing the first budget with the second one is exactly what improves the dependence of Theorem~\ref{thm:sstm} on $\sigma_1$.
\end{remark}

\subsection{Proof of Theorem~\ref{thm:sstm}}\label{app:sstm-proof}

As in the proof of Proposition~\ref{prop:master}, our proof is induction-based: by induction, we show that the quantities $B_t$ and the distances $\|z^t-x^*\|$ remain bounded with high probability.

\paragraph{Step 1: induction events and applicability of the potential.} To formulate the statement rigorously, for $k=0,\dots,K$, we introduce probability event $E_k\in\G_k$ as follows: for all $0\le t\le k$,
\begin{equation}\label{eq:sstm-events}
    B_t\le R^2\qquad\text{and}\qquad\|z^t-x^*\|\le2R .
\end{equation}
We will prove by induction that $\PP(E_k)\ge1-k\beta/K$ for all $k=0,1,\dots,K$. The base of the induction ($k=0$) follows from $B_0=0$ and $\|z^0-x^*\|\le R$. Next, we assume that the statement holds for $k=T-1$, where $1\le T\le K$, and prove it for $k=T$. On $E_{T-1}$, by the convexity of the ball and the recursions, we have $y^t,z^t,x^{t+1}\in B_{2R}(x^*)$ for $t<T$. Moreover, since $H\ge\ln4>1$, we have
\[
    \|z^T-x^*\|\le\|z^{T-1}-x^*\|+\omega_T\lambda_{T-1}\le2R+\frac{R}{30H}<3R .
\]
The point $y^T$ is a convex combination of $y^{T-1}\in B_{2R}(x^*)$ and $z^T\in B_{3R}(x^*)$, implying $y^T\in Q$. Thus, on $E_{T-1}$, Lemma~\ref{lem:potential} can be applied to the first $T$ steps.

\paragraph{Step 2: the clipping condition.} Lemma~\ref{lem:budget} and $a\ge176400H^2$, i.e., $\sqrt a\ge420H$, give that $E_{T-1}$ implies $\omega_{t+1}\|\nabla f(x^{t+1})\|\le7R/\sqrt a\le R/(60H)$, i.e.,
\begin{equation}\label{eq:sstm-clip-cond}
    \|\nabla f(x^{t+1})\|\le\frac{R}{60\,\omega_{t+1}H}=\frac{\lambda_t}{2},\qquad t<T .
\end{equation}
Let $D_t:=\sigma_0^\alpha+\sigma_1^\alpha\|\nabla f(x^{t+1})\|^\alpha$. To apply the conditional lemmas correctly, we consider the predictable event $\mathcal C_t:=\{x^{t+1}\in Q,\ \|\nabla f(x^{t+1})\|\le\lambda_t/2\}\in\G_t$, the point $\check x^{t+1}:=x^{t+1}\ind_{\mathcal C_t}+x^*\ind_{\mathcal C_t^c}\in Q$, and $U_t:=g(\check x^{t+1},\xi^t)$. Lemma~\ref{lem:adaptive} gives $U_t\in L^1$ and $\E[U_t\mid\G_t]=\nabla f(\check x^{t+1})$, whose norm is at most $\lambda_t/2$ everywhere. Thus, Corollary~\ref{cor:conditional} can be applied to $U_t$. On $\mathcal C_t$, the vectors $\clip(U_t,\lambda_t)$ and $\tg_t$ coincide. Since both of them are bounded and $\mathcal C_t\in\G_t$, we have
\[
    \ind_{\mathcal C_t}\E\left[\clip(U_t,\lambda_t)-\tg_t\mid\G_t\right]=\E\left[\ind_{\mathcal C_t}\left(\clip(U_t,\lambda_t)-\tg_t\right)\mid\G_t\right]=0,
\]
and the same argument applies to the conditional second moments of the centered vectors. Therefore, on $\mathcal C_t$, and, as a consequence, on $E_{T-1}\subseteq\mathcal C_t$ for $t<T$, we have
\begin{equation}\label{eq:sstm-clip-bounds}
    \|\theta_{t+1}^b\|\le b\,D_t\lambda_t^{1-\alpha},\qquad \E\left[\|\theta_{t+1}^u\|^2\mid\G_t\right]\le\lambda_t^{2-\alpha}\phi_{\lambda_t}(D_t),
\end{equation}
where $D_t=A(\nabla f(x^{t+1}))$ is exactly the noise scale from the lemmas at $x^{t+1}$. Here we always condition on $\G_t$ and never on $E_{T-1}$, which may depend on future samples for $t<T-1$. The global bound $\|\theta_{t+1}^u\|\le2\lambda_t$ was mentioned above.

\paragraph{Step 3: a single affine budget.} Let
\begin{equation}\label{eq:PsiT}
    \Psi_T:=H^{\alpha-1}R^{-\alpha}\sum_{t=0}^{T-1}\omega_{t+1}^\alpha D_t .
\end{equation}
On $E_{T-1}$, we have $\sum_{t<T}\omega_{t+1}^\alpha D_t=\sigma_0^\alpha\sum_{t<T}\omega_{t+1}^\alpha+\sigma_1^\alpha\sum_{t<T}\omega_{t+1}^\alpha\|\nabla f(x^{t+1})\|^\alpha$. The first sum is at most $T(T+1)^\alpha/(2^\alpha a^\alpha L^\alpha)$, and the second one is bounded by Lemma~\ref{lem:budget}. Thus, since $T\le K$, we get
\begin{equation}\label{eq:Psi-bound}
    \Psi_T\le\frac{\sigma_0^\alpha H^{\alpha-1}K(K+1)^\alpha}{2^\alpha a^\alpha L^\alpha R^\alpha}+\frac{7^\alpha\sigma_1^\alpha H^{\alpha-1}K}{a^{\alpha/2}} .
\end{equation}
The right-hand side of~\eqref{eq:Psi-bound} is exactly $\bar\Psi_K(a)$ from~\eqref{eq:Psibar}. If $\sigma_0>0$, then the second term in the definition of $a_*(K,\beta)$ gives $a^\alpha\ge\frac{2}{q_\alpha}\frac{\sigma_0^\alpha}{L^\alpha R^\alpha}(K+1)^\alpha KH^{\alpha-1}$, meaning that the first term in~\eqref{eq:Psi-bound} is at most $q_\alpha/2^{\alpha+1}$; for $\sigma_0=0$, it vanishes. Similarly, if $\sigma_1>0$, then the third term gives $a^{\alpha/2}\ge\frac{4}{q_\alpha}7^\alpha\sigma_1^\alpha KH^{\alpha-1}$, meaning that the second term is at most $q_\alpha/4$; for $\sigma_1=0$, it vanishes. Thus,
\begin{equation}\label{eq:Psi-final}
    \Psi_T\le\bar\Psi_K(a)\le\frac{q_\alpha}{2^{\alpha+1}}+\frac{q_\alpha}{4}<\frac{q_\alpha}{2}\le q_\alpha .
\end{equation}

\paragraph{Step 3$'$: the logarithm at $a=a_*(K,\beta)$.} Under Assumption~\ref{as:weak}, we also need the variance constant $v=v_\alpha(\Lambda_K^{\rm S})$ to dominate the actual logarithmic ratio at $a_*:=a_*(K,\beta)$:
\begin{equation}\label{eq:sstm-fixed-point}
    \Lambda(a_*)\le\Lambda_K^{\rm S}.
\end{equation}
If $\bar\Psi_K(a_*)=0$, then $\Lambda(a_*)=0$. Otherwise, we consider which term attains the maximum in $a_*$. If $a_*=176400H^2$, then direct computations from~\eqref{eq:Psibar} give $\bar\Psi_K(a_*)=\frac KH\big[\big(\frac{\sigma_0(K+1)}{352800\,LRH}\big)^\alpha+\big(\frac{\sigma_1}{60}\big)^\alpha\big]$ (here we use $2\cdot176400=352800$ and $(176400H^2)^{\alpha/2}=420^\alpha H^\alpha$); hence, $\Lambda(a_*)=\big[20^\alpha\big(\frac{\sigma_0(K+1)}{352800\,LRH}\big)^\alpha+20^\alpha\big(\frac{\sigma_1}{60}\big)^\alpha\big]^{-1}=\Lambda_K^0\le\Lambda_K^{\rm S}$. If $a_*=c_{0,\alpha}\frac{\sigma_0}{LR}(K+1)K^{1/\alpha}H^r$, then the first term of $\bar\Psi_K(a_*)$ equals $\frac{1}{2^\alpha c_{0,\alpha}^\alpha}=\frac{q_\alpha}{2^{\alpha+1}}$, and if $a_*=c_{1,\alpha}\sigma_1^2K^{2/\alpha}H^{2r}$, then the second term equals $\frac{7^\alpha}{c_{1,\alpha}^{\alpha/2}}=\frac{q_\alpha}{4}$. In both cases, $\bar\Psi_K(a_*)\ge q_\alpha/2^{\alpha+1}$, which gives
\[
    \Lambda(a_*)\le\frac{2^{\alpha+1}K}{20^\alpha Hq_\alpha}=\frac{2^{\alpha+1}K}{20^\alpha H}\max\left\{1350\,v\,30^{\alpha-2},\ 15\,b\,30^{\alpha-1}\right\}.
\]
The second branch is $\frac{2^{\alpha+1}\cdot15\cdot30^{\alpha-1}}{20^\alpha}\frac{bK}{H}=3^\alpha\frac{bK}{H}\le9bK\le\Lambda_K^{\rm S}$, since $H>1$. For the first branch, $v\le(\frac32)^{2-\alpha}(1+\frac2\alpha\ln(1+\Lambda_K^{\rm S}))$ and $2^{\alpha+1}\cdot1350\cdot30^{\alpha-2}(\frac32)^{2-\alpha}/20^\alpha=2^{\alpha+1}\cdot1350\cdot20^{-2}\le27$ give the bound $27\frac KH\big(1+2\ln(1+\Lambda_K^{\rm S})\big)$. Since $b=b_\alpha\ge b_2=2$ under Assumption~\ref{as:weak}, we have $\Lambda_K^{\rm S}\ge400K$, and, since $\Lambda\mapsto\Lambda/(1+2\ln(1+\Lambda))$ is increasing, it is sufficient to check that $27K(1+2\ln(1+400K))\le400KH$, i.e., $27(1+2\ln(1+400K))\le400H$. Since $H\ge\ln(4K)$, this holds at $K=1$ ($351<554$), and the difference $400\ln(4K)-27(1+2\ln(1+400K))$ is increasing in $K\ge1$, since its derivative equals $400/K-21600/(1+400K)>0$. This proves~\eqref{eq:sstm-fixed-point}.

\paragraph{Step 4: bounded predictable vectors.} Before we proceed, we introduce a new notation: on the whole probability space, we define $w_t:=x^*-z^t+\omega_{t+1}\nabla f(x^{t+1})$ and $\eta_t:=w_t\ind_{\{\|w_t\|\le3R\}}$. The vector $\eta_t$ is $\G_t$-measurable and $\|\eta_t\|\le3R$ almost surely. On $E_{T-1}$, by~\eqref{eq:sstm-clip-cond}, we have $\|w_t\|\le2R+R/(60H)<3R$ for $t<T$, and, therefore, $\eta_t=w_t$ on this event.

\paragraph{Step 5: decomposition of $B_T$.} Let $Z_t:=\|\theta_{t+1}^u\|^2$. On $E_{T-1}$, using $\|u+v\|^2\le2\|u\|^2+2\|v\|^2$, we get $B_T\le\circledOne+\circledTwo+\circledThree+\circledFour+\circledFive$, where the following quantities are defined globally:
\begin{gather*}
    \circledOne:=\sum_{t<T}\omega_{t+1}\abr{\theta_{t+1}^u,\eta_t},\qquad \circledTwo:=\sum_{t<T}\omega_{t+1}\abr{\theta_{t+1}^b,\eta_t},\qquad \circledThree:=2\sum_{t<T}\omega_{t+1}^2\left(Z_t-\E[Z_t\mid\G_t]\right),\\
    \circledFour:=2\sum_{t<T}\omega_{t+1}^2\E[Z_t\mid\G_t],\qquad \circledFive:=2\sum_{t<T}\omega_{t+1}^2\|\theta_{t+1}^b\|^2 .
\end{gather*}
That is, it remains to derive the upper bounds for $\circledOne$--$\circledFive$.

\paragraph{Step 6: upper bound for $\circledOne$ (the linear martingale term).} To estimate this sum, we will use Bernstein's inequality in the martingale form (Lemma~\ref{lem:freedman}). The summands $X_t:=\omega_{t+1}\abr{\theta_{t+1}^u,\eta_t}$ have conditional expectations equal to zero given $\G_t$, and they are bounded globally:
\[
    |X_t|\le6R\,\omega_{t+1}\lambda_t=\frac{R^2}{5H}=:c_1,\qquad b_1:=\frac{R^2}{5},\qquad G:=\frac{R^4}{150H}.
\]
Finally, the conditional quadratic variation satisfies $V_1:=\sum_{t<T}\E[X_t^2\mid\G_t]\le9R^2S_T$, where $S_T:=\sum_{t<T}\omega_{t+1}^2\E[Z_t\mid\G_t]$. By the choice of $\lambda_t$, we have $\omega_{t+1}^2\lambda_t^{2-\alpha}=R^{2-\alpha}30^{\alpha-2}H^{\alpha-2}\omega_{t+1}^\alpha$, and, therefore, by~\eqref{eq:sstm-clip-bounds}, on $E_{T-1}$,
\[
    S_T\le R^{2-\alpha}30^{\alpha-2}H^{\alpha-2}\sum_{t<T}\omega_{t+1}^\alpha D_t\,\frac{\phi_{\lambda_t}(D_t)}{D_t}\qquad(\text{terms with }D_t=0\text{ are }0).
\]
{If $\bar\Psi_K(a)=0$, then $D_t=0$ for every $t<T$ on $E_{T-1}$, and~\eqref{eq:sstm-clip-bounds} gives $S_T=0$. The desired bound on $S_T$ is therefore immediate in this case. In the remaining estimates, assume $\bar\Psi_K(a)>0$, so that the budget used in Lemma~\ref{lem:log-sum} is positive.} Under Assumption~\ref{as:moment}, we have $\phi_{\lambda_t}(D_t)=vD_t$ with $v=v_\alpha^{\rm mom}$ (also when $D_t=0$), which gives $\sum_{t<T}\omega_{t+1}^\alpha\phi_{\lambda_t}(D_t)=v\sum_{t<T}\omega_{t+1}^\alpha D_t=vR^\alpha H^{1-\alpha}\Psi_T$. Under Assumption~\ref{as:weak}, let $q_t:=\omega_{t+1}^\alpha D_t$; since $\omega_{t+1}\lambda_t=R/(30H)$, {for $D_t>0$,} we have $(3\lambda_t/2)^\alpha/D_t=C_{\rm S}/q_t$ with $C_{\rm S}:=(R/(20H))^\alpha$, and, therefore, $\omega_{t+1}^\alpha\phi_{\lambda_t}(D_t)=q_t\,v_\alpha(C_{\rm S}/q_t)$. {When $q_t=0$, we define the latter product as $0$, consistently with $\phi_{\lambda_t}(0)=0$.} By Lemma~\ref{lem:log-sum} with $\bar B:=R^\alpha H^{1-\alpha}\bar\Psi_K(a)\ge\sum_{t<T}q_t=R^\alpha H^{1-\alpha}\Psi_T$ (by~\eqref{eq:Psi-final}), $T\le K$, and the monotonicity of $v_\alpha$, we get
\[
    \sum_{t<T}q_t\,v_\alpha\!\left(\frac{C_{\rm S}}{q_t}\right)\le\bar B\,v_\alpha\!\left(\frac{KC_{\rm S}}{\bar B}\right)=R^\alpha H^{1-\alpha}\,\bar\Psi_K(a)\,v_\alpha(\Lambda(a)),
\]
since $KC_{\rm S}/\bar B=K/(20^\alpha H\bar\Psi_K(a))=\Lambda(a)$. Moreover, $u\mapsto u\,v_\alpha(KC_{\rm S}/u)$ is non-decreasing (see Lemma~\ref{lem:log-sum}) and $\bar\Psi_K(a)\le\bar\Psi_K(a_*)$ for $a\ge a_*$, implying $\bar\Psi_K(a)v_\alpha(\Lambda(a))\le\bar\Psi_K(a_*)v_\alpha(\Lambda(a_*))\le\frac{q_\alpha}{2}\,v$ by~\eqref{eq:Psi-final} and~\eqref{eq:sstm-fixed-point}. In both cases, by the definition of $q_\alpha$, we have
\begin{equation}\label{eq:S-bound}
    S_T\le v\,30^{\alpha-2}\frac{R^2}{H}\cdot\frac{q_\alpha}{2}\le\frac{R^2}{2700H}\le\frac{R^2}{1350H},\qquad V_1\le\frac{R^4}{150H}=G .
\end{equation}
Since $\frac{b_1^2}{2G+2c_1b_1/3}=H$, Lemma~\ref{lem:freedman} gives
\begin{equation}\label{eq:I1-prob}
    \PP\left\{|\circledOne|>\frac{R^2}{5},\ V_1\le G\right\}\le2e^{-H}=\frac{\beta}{2K}.
\end{equation}

\paragraph{Step 7: upper bound for $\circledTwo$ (the systematic bias).} On $E_{T-1}$, by~\eqref{eq:sstm-clip-bounds}, the choice of $\lambda_t$, and~\eqref{eq:Psi-final}, we have
\begin{equation}\label{eq:I2}
    |\circledTwo|\le3R\sum_{t<T}\omega_{t+1}\|\theta_{t+1}^b\|\le3Rb\sum_{t<T}\omega_{t+1}D_t\lambda_t^{1-\alpha}=3b\,30^{\alpha-1}R^2\Psi_T\le\frac{R^2}{5},
\end{equation}
where the last step uses the second term in the definition of $q_\alpha$. In particular, $\sum_{t<T}\omega_{t+1}\|\theta_{t+1}^b\|\le R/15$.

\paragraph{Step 8: upper bound for $\circledThree$ (the quadratic martingale term).} Let $Y_t:=2\omega_{t+1}^2(Z_t-\E[Z_t\mid\G_t])$. Since $0\le Z_t\le4\lambda_t^2$ globally, the summands are bounded:
\[
    |Y_t|\le8\omega_{t+1}^2\lambda_t^2=\frac{8R^2}{900H^2}\le\frac{R^2}{5H}=c_1 .
\]
Moreover, $\E[Y_t^2\mid\G_t]=4\omega_{t+1}^4\,\mathrm{Var}(Z_t\mid\G_t)\le4\omega_{t+1}^4\E[Z_t^2\mid\G_t]\le16\omega_{t+1}^4\lambda_t^2\E[Z_t\mid\G_t]$. Since $\omega_{t+1}^2\lambda_t^2=R^2/(900H^2)$, on $E_{T-1}$, we have
\[
    V_3:=\sum_{t<T}\E[Y_t^2\mid\G_t]\le\frac{16R^2}{900H^2}S_T\le\frac{16R^4}{1215000H^3}\le\frac{R^4}{150H}=G,
\]
where the last inequality uses $H>1$. Lemma~\ref{lem:freedman} with the same $b_1,c_1,G$ gives
\begin{equation}\label{eq:I3-prob}
    \PP\left\{|\circledThree|>\frac{R^2}{5},\ V_3\le G\right\}\le\frac{\beta}{2K}.
\end{equation}

\paragraph{Step 9: upper bounds for $\circledFour$ and $\circledFive$.} On $E_{T-1}$, by~\eqref{eq:S-bound}, we have $\circledFour=2S_T\le\frac{R^2}{675H}$. By Step~7 and the non-negativity of the summands, we also have $\circledFive\le2\left(\sum_{t<T}\omega_{t+1}\|\theta_{t+1}^b\|\right)^2\le\frac{2R^2}{225}$.

\paragraph{Step 10: closing the induction and the final bound.} That is, we derived the upper bounds for $\circledOne$--$\circledFive$. Let $E_{\circledOne}:=\{V_1>G\}\cup\{|\circledOne|\le R^2/5\}$ and $E_{\circledThree}:=\{V_3>G\}\cup\{|\circledThree|\le R^2/5\}$. By~\eqref{eq:I1-prob} and~\eqref{eq:I3-prob}, each of these events has probability at least $1-\beta/(2K)$. On $E_{T-1}$, we have already shown that $V_1,V_3\le G$. Thus, $E_{T-1}\cap E_{\circledOne}\cap E_{\circledThree}$ implies, by Steps~5--9 and $H>1$,
\[
    B_T\le\left(\frac35+\frac1{675}+\frac2{225}\right)R^2=\frac{412}{675}R^2<R^2 .
\]
By Step~1, Lemma~\ref{lem:potential} can be applied up to step $T$. Since $f(y^T)\ge f^*$, it gives $\|z^T-x^*\|^2\le\|z^0-x^*\|^2+2B_T<3R^2<4R^2$. Therefore, $E_{T-1}\cap E_{\circledOne}\cap E_{\circledThree}\subseteq E_T$, and the union bound, which does not require the independence of the events, gives
\[
    \PP(E_T)\ge1-\frac{(T-1)\beta}{K}-\frac{\beta}{2K}-\frac{\beta}{2K}=1-\frac{T\beta}{K}.
\]
This finishes the induction. Finally, $E_K$ implies, via the recursions, that all points belong to $B_{2R}(x^*)$, and Lemma~\ref{lem:potential} with $z=x^*$ and $B_K\le R^2$ gives
\[
    A_K\left(f(y^K)-f^*\right)\le\frac12R^2+B_K\le\frac32R^2 .
\]
Substituting $A_K=K(K+3)/(4aL)$ from~\eqref{eq:Ak}, we get~\eqref{eq:sstm-a}. For $a=a_*(K,\beta)$, using $K(K+3)\ge K^2$, $K+1\le2K$, and $1-r=1/\alpha$, we get
\[
    \frac{6a_*LR^2}{K(K+3)}\le\max\left\{\frac{1058400\,LR^2H^2}{K^2},\ \frac{12c_{0,\alpha}\sigma_0RH^r}{K^r},\ \frac{6c_{1,\alpha}\sigma_1^2LR^2H^{2r}}{K^{2r}}\right\},
\]
which proves~\eqref{eq:sstm-rate} and concludes the proof.
\hfill$\square$

\subsection{Oracle Complexity}\label{app:sstm-complexity}

\begin{lemma}[Inverting the logarithm]\label{lem:loginv}
Let $M\ge1$ and $\beta\in(0,1]$. If $K\ge4M\ln\frac{4M}{\beta}$, then $\ln\frac{4K}{\beta}\le\frac KM$.
\end{lemma}

\begin{proof}
Let $s:=4M/\beta\ge4$ and $K_0:=4M\ln s\ge M$. The function $\varphi(K):=K/M-\ln(4K/\beta)$ is non-decreasing for $K\ge M$, since $\varphi'(K)=M^{-1}-K^{-1}\ge0$. Therefore, it is sufficient to check the case $K=K_0$: $\varphi(K_0)\ge0\iff\ln(4s\ln s)\le4\ln s\iff4\ln s\le s^3$. The last inequality holds for $s\ge4$, since the function $s^3-4\ln s$ is positive at $s=4$ and has positive derivative $3s^2-4/s$ on $[4,\infty)$, which concludes the proof.
\end{proof}

\paragraph{Proof of Corollary~\ref{cor:sstm-complexity}.} By Lemma~\ref{lem:loginv}, we have $H/K\le M_\varepsilon^{-1}$. By the definition of $M_\varepsilon$, $C_\alpha LR^2M_\varepsilon^{-2}\le\varepsilon$, $C_\alpha\sigma_0RM_\varepsilon^{-r}\le\varepsilon$, and $C_\alpha\sigma_1^2LR^2M_\varepsilon^{-2r}\le\varepsilon$. Since each of the three constants in~\eqref{eq:sstm-rate} is at most $C_\alpha$, substituting these inequalities into~\eqref{eq:sstm-rate}, we get $f(y^K)-f^*\le\varepsilon$. To count the oracle calls, we take $K=\lceil4M_\varepsilon\ln(4M_\varepsilon/\beta)\rceil$; the number of calls equals the number of steps $K$, which concludes the proof.
\hfill$\square$

\begin{remark}[Why the choice of $a$ matters]\label{rem:choice-a}
Corollary~\ref{cor:sstm-complexity} is not a statement about all $a\ge a_*(K,\beta)$. Consider $\beta=1/2$, $f(x)=Lx^2/2$, $x^*=0$, $x^0=R>0$, an exact oracle, and $\sigma_0=\sigma_1=0$. For $a\ge\max\{a_*(K,\beta),K(K+3)\}$, we have $LA_K\le1/4$. By induction, all points $x^{k+1},y^k,z^k$ up to step $K$ belong to the segment between $3R/4$ and $R$: while the points belong to $[0,R]$, the clipped gradient belongs to $[0,LR]$ and $z^{k+1}=R-\sum_{j=0}^k\omega_{j+1}\clip(Lx^{j+1},\lambda_j)\ge R-LRA_{k+1}\ge3R/4$. The upper bound for $z^{k+1}$ and both bounds for $x$ and $y$ follow from the recursions. Thus, $f(y^K)-f^*\ge9LR^2/32$, and the accuracy $\varepsilon=LR^2/8$ is not achieved, no matter how large $K$ is. Due to this reason, the corollary fixes $a=a_*(K,\beta)$.
\end{remark}

\subsection{Comparison with a Uniform Noise Scale}\label{app:sstm-uniform}

\begin{corollary}[Uniform noise scale]\label{cor:uniform}
Let the assumptions of Theorem~\ref{thm:sstm} hold and let $\bar\sigma:=(\sigma_0^\alpha+(3LR\sigma_1)^\alpha)^{1/\alpha}$. Then the noise assumption in force also holds on $Q$ with the parameters $(\bar\sigma,0)$. As a consequence, Algorithm~\ref{alg:clip-sstm} with $a=a_{\rm uni}:=\max\{176400H^2,\,c_{0,\alpha}\frac{\bar\sigma}{LR}(K+1)K^{1/\alpha}H^r\}$ (where, under Assumption~\ref{as:weak}, $v$ and $\Lambda_K^0$ in $c_{0,\alpha}$ are computed with the parameters $(\bar\sigma,0)$) and $\lambda_k$ defined as in Theorem~\ref{thm:sstm} satisfies, with probability at least $1-\beta$,
\begin{equation}\label{eq:uniform}
    f(y^K)-f^*\le C_\alpha\max\left\{LR^2\left(\frac HK\right)^2,\ \bar\sigma R\left(\frac HK\right)^r\right\},\qquad C_\alpha:=\max\{1058400,12c_{0,\alpha}\}.
\end{equation}
Moreover, with $a:=\min\{a_*(K,\beta),a_{\rm uni}\}$, with probability at least $1-\beta$, the localization from Theorem~\ref{thm:sstm} holds and $f(y^K)-f^*\le\frac{6LR^2}{K(K+3)}\min\{a_*(K,\beta),a_{\rm uni}\}$.
\end{corollary}

\begin{proof}
By~\eqref{eq:grad-lin}, we have $\|\nabla f(x)\|\le3LR$ on $Q$, and, therefore, $\sigma_0^\alpha+\sigma_1^\alpha\|\nabla f(x)\|^\alpha\le\bar\sigma^\alpha$ for all $x\in Q$, and both~\eqref{eq:moment-affine} and~\eqref{eq:weak-tail} hold with $(\bar\sigma,0)$. Theorem~\ref{thm:sstm} with the parameters $(\bar\sigma,0)$ has the threshold $a_{\rm uni}$ and gives~\eqref{eq:uniform}. To prove the last claim, we notice that $a$ is deterministic. If it equals $a_*(K,\beta)$, we apply Theorem~\ref{thm:sstm} with $(\sigma_0,\sigma_1)$; otherwise, we apply it with $(\bar\sigma,0)$. In both cases, \eqref{eq:sstm-a} holds with probability at least $1-\beta$, and no additional union bound is needed, which concludes the proof.
\end{proof}

The factor $3$ in $\bar\sigma$ ensures that the tail condition holds on the whole set $Q$, on which the assumptions of Theorem~\ref{thm:sstm} are imposed. Since all queries are localized in $B_{2R}(x^*)$, repeating the probabilistic argument would allow a smaller factor, but this is not needed for the comparison of the orders. Since $\bar\sigma\le\sigma_0+3LR\sigma_1$, the multiplicative part contributes $O(\sigma_1LR^2(H/K)^r)$ to~\eqref{eq:uniform}. Theorem~\ref{thm:sstm} replaces this term with $\sigma_1^2LR^2(H/K)^{2r}$. The ratio of the two scales is of the order $\sigma_1(H/K)^r$ and tends to zero for fixed $\sigma_1$ and $\beta$ as $K\to\infty$; we do not claim that this is a uniform improvement for all $K$ and $\sigma_1$. For $\sigma_0>0$, the overall order may be determined by the additive term $\sigma_0R(H/K)^r$, meaning that the improvement of the multiplicative part does not necessarily change the leading power. The comparison concerns the polynomial powers of the horizon; at the weak-$L^2$ endpoint, both bounds also contain the tail-dependent logarithmic coefficients $c_{0,2},c_{1,2}$ from Theorem~\ref{thm:sstm}.

\paragraph{Comparison with prior analyses.} {We emphasize that the polynomial rates themselves are not new: for $\sigma_1=0$, the exponents of $K$ and $\varepsilon$ in~\eqref{eq:sstm-rate} coincide with the ones from \citet{gorbunov2020stochastic} at $\alpha=2$ and from \citet{sadiev2023high} for $\alpha\in(1,2]$, and, for \clipsgd, with the ones from \citet{sadiev2023high} (see Table~\ref{tab:complexity}). What changes is the setting in which they hold.}

{First, the noise class: the weak-$L^\alpha$ condition, and, at $\alpha=2$, the Lorentz scale, replace the bounded $\alpha$-th moment, and, therefore, the rates cover the noise whose $\alpha$-th moment is infinite.}

{Second, the affine structure: the gradient-dependent part of the noise enters only through the terms $\sigma_1^2LR^2(H/K)^{2r}$ (\clipsstm) and $\Delta_1$ (\clipsgd), which decay in $K$ faster than the additive term. Affine noise with finite variance ($\alpha=2$, $\sigma_1>0$) is the classical setting from \citet{bottou2018optimization}; related affine conditions appear in the analyses of adaptive stepsizes by \citet{faw2022power,wang2023convergence} (affine conditional variance, non-convex criteria) and \citet{attia2023sgd} (almost surely bounded affine noise in the main condition, with a sub-Gaussian extension). At $\alpha=2$, our multiplicative term has the order $\sigma_1^2LR^2/K$, up to logarithmic factors, which matches the variance term from the finite-variance analyses; the criteria and assumptions differ, meaning that this is a comparison of the orders only.}

{We do not claim that this term is optimal: we do not provide a lower bound for the multiplicative part, and, under purely multiplicative noise with additional structure, e.g., $\sigma_0=0$ and strong convexity, faster guarantees are known \citep{gower2019sgd}.}

{Faster polynomial decay is also possible without strong convexity under a global strong growth condition: under Assumption~\ref{as:moment} with $\alpha=2$ and $\sigma_0=0$, unbiasedness gives $\E\|g(x,\xi)\|^2=\|\nabla f(x)\|^2+\E\|g(x,\xi)-\nabla f(x)\|^2\le(1+\sigma_1^2)\|\nabla f(x)\|^2$, and \citet[Theorem~2]{vaswani2019fast} derive, for an accelerated method under global smoothness and strong growth with constant $\rho=1+\sigma_1^2$, the bound $\E f(x^K)-f^*\le2\rho^2LR^2/K^2$; hence, $2\rho^2LR^2/(\beta K^2)$ with probability at least $1-\beta$ by Markov's inequality, which is a faster polynomial order at fixed confidence with a much worse dependence on $\beta$ than our logarithmic one. \citet{gupta2024nesterov} also obtain acceleration in expectation for smooth convex and strongly convex problems through AGNES, under global multiplicative finite-variance noise and with arbitrary noise-to-gradient ratios. These results concern different algorithms and assumptions and show why our multiplicative term should not be regarded as optimal. An analysis bounding the noise scale uniformly over the ball does not retain the separation of the additive and gradient-dependent components (see Corollary~\ref{cor:uniform}).}

{Third, localization: the smoothness, convexity, and noise assumptions are imposed only on the ball $Q$, and all iterates remain in $B_{2R}(x^*)$ on the high-probability event (in Theorem~\ref{thm:any-lambda}, up to the first iterate in the neighborhood). Fourth, the guarantees are high-probability ones with logarithmic dependence on $1/\beta$ and with explicit constants and parameter choices; the logarithmic dependence on the horizon in $\ell$ and $H$ is an artifact of the union bound and can be removed outside the $\Lambda$-dependent weak-case tuning (see Appendix~\ref{app:maximal}).}

\subsection{The Lorentz Scale at the Endpoint}\label{app:lorentz}

In this section, we provide a direct finite-horizon analysis of \clipsstm under an affine Lorentz condition at the endpoint $\alpha=2$. Since $L^{2,2}=L^2$ and $L^{2,s}\subseteq L^{2,\infty}$, the result extends the finite-variance setting and refines the weak-$L^2$ guarantee; for every $s>2$, the inclusion $L^2\subsetneq L^{2,s}$ is strict, meaning that the assumption allows the noise with infinite second moment. For a real random variable $Y$ with non-increasing rearrangement $Y^*$ (see Appendix~\ref{app:spaces}) and $2\le s<\infty$, we define
\[
    \|Y\|_{2,s}:=\left(\int_0^1\left[t^{1/2}Y^*(t)\right]^s\frac{\dif t}{t}\right)^{1/s},\qquad \|Y\|_{2,\infty}:=\sup_{0<t\le1}t^{1/2}Y^*(t)=\wnorm{Y}\big|_{\alpha=2},
\]
and let $L^{2,s}$ be the class of $Y$ with $\|Y\|_{2,s}<\infty$; $\|Y\|_{2,2}^2=\E Y^2$. Let $\delta_s:=1-2/s$ for $s<\infty$ and $\delta_\infty:=1$.

\begin{lemma}[Lorentz embeddings]\label{lem:lorentz-embed}
For $2\le s<\infty$, $\|Y\|_{2,s}\le(\E Y^2)^{1/2}$ and $\|Y\|_{2,\infty}^2\le(s/2)^{2/s}\|Y\|_{2,s}^2\le2\|Y\|_{2,s}^2$. On an atomless probability space, $L^2\subsetneq L^{2,s}\subsetneq L^{2,\infty}$ for $2<s<\infty$.
\end{lemma}

\begin{proof}
Let $M^2:=\E Y^2=\int_0^1Y^*(u)^2\dif u$ and $F(t):=t^{1/2}Y^*(t)$. Since $Y^*$ is non-increasing, we have $F(t)^2\le\int_0^tY^*(u)^2\dif u\le M^2$; hence, $\int_0^1F^s\frac{\dif t}{t}\le M^{s-2}\int_0^1F^2\frac{\dif t}{t}=M^s$. The second claim follows from $\|Y\|_{2,s}^s\ge Y^*(t)^s\int_0^tu^{s/2-1}\dif u=\frac2s\,t^{s/2}Y^*(t)^s$ for every $t$ and $(s/2)^{2/s}=\rho^{1/\rho}$ with $\rho=s/2\ge1$ is at most $e^{1/e}<2$. To prove the strict inclusions, we consider $Y_b(t):=t^{-1/2}[\ln(e/t)]^{-b}$ with $1/s<b\le1/2$ (a non-increasing function on $(0,1]$, which is its own rearrangement up to the single point $t=1$, which does not affect the integrals): the substitution $u=\ln(e/t)$ gives $\|Y_b\|_{2,s}^s=\int_1^\infty u^{-bs}\dif u=(bs-1)^{-1}<\infty$ and $\E Y_b^2=\int_1^\infty u^{-2b}\dif u=\infty$; moreover, $t^{-1/2}$ has $\|\cdot\|_{2,\infty}=1$ but $\int_0^1\frac{\dif t}{t}=\infty$, which concludes the proof.
\end{proof}

\paragraph{Monotonicity of the Lorentz scale.} For $2\le s_1<s_2\le\infty$ (with $1/\infty:=0$), we have
\begin{equation}\label{eq:lorentz-monotone}
    \|Y\|_{2,s_2}\le\left(\frac{s_1}{2}\right)^{\frac1{s_1}-\frac1{s_2}}\|Y\|_{2,s_1}.
\end{equation}
Indeed, let $\|Y\|_{2,s_1}<\infty$ and $h(t):=t^{1/2}Y^*(t)$. Since $Y^*$ is non-increasing, for every $t\in(0,1)$, we have $\|Y\|_{2,s_1}^{s_1}\ge\int_0^tu^{s_1/2-1}Y^*(u)^{s_1}\dif u\ge\frac{2}{s_1}h(t)^{s_1}$, implying $\|h\|_\infty\le(s_1/2)^{1/s_1}\|Y\|_{2,s_1}$, which is the case $s_2=\infty$. For finite $s_2$, $\int_0^1h^{s_2}\frac{\dif t}{t}\le\|h\|_\infty^{s_2-s_1}\int_0^1h^{s_1}\frac{\dif t}{t}$ gives the claim. Thus, the classes $L^{2,s}$ increase with $s$, and a noise bound at the index $s_1$ with parameters $(\sigma_0,\sigma_1)$ implies the bound at every index $s_2>s_1$ with the parameters $(s_1/2)^{1/s_1-1/s_2}(\sigma_0,\sigma_1)$; the embedding of spaces is not an embedding of balls with the same parameters.

\begin{assumption}[Affine $L^{2,s}$ noise]\label{as:lorentz}
There exist $s\in[2,\infty]$ and $\sigma_0,\sigma_1\ge0$ such that for all $x\in Q$, $\E_\xi[g(x,\xi)]=\nabla f(x)$ and $\big\|\,\|g(x,\xi)-\nabla f(x)\|\,\big\|_{2,s}^2\le\sigma_0^2+\sigma_1^2\|\nabla f(x)\|^2$.
\end{assumption}

For $s=2$, this is Assumption~\ref{as:moment} with $\alpha=2$, and, for $s=\infty$, it is Assumption~\ref{as:weak} with $\alpha=2$; the classes increase with $s$ by~\eqref{eq:lorentz-monotone}, with the noise parameters transformed by the embedding constant when guarantees at different indices are compared, and, for $s>2$, the noise may have infinite variance (it is sufficient to multiply $Y_b$ by an independent Rademacher sign to obtain unbiased examples). The embedding constants from Lemma~\ref{lem:lorentz-embed} are uniform in $s$, but, for a fixed distribution, the quasi-norm $\|Y\|_{2,s}$ is not necessarily uniform in $s$: for $Y_{1/2}$, it equals $(s/2-1)^{-1/s}$, which diverges as $s\downarrow2$. Therefore, when guarantees for one distribution are compared across $s$, the parameters $\sigma_0,\sigma_1$ must be taken relative to the chosen $s$, and their growth may offset a smaller power of the logarithm.

\begin{lemma}[Truncated second moment on the Lorentz scale]\label{lem:lorentz-trunc}
Let $Y\ge0$ with $\|Y\|_{2,s}^2\le D$ for some $s\in[2,\infty]$ and $D\ge0$. Then, for every $c>0$,
\begin{equation}\label{eq:lorentz-trunc}
    \E\min\{Y^2,c^2\}\le2D\left[1+\ln\left(1+\frac{c^2}{D}\right)\right]^{\delta_s},
\end{equation}
with right-hand side $0$ if $D=0$.
\end{lemma}

\begin{proof}
If $D=0$, then $Y=0$ almost surely by Lemma~\ref{lem:lorentz-embed}. If $0<c^2\le D$, then the left-hand side is at most $c^2\le D$. Next, let $c^2>D>0$ and $t_0:=D/c^2$. For $2<s<\infty$, by $\E\min\{Y^2,c^2\}=\int_0^1\min\{Y^*(t)^2,c^2\}\dif t$ and H\"older's inequality with exponents $s/2$ and $s/(s-2)$, we have
\begin{align*}
    \E\min\{Y^2,c^2\}&\le c^2t_0+\int_{t_0}^1\left[t^{1/2}Y^*(t)\right]^2\frac{\dif t}{t}
    \le D+\left(\int_{t_0}^1\left[t^{1/2}Y^*(t)\right]^s\frac{\dif t}{t}\right)^{2/s}\left[\ln\frac1{t_0}\right]^{1-2/s}\\
    &\le D\left\{1+\left[\ln\frac{c^2}{D}\right]^{\delta_s}\right\},
\end{align*}
which is at most the right-hand side of~\eqref{eq:lorentz-trunc}. For $s=2$, we have $\E\min\{Y^2,c^2\}\le\E Y^2\le D$. Finally, for $s=\infty$, the tail is at most $\min\{1,D/u^2\}$, and $2\int_0^cu\,\PP\{Y>u\}\dif u\le D[1+\ln(c^2/D)]$, as in the proof of Lemma~\ref{lem:weak-clip}. This concludes the proof.
\end{proof}

\begin{lemma}[Concavity with an exponent]\label{lem:log-sum-delta}
Let $C\ge0$ and $\delta\in[0,1]$, and let $\phi_\delta(q):=q[1+\ln(1+C/q)]^\delta$ for $q>0$, $\phi_\delta(0):=0$. Then $\phi_\delta$ is non-decreasing and concave on $[0,\infty)$. Consequently, $\sum_{t<T}\phi_\delta(q_t)\le\bar B[1+\ln(1+TC/\bar B)]^\delta$ whenever $q_t\ge0$ and $\sum_{t<T}q_t\le\bar B$, $\bar B>0$.
\end{lemma}

\begin{proof}
If $C=0$, then $\phi_\delta(q)=q$. Next, let $C>0$, $w:=1+\ln(1+C/q)\ge1$, and $\rho:=C/(q+C)\in(0,1)$; then $\frac{\dif w}{\dif q}=-\frac\rho q$ and $\frac{\dif\rho}{\dif q}=-\frac{\rho^2}{C}$. Then $\phi_\delta'(q)=w^{\delta-1}(w-\delta\rho)>0$ and
\[
    \phi_\delta''(q)=-\frac{\delta\rho^2}{q}\,w^{\delta-2}\left(w+1-\delta\right)\le0 .
\]
Since $\phi_\delta(q)\le q(1+\ln(1+C/q))\to0$ as $q\downarrow0$, the function is continuous, non-decreasing, and concave on $[0,\infty)$, and the last claim follows from Jensen's inequality with equal weights and monotonicity, as in Lemma~\ref{lem:log-sum}, which concludes the proof.
\end{proof}

\begin{proposition}[\clipsstm on the Lorentz scale]\label{prop:lorentz-sstm}
Let Assumptions~\ref{as:smooth}--\ref{as:convex} hold for $Q=B_{3R}(x^*)$, where $R\ge\|x^0-x^*\|>0$, and let Assumption~\ref{as:lorentz} hold with some $s\in[2,\infty]$. Then Theorem~\ref{thm:sstm} holds with $\alpha=2$ and the clipping constants
\[
    b:=4,\qquad v:=v_s(\Lambda_K^{\rm S}),\qquad v_s(\Lambda):=2\left[1+\ln(1+\Lambda)\right]^{\delta_s},\qquad \Lambda_K^{\rm S}=\max\{800K,\Lambda_K^0\}.
\]
In particular, with $a=a_*(K,\beta)$, $\mathcal L_K:=1+\ln(1+\Lambda_K^{\rm S})$, and $v=2\mathcal L_K^{\delta_s}$, on an event of probability at least $1-\beta$, we have
\[
    f(y^K)-f^*\le\max\left\{\frac{1058400\,LR^2H^2}{K^2},\ 12\sqrt{2700\,v}\,\sigma_0R\sqrt{\frac HK}\,,\ 6\cdot49\cdot5400\,v\,\sigma_1^2LR^2\frac HK\right\};
\]
in other words, the additive term contains $\mathcal L_K^{\delta_s/2}$ and the multiplicative term contains $\mathcal L_K^{\delta_s}$.
\end{proposition}

\begin{proof}
The proof of Theorem~\ref{thm:sstm} uses the noise assumption only through the conditional clipping bounds~\eqref{eq:sstm-clip-bounds}, i.e., through the following three properties of $(b,v)$: (i)~for $\|x\|\le\lambda/2$ and $X$ with $\E X=x$, $\|\E\clip(X,\lambda)-x\|\le bD/\lambda$; (ii)~$\E\|\clip(X,\lambda)-x\|^2\le\phi_\lambda(D)$ with $\phi_\lambda(D)=D\,v_s((3\lambda/2)^2/D)$, where $v_s$ is non-decreasing and $q\mapsto q\,v_s(C/q)$ is concave and non-decreasing for every $C\ge0$ (this is exactly what Lemma~\ref{lem:log-sum} provides in Step~6, both for the summation and for the monotonicity in $a$, and what makes the right-hand side of~(ii) non-decreasing in $D$); (iii)~the fixed point~\eqref{eq:sstm-fixed-point}. Here $D=\sigma_0^2+\sigma_1^2\|x\|^2$ bounds $\|\,\|X-x\|\,\|_{2,s}^2$, and the conditional versions follow exactly as in Corollary~\ref{cor:conditional} and Lemma~\ref{lem:adaptive}, since the quasi-norm bound is applied pointwise in $x$ via the freezing formula. To prove~(i), we use Lemma~\ref{lem:lorentz-embed}, which gives $\PP\{\|X-x\|>u\}\le2D/u^2$, i.e., the tail condition from Lemma~\ref{lem:weak-clip} with $\alpha=2$ and $2D$ in place of $A(x)$, and, therefore, \eqref{eq:weak-clip-bias} gives the bias bound $b_2\cdot2D/\lambda=4D/\lambda$. To prove~(ii), we use the contraction~\eqref{eq:clip-contract} and Lemma~\ref{lem:lorentz-trunc} with $c=3\lambda/2$, which give $\E\|\clip(X,\lambda)-x\|^2\le\E\min\{\|X-x\|^2,(3\lambda/2)^2\}\le D\,v_s((3\lambda/2)^2/D)$, and the required monotonicity and concavity follow from Lemma~\ref{lem:log-sum-delta} with $\delta=\delta_s$ (the factor $2$ is harmless). To prove~(iii), we repeat Step~3$'$ with $b=4$ and $v\le2(1+\ln(1+\Lambda_K^{\rm S}))$: the second branch is $9bK=36K\le\Lambda_K^{\rm S}$, and the first branch is at most $\frac{8\cdot1350\cdot2}{400}\frac KH\big(1+\ln(1+\Lambda_K^{\rm S})\big)=54\frac KH\big(1+\ln(1+\Lambda_K^{\rm S})\big)$; since $\Lambda_K^{\rm S}\ge800K$ and $\Lambda\mapsto\Lambda/(1+\ln(1+\Lambda))$ is increasing, it is sufficient to show that $54(1+\ln(1+800K))\le800H$, which holds for $K=1$ ($416<1109$, using $H\ge\ln4$) and whose two sides have increasing difference in $K$ (the derivative is $800/K-43200/(1+800K)>0$). The displayed bound is~\eqref{eq:sstm-rate} at $\alpha=2$: since $v\ge2>4/3$, we have $q_2=\min\{\frac1{1350v},\frac1{1800}\}=\frac1{1350v}$; hence, $c_{0,2}=\sqrt{2/q_2}=\sqrt{2700v}$ and $c_{1,2}=49\cdot4/q_2=49\cdot5400v$, which concludes the proof.
\end{proof}

\begin{remark}[Sharpness of the exponent]\label{rem:lorentz-sharp}
For $2<s<\infty$, the exponent $\delta_s$ in~\eqref{eq:lorentz-trunc} cannot be decreased uniformly over the unit ball of $L^{2,s}$ and over $c$: for $T\ge1$, the non-increasing function $Y_T(t):=c_T\min\{t^{-1/2},e^{T/2}\}$ on $(0,1]$ (its own rearrangement up to the point $t=1$) with $c_T:=(T+2/s)^{-1/s}$ satisfies $\|Y_T\|_{2,s}=1$ and $\E Y_T^2=c_T^2(1+T)$, which is the truncated second moment at level $c=c_Te^{T/2}\ge Y_T$; it grows as $T^{1-2/s}$, while $\ln(1+c^2)\sim T$. For $s=\infty$, the full logarithm is necessary, see Remark~\ref{rem:alpha-two}. This sharpness statement concerns the truncated second-moment bound over the Lorentz unit ball. Proposition~\ref{prop:lorentz-opt-lower} below separately establishes an optimization lower bound and proves the necessity of the exponent $\frac12-\frac1s$ in the additive noise term at fixed confidence; it does not establish the optimality of the full logarithmic dependence in the \clipsstm guarantee or of its multiplicative term.
\end{remark}

\begin{remark}[Comparison with the moment guarantee]\label{rem:lorentz-moment}
The orders in $K$ from Proposition~\ref{prop:lorentz-sstm} can also be recovered from Theorem~\ref{thm:sstm} under Assumption~\ref{as:moment}. For $1<p<2\le s<\infty$, writing $\E|Y|^p=\int_0^1[t^{1/2}Y^*(t)]^p\,t^{1-p/2}\frac{\dif t}{t}$ and applying H\"older's inequality with the exponents $s/p$ and $s/(s-p)$, we get
\begin{equation}\label{eq:lorentz-moment-embed}
    \|Y\|_p\le C_{p,s}\|Y\|_{2,s},\qquad C_{p,s}:=\left(\frac{2(s-p)}{s(2-p)}\right)^{\frac1p-\frac1s},
\end{equation}
and, for $s=\infty$, $C_{p,\infty}=(2/(2-p))^{1/p}$ (see Proposition~\ref{prop:embeddings}). As a consequence, Assumption~\ref{as:lorentz} implies Assumption~\ref{as:moment} with the exponent $p$ and the parameters $(C_{p,s}\sigma_0,C_{p,s}\sigma_1)$. {Choose $p_K:=2-\min\{\frac12,\frac{1}{\ln(eK)}\}\in[\frac32,2)$. The constants $c_{0,p},c_{1,p}$ from Theorem~\ref{thm:sstm} under Assumption~\ref{as:moment} are bounded uniformly in $p\in[\frac32,2)$. We next verify the uniform embedding bound $C_{p_K,s}\le4[\ln(eK)]^{1/2-1/s}$.}

{Set $T:=\ln(eK)$ and $p:=p_K$. If $1\le T\le2$, then $p=3/2$, the base in~\eqref{eq:lorentz-moment-embed} is at most $4$, and its exponent is at most $2/3$. Hence $C_{p,s}\le4^{2/3}\le4T^{1/2-1/s}$. If $T\ge2$, the base is at most $2T$ and $1/p-1/s=1/2-1/s+1/(2pT)$. Since $p\ge3/2$ and $\ln(2T)/T$ decreases for $T\ge2$,}
\[
    {C_{p,s}\le\sqrt2\,T^{1/2-1/s}\exp\!\left(\frac{\ln(2T)}{2pT}\right)
    \le\sqrt2\,e^{\ln4/6}\,T^{1/2-1/s}<4T^{1/2-1/s}.}
\]
{The case $s=\infty$ follows by taking the limit. Also, $(2-p_K)\ln K\le1$ and $H>1$ give $K^{-(p_K-1)/p_K}\le e^{1/(2p_K)}K^{-1/2}$ and $H^{(p_K-1)/p_K}\le H^{1/2}$. Applying Theorem~\ref{thm:sstm} with its prescribed parameters therefore gives, with probability at least $1-\beta$,}
\[
    f(y^K)-f^*\lesssim\frac{LR^2H^2}{K^2}+\sigma_0R\sqrt{\frac HK}\,[\ln(eK)]^{\frac12-\frac1s}+\sigma_1^2LR^2\frac HK\,[\ln(eK)]^{1-\frac2s},\qquad H=\ln\frac{4K}{\beta},
\]
{The constant in this bound is absolute. For fixed problem parameters and confidence, this reduction gives the same powers of $K$ and $\ln K$ as Proposition~\ref{prop:lorentz-sstm}. The direct proof provides explicit finite-horizon constants and parameter choices under the original Lorentz condition, with $1+\ln(1+\Lambda_K^{\rm S})$ in place of $\ln(eK)$.}

{Theorem~\ref{thm:maximal} improves the reduction by replacing $H=\ln(4K/\beta)$ with $h=\ln(4/\beta)$, yielding Corollary~\ref{cor:lorentz-upper}. For fixed problem parameters, fixed confidence, and $\sigma_0>0$, the additive terms in the direct and improved bounds have orders $\sigma_0RK^{-1/2}(\ln K)^{1-1/s}$ and $\sigma_0RK^{-1/2}(\ln K)^{1/2-1/s}$, respectively. The improved bound therefore removes a factor $\sqrt{\ln K}$ and matches the additive lower bound asymptotically (see Remark~\ref{rem:lorentz-opt-lower-scope}).}

{This improvement does not follow by substituting $h$ directly into the parameters of Proposition~\ref{prop:lorentz-sstm}. It uses the moment theorem with exponent $p_K$ and noise parameters $(C_{p_K,s}\sigma_0,C_{p_K,s}\sigma_1)$. For $s=2$, the finite-variance theorem with exponent $2$ is also available. The same moment-embedding argument applies to \clipsgd.}
\end{remark}

\paragraph{An optimization lower bound on the Lorentz scale.} Next, we show that the Lorentz logarithm is necessary for optimization itself, in the additive noise class. The oracle model is the same as in Proposition~\ref{prop:lower-rate}: the algorithm makes at most $K$ adaptive, possibly randomized, calls to a stochastic gradient oracle, which returns fresh independent samples at successive calls; query points, the decision to stop, and the output are measurable functions of the known class parameters, the previous responses, and an independent random seed; no exact function values are supplied. For a real random variable $Y$ with law $P$, we write $\|Y\|_{2,s;P}$ for the quasi-norm defined above (the rearrangement is taken under $P$) and omit $P$ when it is clear from the context.

\begin{proposition}[Optimization lower bound under additive $L^{2,s}$ noise]\label{prop:lorentz-opt-lower}
Let $K\ge1$ be an integer, $L,R,\sigma_0>0$, and $s\in[2,\infty]$. For every algorithm in the above oracle model that makes at most $K$ stochastic gradient calls and returns a point $\hat x\in\R$, there exist a convex $L$-smooth quadratic $f$ on $\R$ and an unbiased stochastic gradient oracle $g(x,\xi)$ with $|x^0-x^*|=R$ and $\|g(x,\xi)-\nabla f(x)\|_{2,s}\le\sigma_0$ for every $x\in\R$ (the law of the noise does not depend on $x$, and Assumption~\ref{as:lorentz} holds on $\R$ with $(\sigma_0,0)$) such that
\begin{equation}\label{eq:lorentz-opt-lower}
    \PP\left\{f(\hat x)-f^*\ge\frac1{64}\min\left\{LR^2,\ \frac{\sigma_0R}{\sqrt K}\,[\ln(eK)]^{\frac12-\frac1s}\right\}\right\}\ge\frac13 ,
\end{equation}
where the probability is taken w.r.t.\ the oracle samples and the internal randomness of the algorithm, which may know $L,R,\sigma_0,s$, and $K$. The constants $1/64$ and $1/3$ do not depend on $s$.
\end{proposition}

\begin{proof}
Without loss of generality, we assume that $x^0=0$ (this can be achieved by translation). Let $T:=\ln(eK)\ge1$, $a:=\sigma_0/4$, $c_T:=(T+2/s)^{-1/s}$ for $s<\infty$, and $c_T:=1$ for $s=\infty$. Let $U$ be uniformly distributed on $(0,1)$ and $\epsilon$ be an independent Rademacher sign, and define the symmetric reference law $P_0$ via $Z_0:=a\epsilon c_T\min\{U^{-1/2},e^{T/2}\}$. The rearrangement of $|Z_0|$ is $ac_T\min\{t^{-1/2},e^{T/2}\}$ for $0<t<1$ (its value at the single point $t=1$ is immaterial), and, therefore, for $s<\infty$,
\[
    \|Z_0\|_{2,s}^s=a^sc_T^s\left(e^{sT/2}\int_0^{e^{-T}}t^{\frac s2-1}\dif t+\int_{e^{-T}}^1\frac{\dif t}{t}\right)=a^sc_T^s\left(\tfrac2s+T\right)=a^s,
\]
and, for $s=\infty$, the supremum equals $a$; similarly,
\[
    V:=\E_0Z_0^2=a^2c_T^2(1+T),\qquad M:=\sup|Z_0|=ac_Te^{T/2}.
\]
Let $\Delta:=\min\{LR,\frac18\sqrt{V/K}\}$, $\vartheta:=\Delta/V$, and $\mu:=\Delta/R\le L$. For $v\in\{-1,+1\}$, we define the tilted laws $P_v$ via $\frac{\dif P_v}{\dif P_0}(z):=1-v\vartheta z$. They integrate to one since $\E_0Z_0=0$, and they are positive since $\vartheta M\le\frac{M}{8\sqrt{KV}}=\frac{e^{T/2}}{8\sqrt{K(T+1)}}=\frac{\sqrt e}{8\sqrt{T+1}}<\frac14$. For $Z\sim P_v$, we have $\E_vZ=-v\vartheta V=-v\Delta$.

\emph{Noise class.} For $u\ge0$, we have $P_v\{|Z|>u\}=\E_0[(1-v\vartheta Z_0)\ind_{\{|Z_0|>u\}}]=P_0\{|Z_0|>u\}$, since $z\ind_{\{|z|>u\}}$ is odd and $P_0$ is symmetric; hence, the rearrangement of $|Z|$ under $P_v$ coincides with the one of $|Z_0|$ under $P_0$, and $\|Z\|_{2,s;P_v}=a$. Next, $c_T\le1$ and $(T+1)/K=(\ln K+2)/K\le2$ for $K\ge1$ give $\Delta\le\frac{ac_T}8\sqrt{(T+1)/K}\le\frac{\sqrt2}{8}a<\frac a4$. Since $|Z+v\Delta|\le|Z|+\Delta$, we have $P_v\{|Z+v\Delta|>u+\Delta\}\le P_v\{|Z|>u\}$ for all $u\ge0$, i.e., $(Z+v\Delta)^*(t)\le Z^*(t)+\Delta$ for $0<t<1$. For $s<\infty$, Minkowski's inequality in $L^s((0,1),\dif t/t)$ gives
\[
    \|Z+v\Delta\|_{2,s;P_v}\le\|Z\|_{2,s;P_v}+\Delta\left(\int_0^1t^{\frac s2}\frac{\dif t}{t}\right)^{1/s}=a+\left(\tfrac2s\right)^{1/s}\Delta\le a+\Delta ,
\]
and, for $s=\infty$, the same bound follows by taking suprema. Thus, $\|Z+v\Delta\|_{2,s;P_v}\le a+\Delta<\frac54a=\frac5{16}\sigma_0<\sigma_0$.

\emph{Instances.} Let $f_v(x):=\frac\mu2(x-vR)^2$ and $\nabla f_{v,Z}(x):=\mu x+Z$ with $Z\sim P_v$ (as a stochastic function, $f_{v,Z}(x)=\frac\mu2x^2+Zx+\frac\mu2R^2$, and $\E_vf_{v,Z}=f_v$). Both objectives are convex and $L$-smooth on $\R$, with minimizers $vR$ and $f_v^*=0$; $\E_v[\mu x+Z]=\mu(x-vR)=\nabla f_v(x)$, and the noise is $\nabla f_{v,Z}(x)-\nabla f_v(x)=Z+v\Delta$, whose law does not depend on $x$ and satisfies the required bound by the previous step.

\emph{Indistinguishability.} With $P_\pm:=P_{\pm1}$ and $\chi^2(P\|Q):=\int(\dif P/\dif Q-1)^2\dif Q$, we have
\[
    \chi^2(P_+\|P_-)=\E_0\frac{4\vartheta^2Z_0^2}{1+\vartheta Z_0}\le\frac{16}{3}\,\frac{\Delta^2}{V}\le\frac{1}{12K},
\]
by $1+\vartheta Z_0\ge3/4$ and $\Delta^2\le V/(64K)$. By independence, $1+\chi^2(P_+^{\otimes K}\|P_-^{\otimes K})=(1+\chi^2(P_+\|P_-))^K\le e^{1/12}$, and the Cauchy--Schwarz inequality gives $\TV(P_+^{\otimes K},P_-^{\otimes K})\le\frac12\sqrt{e^{1/12}-1}<\frac14$. The response to a query $x$ is $\mu x+Z$ with $\mu$ common for both hypotheses, meaning that the whole transcript and the output are the same measurable function of $K$ independent observations $Z_1,\dots,Z_K$ (unused ones are supplied to the algorithms with fewer calls) and of the seed under both hypotheses; hence, the output laws also have total variation smaller than $1/4$. Let $A:=\{\hat x>0\}$. Then $\PP_+(A^c)+\PP_-(A)=1-(\PP_+(A)-\PP_-(A))>\frac34$, meaning that one of the two error probabilities exceeds $\frac38>\frac13$. On an error under hypothesis $v$, we have $|\hat x-vR|\ge R$ and $f_v(\hat x)-f_v^*\ge\frac{\mu R^2}{2}=\frac{\Delta R}{2}$. Finally, $c_T\ge(T+1)^{-1/s}$ gives $\sqrt V=ac_T\sqrt{T+1}\ge a(T+1)^{\frac12-\frac1s}\ge aT^{\frac12-\frac1s}$ (also for $s=\infty$), and, therefore,
\[
    \frac{\Delta R}{2}=\min\left\{\frac{LR^2}2,\frac{R\sqrt V}{16\sqrt K}\right\}\ge\frac1{64}\min\left\{LR^2,\frac{\sigma_0R}{\sqrt K}T^{\frac12-\frac1s}\right\}.\qedhere
\]
\end{proof}

\begin{remark}[Scope of the optimization lower bound]\label{rem:lorentz-opt-lower-scope}
For fixed $L,R,\sigma_0>0$ and large $K$, the stochastic term in~\eqref{eq:lorentz-opt-lower} is the smaller one, and, therefore, at any fixed failure probability $\beta<1/3$, no uniform guarantee over the additive $L^{2,s}$ class can improve the order $\sigma_0R[\ln(eK)]^{1/2-1/s}/\sqrt K$: this recovers the finite-variance order at $s=2$, proves the necessity of the Lorentz logarithm for every $s>2$, and gives the factor $\sqrt{\ln(eK)}$ at $s=\infty$. Unlike Remark~\ref{rem:lorentz-sharp}, it is a statement about the optimization error of every algorithm in the stated oracle model. Nevertheless, it does not establish the optimality of the whole \clipsstm guarantee: the additive term from Proposition~\ref{prop:lorentz-sstm} contains $\sqrt H\,\mathcal L_K^{1/2-1/s}$, and, for fixed $L,R,\sigma_0>0$, $\sigma_1=0$, and fixed $\beta$, $\Lambda_K^0=(17640LRH/(\sigma_0(K+1)))^2=O([\ln(eK)]^2/K^2)$, meaning that $\Lambda_K^{\rm S}=800K$ for large $K$ and both $H$ and $\mathcal L_K$ are of the order $\ln(eK)$; therefore, a factor of the order $\sqrt{\ln(eK)}$ remains between the upper and the lower bound, and nothing is claimed for the multiplicative term or for the joint dependence on $K$ and a vanishing $\beta$. The construction is a minimax one: the hard instance, its curvature $\mu$, and its noise law depend on $K$ and $s$; each noise law is bounded with variance $\Var_{P_v}(Z+v\Delta)=V-\Delta^2\in[\frac{63}{1024},\frac18]\cdot\sigma_0^2[\ln(eK)]^{1-2/s}$ (this follows from $a^2T^{1-2/s}\le V\le2a^2T^{1-2/s}$ and $\Delta^2\le V/(64K)$), which is not bounded uniformly in $K$ for $s>2$. That is, no hardness is claimed for one fixed infinite-variance distribution, for a class with a prescribed strong convexity constant, or for centrally symmetric noise (the centered tilted laws are asymmetric).
\end{remark}

\subsection{Removing the Horizon Logarithm}\label{app:maximal}

The factors $\ell=\ln\frac{4(K+1)}{\beta}$ and $H=\ln\frac{4K}{\beta}$ appear since the proofs of Proposition~\ref{prop:master} and Theorem~\ref{thm:sstm} distribute the failure probability over the $K+1$ (resp.\ $K$) steps of the induction. This is an artifact of the union bound, not of the heavy tails; under the moment assumption, its removal is studied by \citet{nguyen2023improved}. Let
\begin{equation}\label{eq:h}
    h:=\ln\frac4\beta>1 .
\end{equation}

\begin{theorem}[Horizon-free confidence]\label{thm:maximal}
\begin{enumerate}
    \item Proposition~\ref{prop:master} remains valid when $\ell$ is replaced by $h$ in~\eqref{eq:master-conds}, under either noise assumption and for all $\alpha\in(1,2]$. Consequently, Theorem~\ref{thm:fixed} and Corollary~\ref{cor:fixed} hold with $\ell$ replaced by $h$ throughout.
    \item Theorem~\ref{thm:main} and Corollary~\ref{cor:rate} hold with $\ell$ replaced by $h$ throughout (in $\gamma_L,\gamma_0,\gamma_1,\gamma_2$, $\lambda$, $\Delta_L,\Delta_0,\Delta_1$), under Assumption~\ref{as:moment} with $\alpha\in(1,2]$, and under Assumption~\ref{as:weak} with $\alpha<2$ if $c$ is computed with $v=v_\alpha(\infty)$.
    \item Theorem~\ref{thm:sstm} and Corollary~\ref{cor:sstm-complexity} hold with $H$ replaced by $h$ throughout (in $\lambda_k$, $a_*(K,\beta)$ and~\eqref{eq:sstm-rate}), under Assumption~\ref{as:moment} with $\alpha\in(1,2]$, and under Assumption~\ref{as:weak} with $\alpha<2$ if $q_\alpha$ is computed with $v=v_\alpha(\infty)$.
\end{enumerate}
{In the cases of items~2 and~3, the corresponding bound below holds for each method separately with probability at least $1-\beta$:}
\begin{equation}\label{eq:improved-rates}
\begin{aligned}
    f(\bar x^K)-f^*&\le C_\alpha\left[\frac{LR^2h}{K+1}+\sigma_0R\Big(\frac h{K+1}\Big)^r+\sigma_1^2LR^2\Big(\frac h{K+1}\Big)^{2r}\right],\\
    f(y^K)-f^*&\le C_\alpha\left[\frac{LR^2h^2}{K^2}+\sigma_0R\Big(\frac hK\Big)^r+\sigma_1^2LR^2\Big(\frac hK\Big)^{2r}\right],
\end{aligned}
\end{equation}
with $r=\frac{\alpha-1}{\alpha}$ and constants $C_\alpha$ depending only on $\alpha$; under Assumption~\ref{as:moment}, $C_\alpha$ is bounded uniformly in $\alpha\in[3/2,2]$.
\end{theorem}

\begin{proof}
\emph{Item 1.} We modify the proof of Proposition~\ref{prop:master}, with $h$ in place of $\ell$ in~\eqref{eq:master-conds} and in all quantities derived from it ($c_1=R^2/(5h)$, $G=R^4/(150h)$). The consequences~\eqref{eq:gamma-L} use only $h>1$. The events $E_k$ from~\eqref{eq:events} are non-increasing in $k$, $E_k\in\G_k$, and all deterministic bounds from the proof (localization, \eqref{eq:grad-small}, \eqref{eq:clip-bounds}, \eqref{eq:Sigma}, \eqref{eq:var-sum}, \eqref{eq:master}, \eqref{eq:master2}, the bounds on $\circledTwo,\circledFour,\circledFive$, and the variance bounds~\eqref{eq:var-T1} and its analogue for $\circledThree$) hold on $E_{T-1}$ for every $T\le K+1$ exactly as before, since they use $\ell$ only through~\eqref{eq:master-conds}. Instead of the induction, we define
\[
    \tau:=\min\{T\in\{1,\dots,K+1\}:\ E_T\text{ fails}\}\qquad(\tau:=K+2\text{ if }E_{K+1}\text{ holds}),
\]
meaning that $\{t<\tau\}=E_t\in\G_t$ for $t\le K$, and consider the stopped increments $X_t\ind_{E_t}$ and $Y_t\ind_{E_t}$, $t=0,\dots,K$, where $X_t,Y_t$ are the summands of $\circledOne,\circledThree$. Since $\ind_{E_t}$ is $\G_t$-measurable, these are martingale difference sequences w.r.t.\ $(\G_{t+1})$, bounded by $c_1$, with conditional variances $\sigma_t^2\ind_{E_t}$ and $\widetilde\sigma_t^2\ind_{E_t}$. Their total conditional variances are at most $G$ almost surely: on $\{\tau=T\}$ with $T\le K+1$, the event $E_{T-1}$ holds and the sums reduce to $\sum_{t<T}\sigma_t^2$ and $\sum_{t<T}\widetilde\sigma_t^2$, which are at most $G$ by~\eqref{eq:var-T1} and its analogue; on $\{\tau=K+2\}$, the same holds with $T=K+1$. Therefore, the maximal form of Lemma~\ref{lem:freedman} with $b_1=R^2/5$, $c_1$, and $G$ gives, for the two channels,
\[
    \PP\Big\{\exists T\le K+1:\ \Big|\sum_{t<T}X_t\ind_{E_t}\Big|>\frac{R^2}5\Big\}\le2e^{-h}=\frac\beta2,\qquad
    \PP\Big\{\exists T\le K+1:\ \Big|\sum_{t<T}Y_t\ind_{E_t}\Big|>\frac{R^2}5\Big\}\le\frac\beta2 .
\]
On the complement $E^{\max}$ of the union of these two events, $\PP(E^{\max})\ge1-\beta$, we assume that $\tau=T\le K+1$. Then $E_{T-1}$ holds, implying $\ind_{E_t}=1$ for $t\le T-1$ and the stopped sums up to $T-1$ coincide with $\circledOne$ and $\circledThree$ at time $T$; hence, $|\circledOne|,|\circledThree|\le R^2/5$, and $|\circledTwo|,\circledFour,\circledFive\le R^2/5$ by the deterministic bounds on $E_{T-1}$. As in the paragraph ``Completing the induction'', this gives the condition from~\eqref{eq:events} at $t=T$ and $R_T\le\sqrt2R$, meaning that $E_T$ holds, which contradicts $\tau=T$. Thus, $E_{K+1}\supseteq E^{\max}$, and the conclusion of Proposition~\ref{prop:master} follows as before. The proof of Theorem~\ref{thm:fixed} applies Proposition~\ref{prop:master} with $v=v_\alpha(\Lambda_\lambda)$, where $\Lambda_\lambda$ from~\eqref{eq:Lambda} does not depend on $\ell$, and the proof of Corollary~\ref{cor:fixed} is algebraic in $\ell$; both of them remain valid with $h$.

\emph{Item 2.} The proof of Theorem~\ref{thm:main} verifies~\eqref{eq:master-conds} from $\gamma\le\min\{\gamma_L,\gamma_0,\gamma_1,\gamma_2\}$ via algebraic manipulations in which $\ell$ is a parameter (Step~1 in Appendix~\ref{app:main}); only the fixed-point check from Step~2, which bounds $\Lambda_{\gamma_*,\lambda_*}$ by $\Lambda_K$ under Assumption~\ref{as:weak}, uses $\ell\ge\ln(4(K+1))$. Under Assumption~\ref{as:moment}, the constant is $c=c_\alpha^{\rm mom}$ and no fixed point is needed; under Assumption~\ref{as:weak} with $\alpha<2$ and $v=v_\alpha(\infty)$, \eqref{eq:var-sum} holds with this $v$ directly, since $v_\alpha(\cdot)\le v_\alpha(\infty)$, and, again, no fixed point is needed. Then Item~1 gives the claim, and Corollary~\ref{cor:rate} follows from algebraic manipulations.

\emph{Item 3.} In the proof of Theorem~\ref{thm:sstm} (see Appendix~\ref{app:sstm-proof}), $H$ enters through the parametrization $\lambda_t=R/(30\omega_{t+1}H)$, $a\ge176400H^2$, the Freedman parameters $c_1=R^2/(5H)$, $G=R^4/(150H)$, the inequalities $H>1$ in Steps~1, 8, and~10, and the fixed-point check from Step~3$'$, which is needed only under Assumption~\ref{as:weak} with $v=v_\alpha(\Lambda_K^{\rm S})$; in the two cases from the statement, $\phi_{\lambda_t}(D_t)\le vD_t$ holds directly with the horizon-free $v$ (Lemma~\ref{lem:moment-clip}, resp.\ $v_\alpha(\cdot)\le v_\alpha(\infty)$), and Step~6 reads $S_T\le v30^{\alpha-2}\frac{R^2}H\Psi_T\le\frac{R^2}{2700H}$ by~\eqref{eq:Psi-final} and the definition of $q_\alpha$. With $h$ in place of $H$, all deterministic steps hold on $E_{T-1}$ for every $T\le K$ as before. Let $\tau$ be the first $T\le K$ at which $E_T$ from~\eqref{eq:sstm-events} fails ($\tau:=K+1$ if $E_K$ holds); $\{t<\tau\}=E_t\in\G_t$, and the stopped increments $X_t\ind_{E_t}$, $Y_t\ind_{E_t}$ of $\circledOne,\circledThree$ are martingale differences bounded by $c_1=R^2/(5h)$ with total conditional variances at most $G=R^4/(150h)$ almost surely (Steps~6 and~8 on $E_{\tau-1}$). The maximal form of Lemma~\ref{lem:freedman} bounds the probability that some partial sum of either channel exceeds $R^2/5$ by $\beta/2$ for each of them. On the complement, if $\tau=T\le K$, then $E_{T-1}$ holds, the stopped sums coincide with $\circledOne,\circledThree$ at time $T$, Steps~5--9 give $B_T<R^2$, and Step~1 and Lemma~\ref{lem:potential} give $\|z^T-x^*\|<2R$, meaning that $E_T$ holds, which is a contradiction. Thus, $E_K$ holds with probability at least $1-\beta$, and the final bound follows as in Step~10; Corollary~\ref{cor:sstm-complexity} follows from algebraic manipulations. The rates~\eqref{eq:improved-rates} are Corollary~\ref{cor:rate} and~\eqref{eq:sstm-rate} with $h$; under Assumption~\ref{as:moment}, $b_\alpha^{\rm mom}\in[1/2,1]$ and $v_\alpha^{\rm mom}\in[1,3/2]$, meaning that $c_\alpha^{\rm mom}$, $q_\alpha$, $c_{0,\alpha}$, $c_{1,\alpha}$, and, as a consequence, $C_\alpha$ are bounded uniformly on $[3/2,2]$, which concludes the proof.
\end{proof}

Note that Theorem~\ref{thm:maximal} controls the prefixes of \emph{one} run tuned for the horizon $K$; it is not an anytime statement for all horizons of one infinite trajectory. {For the original tuning with $v_\alpha(\Lambda_K)$ and $v_\alpha(\Lambda_K^{\rm S})$, the fixed-point estimates in Steps~2 and~3$'$ must be checked again after replacing $\ell$ or $H$ by $h$. At $\alpha=2$ and fixed confidence, the logarithmic terms in these checks cannot in general be absorbed into $h$ for large $K$. This is not a general obstruction for $\alpha<2$, where $v_\alpha$ is bounded. Theorem~\ref{thm:maximal} uses $v_\alpha(\infty)$ in that case, and the following corollary treats the endpoint through the moment embedding.}

\begin{corollary}[The Lorentz endpoint without the union-bound loss]\label{cor:lorentz-upper}
Let Assumptions~\ref{as:smooth}--\ref{as:convex} hold for $Q=B_{3R}(x^*)$, where $R\ge\|x^0-x^*\|>0$, let Assumption~\ref{as:lorentz} hold with $s\in[2,\infty]$, let $K\ge1$, $\beta\in(0,1)$, and $h=\ln(4/\beta)$. Let $p_K:=2-\min\{\frac12,\frac1{\ln(eK)}\}$ and $C_{p,s}$ be defined as in Remark~\ref{rem:lorentz-moment}. Then Algorithm~\ref{alg:clip-sstm} with the parameters from Theorem~\ref{thm:sstm} under Assumption~\ref{as:moment}, computed with the exponent $p_K$, the parameters $(C_{p_K,s}\sigma_0,C_{p_K,s}\sigma_1)$, and $h$ in place of $H$, satisfies, with probability at least $1-\beta$,
\begin{equation}\label{eq:lorentz-upper}
    f(y^K)-f^*\le C\left[\frac{LR^2h^2}{K^2}+\sigma_0R\sqrt{\frac hK}\,[\ln(eK)]^{\frac12-\frac1s}+\sigma_1^2LR^2\,\frac hK\,[\ln(eK)]^{1-\frac2s}\right]
\end{equation}
with an absolute constant $C$. Under the same assumptions for $Q=B_{2R}(x^*)$, Algorithm~\ref{alg:clip-sgd} with the parameters from Theorem~\ref{thm:main} (exponent $p_K$, the same noise parameters, $h$ in place of $\ell$, $K+1$ steps) satisfies the same bound for $f(\bar x^K)-f^*$ with the deterministic term $LR^2h/(K+1)$.
\end{corollary}

\begin{proof}
{We use the prescribed exponent $p_K$ for every $s\in[2,\infty]$. In particular, $C_{p,2}=1$, so the moment embedding also covers $s=2$.} By~\eqref{eq:lorentz-moment-embed} and the subadditivity of $t\mapsto t^{p/2}$, Assumption~\ref{as:lorentz} implies, for $1<p<2$ and $x\in Q$, $\E\|g(x,\xi)-\nabla f(x)\|^p\le C_{p,s}^p(\sigma_0^2+\sigma_1^2\|\nabla f(x)\|^2)^{p/2}\le C_{p,s}^p(\sigma_0^p+\sigma_1^p\|\nabla f(x)\|^p)$, i.e., Assumption~\ref{as:moment} with the exponent $p$ and the parameters $C_{p,s}(\sigma_0,\sigma_1)$. With $p=p_K\in[\frac32,2)$, we have $C_{p_K,s}\le4[\ln(eK)]^{\frac12-\frac1s}$ (see Remark~\ref{rem:lorentz-moment}), $K^{-(p_K-1)/p_K}\le e^{1/(2p_K)}K^{-1/2}$ (since $\frac12-\frac{p-1}{p}=\frac{2-p}{2p}$ and $(2-p_K)\ln K\le1$), and $h^{(p_K-1)/p_K}\le h^{1/2}$ (since $h>1$). Substituting these inequalities into~\eqref{eq:improved-rates} with the exponent $p_K$ and using the uniform boundedness of $C_{p}$ on $[\frac32,2]$, we get~\eqref{eq:lorentz-upper}; for \clipsgd, the same substitution in the first bound of~\eqref{eq:improved-rates} applies, which concludes the proof.
\end{proof}

{For fixed $L,R,\sigma_0>0$, $\sigma_1=0$, $s\in[2,\infty]$, and $\beta<1/3$, the stochastic term in~\eqref{eq:lorentz-upper} matches the lower bound~\eqref{eq:lorentz-opt-lower} as $K\to\infty$. Thus, $\sigma_0RK^{-1/2}[\ln(eK)]^{1/2-1/s}$ is the asymptotic minimax order over the additive $L^{2,s}$ class at fixed confidence, and both \clipsgd and \clipsstm attain this order. The asymptotic qualification accounts for the cap $LR^2$ in the lower bound and the deterministic terms in the upper bounds.} That is, not all logarithmic factors from the earlier bounds are statistically necessary: the logarithmic dependence on the horizon in $\ell$ and $H$ comes from the union bound, while the Lorentz factor has a lower bound of its own. For $1<\alpha<2$, the dependence on $\beta$ in~\eqref{eq:improved-rates} is optimal as well (see Proposition~\ref{prop:confidence-lower}).

\section{Additional Experiments}\label{app:experiments}

In this section, we present additional experiments complementing Section~\ref{sec:experiments}: an experiment with symmetric noise on a quadratic problem, a ridge regression problem on real data with the asymmetric noise from Section~\ref{sec:experiments}, and the experiments with the full theoretical parameter prescriptions from Theorems~\ref{thm:main} and~\ref{thm:sstm} on a smooth convex benchmark with vanishing curvature. We start with the description of the full experimental protocol.

\paragraph{Protocol.} We use $\alpha=3/2$ and $1024$ independent runs per configuration.
The reported horizons reach $10^6$ oracle calls.
For \clipsgd, we report the mean of $f(\bar x^{n-1})-f^*$, where
$\bar x^{n-1}=n^{-1}\sum_{k=0}^{n-1}x^k$ is the averaged iterate after $n$ oracle calls (one call per iteration); for \clipsstm, we report the mean of $f(y^n)-f^*$, using the output prescribed by Algorithm~\ref{alg:clip-sstm}, without additional averaging.
The bands are pointwise approximate $95\%$ confidence intervals for the mean.
Horizon-dependent configurations restart from the same initial point
for each terminal horizon; the theoretical failure parameter is $\beta=0.05$.

Let $P$ satisfy $\PP\{P>t\}=t^{-\alpha}$ for $t\ge1$; then
$\E P=3$ and $\E P^\alpha=\infty$.
We use either $Z=\epsilon P$, with an independent Rademacher sign $\epsilon$,
or the centered asymmetric $Z=(P-3)/2$.
Both of them satisfy $\E Z=0$ and
$\sup_{u>0}u\,\PP\{|Z|>u\}^{1/\alpha}\le1$.
In the scalar experiments,
$
 g(x)=x+S(x)Z,\;
 S(x)=\bigl(\sigma_0^\alpha+\sigma_1^\alpha|x|^\alpha\bigr)^{1/\alpha}.
$
The vector oracles are specified in each experiment. All of them satisfy the affine weak-tail
inequality with the stated $(\sigma_0,\sigma_1)$; whenever the noise scale
is non-zero, its $\alpha$-th moment is infinite.

\paragraph{Small additive noise under symmetric noise.} We minimize $f(x)=x^2/2$ on $\R$ ($L=1$, $x^*=0$) starting from $x^0=1$ with the symmetric noise, $\sigma_1=0.03$, the constant stepsize $\gamma=0.05$, and the constant clipping level $\lambda=1$, and vary $\sigma_0\in\{0,0.01,0.1,0.5,1\}$ (see Figure~\ref{fig:exp-small-noise}).

{For $\sigma_0=0$, the noise scale vanishes at $x^*$, and the observed error closely tracks the noiseless reference $(\gamma n)^{-2}/2$ over the plotted range. For $\sigma_0>0$, the curves leave this reference earlier as $\sigma_0$ increases and then decay approximately as $1/n$, at levels increasing with $\sigma_0$. The exact noiseless averaged error, when clipping is inactive, is $(1-(1-\gamma)^n)^2/(2\gamma^2n^2)$; the plotted reference is its large-$n$ limit.}

{Two features of this example go beyond the tail assumption. For symmetric $Z$ and even $S$, the mean clipped field $m_\lambda(x):=\E\clip(x+S(x)Z,\lambda)$ is odd, and, therefore, $m_\lambda(0)=0$; its zero is unique here, although, in general, $m_\lambda(x)\ne\nabla f(x)$ for $x\ne0$: since $\sigma_0\le1$ and $\sigma_1=0.03$, we have $S(x)\le\sigma_0+\sigma_1|x|<1+|x|$ for $x\ne0$; for $x>0$ and $a\ge0$, the sum $\clip(x+a,1)+\clip(x-a,1)$ is non-negative and positive when $a<x+1$, which is provided by the event $\{S(x)P<x+1\}$ of positive probability (the case $S(x)=0$ is immediate); hence, $m_1(x)>0$ for $x>0$ and, by oddness, $m_1(x)<0$ for $x<0$. Thus, clipping does not displace the stationary point in this experiment, and the stepsize is constant rather than tied to the horizon as in Corollary~\ref{cor:fixed}. That is, the experiment illustrates the dependence on $\sigma_0$, through which the additive noise enters both the neighborhood $\mathcal N_\lambda$ and the vanishing term $\mathcal V_{\lambda,K}$ from Corollary~\ref{cor:fixed}, and not the convergence to $x^*$ with a fixed stepsize for arbitrary weak-$L^\alpha$ noise, which fails in general (see Theorem~\ref{thm:lower}).}

\begin{figure}[htbp]
    \centering
    \includegraphics[width=0.6\linewidth]{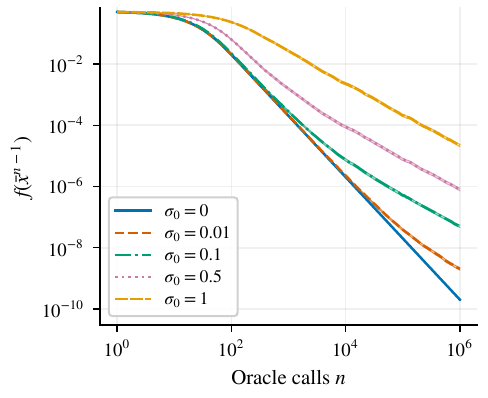}
    \caption{Mean error of the averaged iterate on $f(x)=x^2/2$ under symmetric Pareto noise}
    \label{fig:exp-small-noise}
\end{figure}

\paragraph{Details of the clipping-level experiment.} In the experiment from Figure~\ref{fig:exp-clipping-bias} (see Section~\ref{sec:experiments}), $\gamma_*(n)=\min\{\gamma_L,\gamma_0,\gamma_1,\gamma_2\}$ and $\ell_n=\ln(4n/\beta)$ are computed for the horizon $n$ with $\beta=0.05$, which gives $\lambda(10^2)=2.8$, $\lambda(10^4)=32$, and $\lambda(10^6)\approx570$; each horizon $n$ is a separate run with $\lambda(n)$ being constant within the run. The dotted lines are the values $f(x_\lambda)$ at the zero $x_\lambda$ of the mean clipped field $x\mapsto\E\clip(g(x),\lambda)$, computed numerically; $x_\lambda>0$, since clipping cuts the heavy right tail of $Z$. They decrease with $\lambda$ roughly as $\lambda^{-2(\alpha-1)}$, since, on a quadratic, the error at the biased fixed point is the squared bias divided by $2L$, and the bias scales as $\sigma_0^\alpha/\lambda^{\alpha-1}$ (see Lemma~\ref{lem:clip}). The zero $x_\lambda$ of the mean clipped field is a reference for the observed plateau: with a non-zero constant stepsize, the stochastic iterates may fluctuate around this point, and the limiting error of their average is not necessarily equal to $f(x_\lambda)-f^*$ exactly.

\paragraph{Regression on real data.} {We minimize $f(w)=\frac1{2m}\|Aw-y\|^2+\frac{0.01}{2}\|w\|^2$ on the Diabetes dataset \citep{efron2004least}, with standardized features and targets ($m=442$, $d=10$; loaded with scikit-learn, \citealp{pedregosa2011scikit}), starting from $w^0=0$. Here $L=\lambda_{\max}(\nabla^2f)\approx4.0$ and $\lambda_{\min}(\nabla^2f)\approx0.019$. We use the same clipping levels as above, with $R=\|w^0-w^*\|$ and this $L$ in $\lambda(n)$, and the common stepsize $\gamma=0.05/L$ (Figure~\ref{fig:exp-regression}).}

{The oracle is $g(w,Z)=\nabla f(w)+S(w)Ze$, where $Z$ is the centered asymmetric Pareto variable, $(\sigma_0,\sigma_1)=(0.05,0.001)$, and $e$ is a fixed unit eigenvector of the smallest Hessian eigenvalue. The projected scale $S(w):=(\sigma_0^\alpha+\sigma_1^\alpha|\abr{\nabla f(w),e}|^\alpha)^{1/\alpha}$ satisfies Assumption~\ref{as:weak}, since $|\abr{\nabla f(w),e}|\le\|\nabla f(w)\|$. This oracle is also the gradient of $f(w)+Z\Psi(\abr{w-w^*,e})$, with $\Psi'(u)=(\sigma_0^\alpha+\sigma_1^\alpha\lambda_{\min}^\alpha|u|^\alpha)^{1/\alpha}$, although our model does not require a sample-objective representation.}

{For small clipping levels, the error first decreases and then rises toward a plateau. This behavior is consistent with a displacement of the equilibrium of the mean clipped field. If $\nabla f(w_\lambda)=be$, the quadratic structure gives $w_\lambda-w^*=(b/\lambda_{\min})e$ and $f(w_\lambda)-f^*=b^2/(2\lambda_{\min})$, illustrating how low curvature can amplify clipping bias. Since the experiment uses a positive constant stepsize, this equilibrium value is a reference for the observed plateau, not an exact formula for the limiting error of the averaged iterates. Larger levels lower the observed plateau, and $\lambda(n)$ yields decreasing error over the tested horizons.}

\begin{figure}[htbp]
    \centering
    \includegraphics[width=0.6\linewidth]{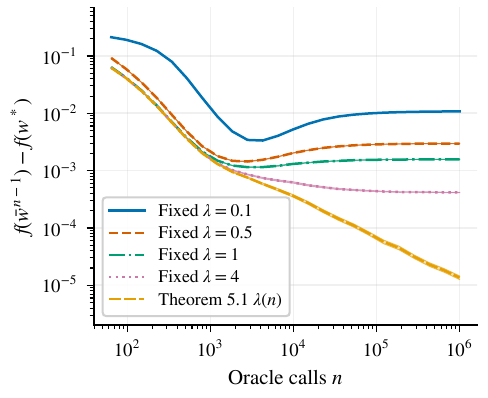}
    \caption{Mean objective gap of the averaged iterate for ridge regression on the Diabetes dataset with asymmetric Pareto noise along the eigenvector of the smallest Hessian eigenvalue}
    \label{fig:exp-regression}
\end{figure}

\paragraph{Relation to the theory.} The clipping level $\lambda(n)$ is the one prescribed by Theorem~\ref{thm:sgd} for the horizon $n$, but the stepsize used is the common empirical value rather than the much smaller $\gamma_*(n)$, and the constant-level runs use a fixed stepsize instead of the horizon-dependent one from Theorem~\ref{thm:fixed-main}; therefore, the comparisons isolate the effect of the clipping level and do not test the high-probability bounds themselves.

{For the asymmetric-noise examples, larger clipping levels lower the observed plateaus, while the horizon-dependent level yields decreasing error over the tested horizons. The symmetric-noise example illustrates a different case: clipping preserves the zero of the mean field, and the averaged error continues to decrease with a constant level. These observations are consistent with the mechanisms discussed in Sections~\ref{sec:results} and~\ref{sec:planned}.}

{The pointwise $95\%$ intervals quantify uncertainty in the estimated mean across runs; they are not $95\%$ guarantees for the error of an individual run. All experiments use $\alpha=3/2$, so they do not directly test the Lorentz endpoint or compare its two theoretical tunings.}

\paragraph{A convex benchmark with vanishing curvature.}
In the remaining experiments, we use a chain-inspired objective in $d=16$ dimensions.
Let $Dw=(w_1-w_2,\ldots,w_{d-1}-w_d)$ and
$M=\lambda_{\max}(D^\top D)$. We choose a zero-mean linear profile
$w_i^*=c((d+1)/2-i)$, $1\le i\le d$, with $\|w^*\|=R$, and let
$\tau=\|Dw^*\|/\sqrt{d-1}$. We set $L=R=1$, $w^0=0$, and
\begin{align*}
 f(w)&=\frac{8L\tau^2}{9M}\sum_{i=1}^{d-1}
       \psi\!\left(\frac{[D(w-w^*)]_i}{\tau}\right),\\
 \psi(t)&=\frac{t^2-\ln(1+t^2)}2.
\end{align*}
Since $0\le\psi''(t)\le9/8$, this objective is globally convex and
$L$-smooth. Moreover, $\nabla^2f(w^*)=0$ and
$\psi(t)=t^4/4+O(t^6)$ near zero; therefore, it is not strongly convex even on the
active zero-mean subspace.
We use either additive noise
$g(w)=\nabla f(w)+\sigma_0Zv$, where
$v=\nabla f(w^0)/\|\nabla f(w^0)\|$ is fixed, or multiplicative noise
$g(w)=(1+\sigma_1Z)\nabla f(w)$, always with $Z=(P-3)/2$.
These are the gradients of $f_Z=f+\sigma_0Zv^\top w$ and
$f_Z=(1+\sigma_1Z)f$, respectively, and they satisfy
Assumption~\ref{as:weak} with the stated coefficients.

\paragraph{Three noise regimes for \clipsgd.}
We compare $(\sigma_0,\sigma_1)=(0,0),(0,0.004),(0.002,0)$ on this objective,
using the full pair $\gamma_*(n)$ and
$\lambda(n)=R/(40\gamma_*(n)\ln(4n/\beta))$ from
Theorem~\ref{thm:main}, recomputed for each configuration and horizon
(see Figure~\ref{fig:exp-sgd-regimes}).
The mean averaged-iterate gaps are close at short horizons and separate
as the horizon increases: the noiseless case decreases most rapidly,
followed by the multiplicative and then the additive noise.
Since the prescribed stepsizes depend on the noise coefficients,
this separation reflects both the stochastic perturbations and the
corresponding conservative tuning.
The comparison illustrates the three noise regimes from
Corollary~\ref{cor:rate}, whose upper bounds have the orders
$n^{-1}$, $n^{-2/3}$, and $n^{-1/3}$ up to logarithmic factors.
These orders describe theoretical upper bounds and are not claimed
to be fitted empirical slopes for this instance.

\begin{figure}[htbp]
    \centering
    \includegraphics[width=0.6\linewidth]{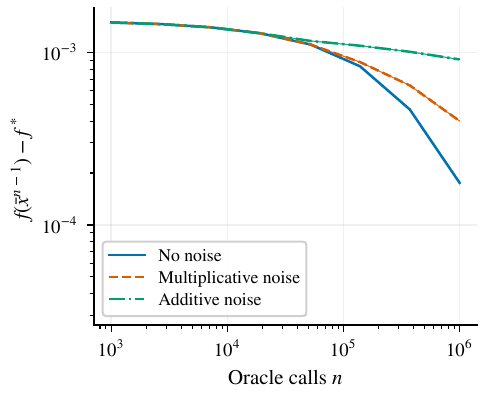}
    \caption{Mean averaged-iterate gap for \clipsgd on the smooth convex chain
    with vanishing curvature, under the full horizon-dependent prescription
    from Theorem~\ref{thm:main}}
    \label{fig:exp-sgd-regimes}
\end{figure}

\paragraph{Sensitivity of \clipsstm to the noise.}
On the same objective, we use the full prescription from
Theorem~\ref{thm:sstm}: $a=a_*(n,\beta)$ and
$\lambda_k=R/(30\omega_{k+1}H_n)$, where
$\omega_{k+1}=(k+2)/(2aL)$ and $H_n=\ln(4n/\beta)$.
We vary $\sigma_0\in\{0,10^{-5},10^{-4},10^{-3}\}$ with $\sigma_1=0$
(see Figure~\ref{fig:exp-sstm-additive}), and, separately,
$\sigma_1\in\{0,10^{-4},3\cdot10^{-4},10^{-3}\}$ with $\sigma_0=0$
(see Figure~\ref{fig:exp-sstm-multiplicative}).
The two plots share the same noiseless control and cover
$10^3\le n\le10^6$ on logarithmic axes.
The parameter $a_*(n,\beta)$ and the associated clipping schedule are
recomputed for each noise level and terminal horizon.
At the smallest noise levels, the curves closely follow the noiseless
control, while larger coefficients lead to markedly slower observed progress.
This separation reflects both the stochastic perturbations and the more
conservative theoretical tuning required by larger noise coefficients.

\begin{figure}[htbp]
    \centering
    \begin{minipage}[t]{0.48\linewidth}
        \centering
        \includegraphics[width=\linewidth]{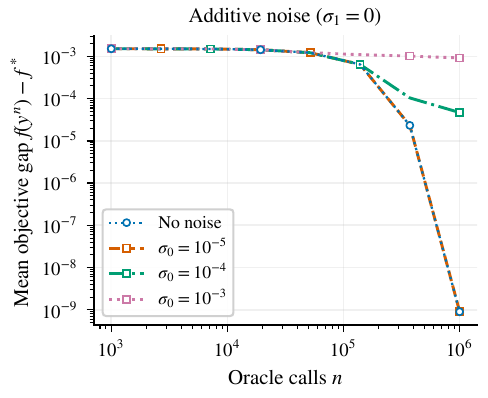}
        \caption{Mean output gap $f(y^n)-f^*$ of \clipsstm on the smooth convex chain
        under the full prescription from Theorem~\ref{thm:sstm}: additive noise ($\sigma_1=0$).}
        \label{fig:exp-sstm-additive}
    \end{minipage}\hfill
    \begin{minipage}[t]{0.48\linewidth}
        \centering
        \includegraphics[width=\linewidth]{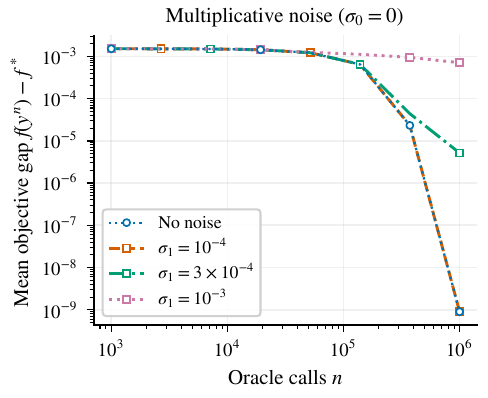}
        \caption{Mean output gap $f(y^n)-f^*$ of \clipsstm on the smooth convex chain
        under the full prescription from Theorem~\ref{thm:sstm}: multiplicative noise ($\sigma_0=0$).}
        \label{fig:exp-sstm-multiplicative}
    \end{minipage}
\end{figure}

\paragraph{Software and data.} {We use the diabetes dataset from \citet{efron2004least}, loaded through scikit-learn \citep{pedregosa2011scikit}. The scikit-learn software is distributed under the BSD-3-Clause license. The synthetic problems are fully specified by the formulas above.}

\end{document}